\DeclareSymbolFont{AMSb}{U}{msb}{m}{n}
\documentclass[reqno,centertags,11pt,a4paper,noamsfonts]{amsart}
\pdfoutput=1
\usepackage[svgnames]{xcolor}
\usepackage{graphicx}
\usepackage[english]{babel}
\usepackage[babel=true,expansion=alltext,protrusion=alltext-nott,final]{microtype}
\usepackage[margin={3cm},marginparwidth={2.5cm}]{geometry}
\usepackage[utf8]{inputenc}
\usepackage{textcomp}
\usepackage[sb]{libertine}
\usepackage[libertine,bigdelims,vvarbb]{newtxmath}
\useosf
\usepackage[varqu,varl]{inconsolata}
\usepackage{mathtools}
\usepackage[T1]{fontenc}
\usepackage{textcomp}
\usepackage{amsthm}
\usepackage[inline]{enumitem}
\numberwithin{equation}{section}
\usepackage[normalem]{ulem}
\usepackage{tikz}
\usetikzlibrary{arrows}
\usepackage{pgfplots}
\pgfplotsset{compat=1.15}
\usepackage{bm}

\newcommand{\RR}{\mathbb{R}}

\newcommand{\NN}{\mathbb{N}}

\newcommand{\1}{\mathbf{1}}

\newcommand{\cor}{\mathrm{cor}}

\newcommand{\E}{\mathbf{E}}
\newcommand{\m}{\mathbf{m}}
\newcommand{\M}{\mathbf{M}}
\newcommand{\h}{\mathbf{h}}

\DeclareMathOperator{\weaklystar}{\rightharpoonup\kern-2.2ex ^* \, \,}
\DeclareMathOperator{\sgn}{sign}
\DeclareMathOperator{\argmin}{argmin}
\DeclareMathOperator{\rel}{rel}

\newcommand{\ltheta}{\ell_{\theta}}
\newcommand{\ttheta}{t_{\theta}}
\newcommand{\ttau}{t_{\tau}}

\def\Xint#1{\mathchoice
{\XXint\displaystyle\textstyle{#1}}%
{\XXint\textstyle\scriptstyle{#1}}%
{\XXint\scriptstyle\scriptscriptstyle{#1}}%
{\XXint\scriptscriptstyle\scriptscriptstyle{#1}}%
\!\int}
\def\XXint#1#2#3{{\setbox0=\hbox{$#1{#2#3}{\int}$ }
\vcenter{\hbox{$#2#3$ }}\kern-.6\wd0}}
\def\dashint{\;\,  \Xint-}

\DeclareMathOperator{\diver}{div}
\DeclareMathOperator{\curl}{curl}

\theoremstyle{plain}
\newtheorem{theorem}{Theorem}[section]
\newtheorem{proposition}[theorem]{Proposition}
\newtheorem{corollary}[theorem]{Corollary}
\newtheorem{lemma}[theorem]{Lemma}

\newtheorem*{theorem*}{Theorem}

\theoremstyle{definition}

\newtheorem{remark}[theorem]{Remark}

\usepackage{hyperref}
\hypersetup{
	linktoc=all,
	bookmarksdepth=3,
	colorlinks=true,
	linkcolor=Green,citecolor=Orange,urlcolor=DarkBlue,
	pdftitle={Domain branching in micromagnetism: scaling law for the global and local energies},
	pdfauthor={Tobias Ried, Carlos Román}, 
	pdfsubject={MSC 2020: Primary 49N60; Secondary 82D40, 74N15},
	pdfkeywords={domain branching, regularity, local energy estimates, microstructure, micromagnetism, pattern formation, scaling laws}} 
\usepackage[initials,nobysame,alphabetic,msc-links]{amsrefs}
\begin{document}
\title[Domain branching in micromagnetism]{Domain branching in micromagnetism: scaling law for the global and local energies}

\author[T.~Ried]{Tobias Ried}
\address[T.~Ried]{Max-Planck-Institut für Mathematik in den Naturwissenschaften, Inselstraße 22, 04103 Leipzig, Germany}
\email{tobias.ried@mis.mpg.de}

\author[C.~Rom\'an]{Carlos Rom\'an}
\address[C.~Rom\'an]{Facultad de Matem\'aticas e Instituto de Ingenier\'ia Matem\'atica y Computacional, Pontificia Universidad Cat\'olica de Chile, Vicu\~na Mackenna 4860, 7820436 Macul, Santiago, Chile}
\email{carlos.roman@uc.cl}
\date{\today}
\thanks{\emph{Funding information}: ANID FONDECYT 11190130; ANID FONDECYT 1231593; DFG EXC-2047/1 – 390685813; EPSRC grant no EP/R014604/1.\\[1ex]
\textcopyright 2024 by the authors. Faithful reproduction of this article, in its entirety, by any means is permitted for noncommercial purposes.}


\begin{abstract}
We study the occurrence of \emph{domain branching} in a class of $(d+1)$-dimensional sharp interface models featuring the competition between an interfacial energy and a non-local field energy. Our motivation comes from branching in uniaxial ferromagnets corresponding to $d=2$, but our result also covers twinning in shape-memory alloys near an austenite-twinned-martensite interface (corresponding to $d=1$, thereby recovering a result of \cite{Con00}).  

We prove that the energy density of a minimising configuration in a large cuboid domain $Q_{L,T}=[-L,L]^d\times [0,T]$ scales like $T^{-\frac{2}{3}}$ (irrespective of the dimension $d$) if $L\gg T^{\frac{2}{3}}$. While this already provides a lot of insight into the nature of minimisers, it does not characterise their behaviour close to the top and bottom boundaries of the sample, i.e.\ in the region where the branching is concentrated.
More significantly, we show that minimisers exhibit a self-similar behaviour near the top and bottom boundaries in a statistical sense through local energy bounds: for any minimiser in $Q_{L,T}$, the energy density in a small cuboid $Q_{\ell,t}$ centred at the top or bottom boundaries of the sample, with side lengths $\ell \gg t^{\frac{2}{3}}$, satisfies the same scaling law, that is, it is of order $t^{-\frac{2}{3}}$. 
\end{abstract}
\maketitle
\tableofcontents
\setcounter{secnumdepth}{3}
\section{Introduction}

In this article we revisit the mathematical study of \emph{domain branching} in ferromagnets. We will do so working within the framework of \emph{micromagnetism}, a powerful continuum theory which successfully explains observations of phenomena in ferromagnetic materials on many length scales, ranging from nanometres to microns. All of these scales are large enough to neglect the description of the atomic structure of the material, hence allowing for the use of continuum physics.

A ferromagnet, like iron, is a material having a high susceptibility to magnetisation, that is, they are noticeably attracted to the magnetic fields generated by magnets. The strength of the magnetisation depends on that of the applied field, and may persist even if the external field is removed. At the atomic level, this is explained by parallel magnetic alignment of neighbouring atoms. 

The main quantity of interest in this theory is the \emph{magnetisation density} $\M$. It is defined as the magnetic dipole moment -- which may be thought of as a measure of a dipole's ability to turn itself into alignment with a given applied field -- per unit volume. Denoting by $\Omega\subset\RR^3$ the material sample, it is defined as a vector field in $\Omega$, which, far below the Curie temperature $T_C$, has constant length, that is,
$$
|\M|=M_s\quad \mbox{in }\Omega.
$$
Here, $M_s$ denotes the saturation magnetisation, a material constant (at fixed temperature). 

We will next consider the rescaled and extended magnetisation $\m:\RR^3\to \RR^3$ defined as $\m = \frac{\M}{M_s}$ in $\Omega$ and zero elsewhere, so that
$$
|\m|^2=
\left\{
\begin{array}{cl}
1&\mbox{in }\Omega\\
0&\mbox{elsewhere}.
\end{array}
\right.
$$ 
The magnetisation $\m$ induces a \emph{stray field} (demagnetisation field) $\h:\RR^3\to \RR^3$, which is obtained by solving the (normalised) Maxwell equations of magnetostatics
$$
\curl \h=0\quad \mbox{and}\quad \diver(\h+\m)=0\quad \mbox{in }\RR^3.
$$
Hence, $\h$ is the Helmholtz projection of (the extended) $\m$. 
These equations have to be understood in the sense of distributions, the latter equation thus means that
$$
\int_{\RR^3}\h\cdot \nabla \phi \,\mathrm{d}x  =-\int_\Omega \m\cdot \nabla \phi \,\mathrm{d}x  \quad \forall \phi\in C_c^\infty(\RR^3).
$$
In particular, there are two sources for $\h$, corresponding to the two components of the divergence of $\m$. For a (sufficiently) smooth $\m$, the densities of these components are given by $\diver \m$ at points of $\Omega$ and $-\m\cdot \nu$ at points of the boundary $\partial \Omega$, where $\nu$ denotes the outer unit normal to $\partial \Omega$. By an analogy with electrostatics, these are called (magnetic) volume and surface charges. 

\medskip
Landau and Lifschitz in \cite{LL44} introduced the so-called Landau--Lifschitz energy for micromagnetism, which has successfully predicted the behaviour of ferromagnets in a vast range of situations. In normalised form, in absence of an applied magnetic field and at fixed temperature (far below $T_C$), is given by
$$
\E(\m)= d^2\int_\Omega |\nabla \m|^2\,\mathrm{d}x +Q\int_\Omega \varphi(\m) \,\mathrm{d}x +\int_{\RR^3}|\h|^2 \,\mathrm{d}x .
$$
Here:
\begin{itemize}
\item The first term is the \emph{exchange energy}. It favours the alignment of the magnetisation along a common direction, that is, a uniform $\m$ in $\Omega$. It is of quantum mechanical origin and models a short range attraction between the spins. The material parameter $d$, called \emph{exchange length}, is its intrinsic length scale.

\item The second term is the \emph{anisotropy energy}. It comes from the interaction of the magnetisation $\m$ with the lattice structure of the material.\\
The non-negative function $\varphi:\mathbb S^2\to \RR$ enforces a preference for the direction of the magnetisation. We will be interested in the uniaxial case 
$$
\varphi(\m)=|\m'|^2,\quad \mbox{where } \m'=\begin{pmatrix}\m_1\\\m_2\end{pmatrix},
$$
which favours the third axis and thus the directions $\pm(0,0,1)$. \\
The dimensionless material parameter $Q$ is called the quality factor. It separates ferromagnetic materials into two broad classes: soft materials, for which $Q<1$ and hard materials, for which $Q>1$. We will be interested in strongly uniaxial materials, for which $Q$ is a large parameter.
\end{itemize}

From now on, we will be interested in the case of an idealised ferromagnet in the form of an infinite slab of thickness $t$ that is normal to the easy axis, i.e. $\RR^2\times [0,t]$. In order to deal with the unboundedness of the domain, we will impose some artificial periodicity $2\ell$ in the first two space coordinates, that is, $\Omega=[-\ell,\ell)^2\times [0,t]$ and 
$$
\m(x)=\m(x_1+2\ell,x_2,x_3)=\m(x_1,x_2+2\ell,x_3)\quad \forall x\in \RR^3.
$$
\begin{itemize}
\item The third term is the non-local \emph{stray field energy}. It favours magnetisations whose induced stray field is reduced as much as possible. 
One has that 
\small
\begin{align*}
	&\int_{[-\ell,\ell)^2\times\RR} |\h|^2\,\mathrm{d}x \\
	&\qquad= \min\left\{ \int_{[-\ell,\ell)^2\times\RR} |\widetilde{\h}|^2\,\mathrm{d}x : \widetilde{\h}: \RR^3 \to \RR^3 \text{ is } [-\ell,\ell)^2\text{-periodic in } (x_1,x_2),\, \nabla\cdot (\widetilde{\h} + \m) = 0 \right\} \\
	&\qquad= \int_{[-\ell,\ell)^2\times\RR} \left| |\nabla|^{-1}\nabla \cdot \m \right|^2\,\mathrm{d}x,
\end{align*}
\normalsize
see \cite{OV10}*{Appendix} for more details on the stray field energy. Hence, instead of minimising the non-local energy $E(\m)$, we can also include the field $\h$ in the minimisation, which makes the problem more local,\footnote{This is actually the point of view that we will take in the remainder of the article.} and is similar to the localisation of the fractional Laplacian via extension, see \cite{CafSil07}.
\end{itemize}

\medskip
We next heuristically explain why and when the parameters of the model allow for domain branching to occur. For a detailed explanation, we refer the reader to \cites{HS98,DKMO06} and references therein. Also, in these references the variety of microstructures that can be observed in ferromagnets is discussed, including more about the physics background of the model. The expert reader may want to skip this discussion and move on to the statement of the main results in Section~\ref{sec:main}.

Observe that the anisotropy and exchange energies favour the uniform magnetisations $\m=\pm(0,0,1)$. Nevertheless, the distributional divergence of $\pm(0,0,1)$ consist of surface charges of density $\pm1$ and $\mp 1$ respectively at the top ($x_3=t$) and bottom ($x_3=0$) boundaries, which generate a constant stray field $\h=\mp 1$ in $\Omega$. This leads to a stray field energy (and thus a Landau--Lifschitz energy) of $t$ per area in the cross section $x_1x_2$. 

However, the stray field energy can be reduced if the third component of the magnetisation (and therefore the surface charges) alternates between $\pm1$ and $\mp1$ at the top and bottom boundaries. 
This corresponds to \emph{screening}, the main driving principle in electrostatics: the minimising charge distribution has the property that on mesoscopic scales the charges try to arrange themselves in such a way that the macroscopic part of their induced stray field is reduced as much as possible. 

Indeed, for a magnetisation that only depends on $x'=(x_1,x_2)$ and horizontally alternates between $\pm(0,0,1)$ and $\mp(0,0,1)$ with a period $\omega\ll t$, the induced stray field concentrates in a neighbourhood (in the $x_3$-direction) of the top and bottom boundaries of size $\omega$, which leads to a stray field energy of order $\omega$ per area in the cross section $x_1x_2$. 

The exchange energy of course prevents that $\m$ jumps, which leads to the formation of \emph{domains}, i.e.\ (large) regions where $\m$ is nearly constant and equal to $\pm(0,0,1)$, separated by \emph{walls}, i.e.\ (small) smooth transition layers between $\pm (0,0,1)$ and $\mp(0,0,1)$. In strongly uniaxial materials (i.e. when $Q$ is large), the latter are the so-called \emph{Bloch walls}. In this case, $\m$ smoothly rotates in the $x_2x_3$ plane as one crosses the transition layer in the normal $x_1$ direction. 

For a magnetisation that only depends on $x_1$, by balancing the exchange and anisotropy energies, one finds that the width of the walls $\omega_{\mathrm{wall}}$ must be of order $dQ^{-\frac12}$ and that the \emph{specific} wall energy is of order $dQ^\frac12$. Letting $\omega_{\mathrm{domain}}$ denote the domain width, this leads to a wall energy contribution, per area in the cross section $x_1x_2$, of order $dQ^{\frac12}\omega_{\mathrm{domain}}^{-1}t$. Hence, combining with the stray field energy, we obtain a Landau--Lifschitz energy, per area in the cross section $x_1x_2$, of order $dQ^{\frac12}\omega_{\mathrm{domain}}^{-1}t+\omega_{\mathrm{domain}}$, leading to an optimal domain width $\omega_{\mathrm{domain}}$ and energy per area of order $(dQ^\frac12 t)^\frac12$. 

Let us observe that for this to be consistent, i.e. for $\omega_{\mathrm{wall}}\ll \omega_{\mathrm{domain}}\ll t$ to hold, the condition $dQ^\frac12\ll t$ is needed. Thus, in this regime, the uniform magnetisation is energetically beaten. It would then seem natural to think that the minimiser behaves as described above. Nevertheless, by varying the domain width in the $x_3$-direction, one can further reduce the energy. Intuitively, it is convenient to have a very small domain width near the top and bottom boundaries to reduce the stray field energy, but away from these surfaces, it is better to have a large domain width to reduce the wall energy. 

This is achieved by domain branching, see Figure \ref{ferro2}, firstly introduced by Lifschitz \cite{Lif35} (see also \cites{Hub67,HS98}). It is worth mentioning that a similar branching phenomenon, related on a mathematical level to optimal transportation, occurs in type-I superconductors; see \cites{CKO04,ChCKO08,COS16,CGOS18}.
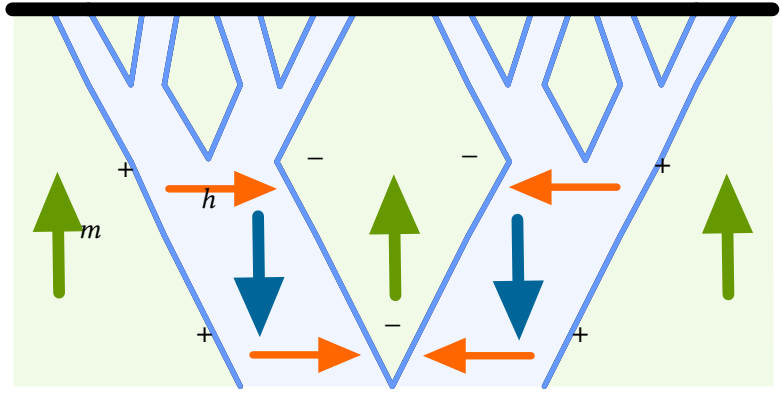
\begin{figure}[h]
\centering
\definecolor{ffzzqq}{rgb}{1.,0.6,0.}
\definecolor{qqwwzz}{rgb}{0.,0.4,0.6}
\definecolor{zzttqq}{rgb}{0.6,0.2,0.}
\definecolor{zzccqq}{rgb}{0.6,0.8,0.}
\definecolor{ffwwqq}{rgb}{1.,0.4,0.}
\definecolor{qqwwzz}{rgb}{0.,0.4,0.6}
\definecolor{wwzzqq}{rgb}{0.4,0.6,0.}
\definecolor{wwccqq}{rgb}{0.4,0.8,0.}
\definecolor{wwzzff}{rgb}{0.4,0.6,1.}
\definecolor{ttttff}{rgb}{0.,0.4,0.6}

\begin{tikzpicture}[line cap=round,line join=round,>=triangle 45,x=1cm,y=1cm]
\fill[line width=2.pt,color=wwzzff,fill=wwzzff,fill opacity=0.10000000149011612] (2.,0.) -- (1.,2.) -- (0.56,2.98) -- (0.,4.) -- (-0.48,5.) -- (0.,5.) -- (0.56,4.) -- (0.72,5.02) -- (1.18,5.) -- (1.,4.) -- (1.58,3.02) -- (2.,4.) -- (1.66,5.) -- (2.24,5.) -- (2.52,3.98) -- (3.,5.) -- (3.52,5.) -- (3.,4.) -- (2.48,2.98) -- (3.,2.) -- (4.,0.) -- cycle;
\fill[line width=2.pt,color=wwzzff,fill=wwzzff,fill opacity=0.10000000149011612] (4.,0.) -- (5.,2.) -- (5.54,2.98) -- (5.,4.) -- (4.52,5.) -- (5.,5.) -- (5.52,4.) -- (5.84,5.) -- (6.34,5.) -- (6.,4.) -- (6.54,3.) -- (7.,4.) -- (6.66,5.) -- (7.14,5.) -- (7.54,4.) -- (8.,5.) -- (8.54,5.) -- (8.,4.) -- (7.52,2.98) -- (7.,2.) -- (6.,0.) -- cycle;
\fill[line width=2.pt,color=wwccqq,fill=wwccqq,fill opacity=0.10000000149011612] (-0.9820282067160223,5.) -- (-0.9820282067160223,0.) -- (2.,0.) -- (1.,2.) -- (0.56,2.98) -- (0.,4.) -- (-0.4971783891120346,5.) -- cycle;
\fill[line width=2.pt,color=wwccqq,fill=wwccqq,fill opacity=0.10000000149011612] (0.03870825139763598,5.) -- (0.56,4.) -- (0.7021869491715138,5.) -- cycle;
\fill[line width=2.pt,color=wwccqq,fill=wwccqq,fill opacity=0.10000000149011612] (1.,4.) -- (1.58,3.02) -- (2.,4.) -- (1.6463681729266477,5.) -- (1.16151835532266,5.) -- cycle;
\fill[line width=2.pt,color=wwccqq,fill=wwccqq,fill opacity=0.10000000149011612] (2.52,3.98) -- (2.9733255684744035,5.) -- (2.24,5.) -- cycle;
\fill[line width=2.pt,color=wwccqq,fill=wwccqq,fill opacity=0.10000000149011612] (9.,5.) -- (9.016780211962905,0.) -- (6.,0.) -- (7.52,2.98) -- (8.,4.) -- (8.54,5.) -- cycle;
\fill[line width=2.pt,color=wwccqq,fill=wwccqq,fill opacity=0.10000000149011612] (8.,5.) -- (7.54,4.) -- (7.168363300657509,5.) -- cycle;
\fill[line width=2.pt,color=wwccqq,fill=wwccqq,fill opacity=0.10000000149011612] (6.345894029332739,5.) -- (6.659114684946519,5.) -- (7.,4.) -- (6.54,3.) -- (6.,4.) -- cycle;
\fill[line width=2.pt,color=wwccqq,fill=wwccqq,fill opacity=0.10000000149011612] (5.846434064975631,5.) -- (5.52,4.) -- (5.,5.) -- cycle;
\fill[line width=2.pt,color=wwccqq,fill=wwccqq,fill opacity=0.10000000149011612] (3.5239743328422217,5.) -- (4.548207381995823,5.) -- (5.,4.) -- (5.54,2.98) -- (5.,2.) -- (4.,0.) -- (3.,2.) -- (2.48,2.98) -- (3.,4.) -- cycle;
\draw [line width=2.pt] (2.,0.)-- (1.,2.);
\draw [line width=2.pt] (4.,0.)-- (3.,2.);
\draw [line width=2.pt] (4.,0.)-- (5.,2.);
\draw [line width=2.pt] (6.,0.)-- (7.,2.);
\draw [line width=2.pt] (1.58,3.02)-- (1.,4.);
\draw [line width=2.pt] (0.56,2.98)-- (0.,4.);
\draw [line width=2.pt] (1.58,3.02)-- (2.,4.);
\draw [line width=2.pt] (3.,2.)-- (2.48,2.98);
\draw [line width=2.pt] (2.48,2.98)-- (3.,4.);
\draw [line width=2.pt] (1.,2.)-- (0.56,2.98);
\draw [line width=2.pt] (5.,2.)-- (5.54,2.98);
\draw [line width=2.pt] (7.,2.)-- (7.52,2.98);
\draw [line width=2.pt] (7.52,2.98)-- (8.,4.);
\draw [line width=2.pt] (5.54,2.98)-- (5.,4.);
\draw [line width=2.pt] (6.54,3.)-- (6.,4.);
\draw [line width=2.pt] (6.54,3.)-- (7.,4.);
\draw [line width=2.pt] (5.,4.)-- (4.52,5.);
\draw [line width=2.pt] (5.52,4.)-- (5.,5.);
\draw [line width=2.pt] (5.52,4.)-- (5.84,5.);
\draw [line width=2.pt] (6.,4.)-- (6.34,5.);
\draw [line width=2.pt] (7.,4.)-- (6.66,5.);
\draw [line width=2.pt] (7.54,4.)-- (7.14,5.);
\draw [line width=2.pt] (7.54,4.)-- (8.,5.);
\draw [line width=2.pt] (8.,4.)-- (8.54,5.);
\draw [line width=2.pt] (3.,4.)-- (3.52,5.);
\draw [line width=2.pt] (2.52,3.98)-- (3.,5.);
\draw [line width=2.pt] (2.52,3.98)-- (2.24,5.);
\draw [line width=2.pt] (2.,4.)-- (1.66,5.);
\draw [line width=2.pt] (1.,4.)-- (1.18,5.);
\draw [line width=2.pt] (0.56,4.)-- (0.72,5.02);
\draw [line width=2.pt] (0.56,4.)-- (0.,5.);
\draw [line width=2.pt] (0.,4.)-- (-0.48,5.);
\draw [line width=2.pt,color=wwzzff] (2.,0.)-- (1.,2.);
\draw [line width=2.pt,color=wwzzff] (1.,2.)-- (0.56,2.98);
\draw [line width=2.pt,color=wwzzff] (0.56,2.98)-- (0.,4.);
\draw [line width=2.pt,color=wwzzff] (0.,4.)-- (-0.48,5.);
\draw [line width=2.pt,color=wwzzff] (-0.48,5.)-- (0.,5.);
\draw [line width=2.pt,color=wwzzff] (0.,5.)-- (0.56,4.);
\draw [line width=2.pt,color=wwzzff] (0.56,4.)-- (0.72,5.02);
\draw [line width=2.pt,color=wwzzff] (0.72,5.02)-- (1.18,5.);
\draw [line width=2.pt,color=wwzzff] (1.18,5.)-- (1.,4.);
\draw [line width=2.pt,color=wwzzff] (1.,4.)-- (1.58,3.02);
\draw [line width=2.pt,color=wwzzff] (1.58,3.02)-- (2.,4.);
\draw [line width=2.pt,color=wwzzff] (2.,4.)-- (1.66,5.);
\draw [line width=2.pt,color=wwzzff] (1.66,5.)-- (2.24,5.);
\draw [line width=2.pt,color=wwzzff] (2.24,5.)-- (2.52,3.98);
\draw [line width=2.pt,color=wwzzff] (2.52,3.98)-- (3.,5.);
\draw [line width=2.pt,color=wwzzff] (3.,5.)-- (3.52,5.);
\draw [line width=2.pt,color=wwzzff] (3.52,5.)-- (3.,4.);
\draw [line width=2.pt,color=wwzzff] (3.,4.)-- (2.48,2.98);
\draw [line width=2.pt,color=wwzzff] (2.48,2.98)-- (3.,2.);
\draw [line width=2.pt,color=wwzzff] (3.,2.)-- (4.,0.);
\draw [line width=2.pt,color=wwzzff] (4.,0.)-- (5.,2.);
\draw [line width=2.pt,color=wwzzff] (5.,2.)-- (5.54,2.98);
\draw [line width=2.pt,color=wwzzff] (5.54,2.98)-- (5.,4.);
\draw [line width=2.pt,color=wwzzff] (5.,4.)-- (4.52,5.);
\draw [line width=2.pt,color=wwzzff] (4.52,5.)-- (5.,5.);
\draw [line width=2.pt,color=wwzzff] (5.,5.)-- (5.52,4.);
\draw [line width=2.pt,color=wwzzff] (5.52,4.)-- (5.84,5.);
\draw [line width=2.pt,color=wwzzff] (5.84,5.)-- (6.34,5.);
\draw [line width=2.pt,color=wwzzff] (6.34,5.)-- (6.,4.);
\draw [line width=2.pt,color=wwzzff] (6.,4.)-- (6.54,3.);
\draw [line width=2.pt,color=wwzzff] (6.54,3.)-- (7.,4.);
\draw [line width=2.pt,color=wwzzff] (7.,4.)-- (6.66,5.);
\draw [line width=2.pt,color=wwzzff] (6.66,5.)-- (7.14,5.);
\draw [line width=2.pt,color=wwzzff] (7.14,5.)-- (7.54,4.);
\draw [line width=2.pt,color=wwzzff] (7.54,4.)-- (8.,5.);
\draw [line width=2.pt,color=wwzzff] (8.,5.)-- (8.54,5.);
\draw [line width=2.pt,color=wwzzff] (8.54,5.)-- (8.,4.);
\draw [line width=2.pt,color=wwzzff] (8.,4.)-- (7.52,2.98);
\draw [line width=2.pt,color=wwzzff] (7.52,2.98)-- (7.,2.);
\draw [line width=2.pt,color=wwzzff] (7.,2.)-- (6.,0.);
\draw [->,line width=4.4pt,color=wwzzqq] (4.041090798733158,1.208533609159039) -- (4.018509499202915,2.8118058758062467);
\draw [->,line width=4.4pt,color=wwzzqq] (-0.38484390919435324,1.242405558454403) -- (-0.40742520872459626,2.84567782510161);
\draw [->,line width=4.4pt,color=wwzzqq] (8.42186290760018,1.2198242589241604) -- (8.399281608069941,2.8230965255713683);
\draw [->,line width=4.4pt,color=qqwwzz] (2.2336858754460764,2.257443130032831) -- (2.2562671749763186,0.6541708633856227);
\draw [->,line width=4.4pt,color=qqwwzz] (5.645955122917289,2.211159623689868) -- (5.668536422447531,0.6078873570426597);
\draw [->,line width=2.8pt,color=ffwwqq] (2.1668429377230387,0.4181881256005565) -- (3.58946480812831,0.4181881256005565);
\draw [->,line width=2.8pt,color=ffwwqq] (5.825013461622307,0.4068974758354355) -- (4.402391591217036,0.4068974758354355);
\draw (1.2861722560435846,0.9149767152658884) node[anchor=north west] {$+$};
\draw (6.231476853166671,0.9149767152658884) node[anchor=north west] {$+$};
\draw (2.7313754259790985,3.1731066682901243) node[anchor=north west] {$-$};
\draw (3.7475339048400067,0.9601393143263732) node[anchor=north west] {$-$};
\draw [->,line width=2.8pt,color=ffwwqq] (1.0603592607411603,2.6198648297991873) -- (2.4829811311464316,2.6198648297991873);
\draw [->,line width=2.8pt,color=ffwwqq] (6.954078438134428,2.642446129329429) -- (5.531456567729157,2.642446129329429);
\draw (0.2474324776524341,3.1053627696993975) node[anchor=north west] {$+$};
\draw (7.315379230618307,3.150525368759882) node[anchor=north west] {$+$};
\draw (4.763692383700915,3.195687967820367) node[anchor=north west] {$-$};
\draw (-0.2493561120128988,2.26985468708043) node[anchor=north west] {$m$};
\draw (1.353916154634312,2.766643276745762) node[anchor=north west] {$h$};
\draw [line width=5.2pt] (0.,5.)-- (8.,5.);
\draw [line width=5.2pt] (8.,5.)-- (9.,5.);
\draw [line width=5.2pt] (0.,5.)-- (-1.,5.);
\end{tikzpicture}
\caption{Branched domain structure.}
\label{ferro2}
\end{figure}

The branching though comes at a price, due to the fact that the inter-facial layers between $\pm (0,0,1)$ and $\mp (0,0,1)$ are now tilted. They thus carry charges in the bulk, which generate a stray field. Letting $\omega_{\mathrm{bulk}}$ denote the domain width in the bulk, arguing similarly as above, one finds that the total (bulk) energy per area in the cross section $x_1x_2$ is of order $dQ^\frac12 \omega_{\mathrm{bulk}}^{-1}t+\omega_{\mathrm{bulk}}^2t^{-1}$, leading to an optimal $\omega_{\mathrm{bulk}}$ of order $(dQ^\frac12t^2)^\frac13$ and energy per area in the cross section $x_1x_2$ of order $((dQ^\frac12)^2t)^\frac13$. We immediately see that, provided $dQ^\frac12\ll t$, the branched configuration is energetically advantageous, and that we have consistency, i.e.
$$
\omega_{\mathrm{wall}}\ll \omega_{\mathrm{domain}}\ll \omega_{\mathrm{bulk}}\ll t
$$
holds.

Choksi, Kohn, and Otto \cite{CKO99}, were the first to mathematically prove that the above considerations are correct. More precisely, they proved that there exist universal constants $c_0,C_0>0$ such that
\begin{equation}\label{eq:globalscalingfullfunctional}
c_0((dQ^\frac12)^2t)^\frac13\leq \frac1{\ell^2}\min_{\m} \E(\m)\leq C_0((dQ^\frac12)^2t)^\frac13
\end{equation}
in the regime 
\begin{equation}\label{eq:parametersregime}
Q\gg 1,\quad dQ^\frac12\ll t,\quad \mbox{and}\quad \ell\gg ((dQ^\frac12)t^2)^\frac13.
\end{equation}
Notice that the last condition ensures that the artificially imposed period $2\ell$ is larger than $\omega_{\mathrm{bulk}}$, i.e. there is enough room for bulk domains to occur. Actually, the model considered in \cite{CKO99} is a sharp-interface reduction of micromagnetics, hence slightly simpler. The exchange energy is replaced by the term $d^2\int_\Omega |\nabla \m|$, which has to be understood as the total variation in $\Omega$ of the measure $\nabla \m$. This part of the energy essentially captures the total area of the domain walls (up to the factor $2d^2$). Later, in \cite{DKMO06}, an argument to obtain the lower bound for the minimal micromagnetics energy (instead of the sharp-interface model) with the same energy scaling was provided. 

Moreover, in \cite{CKO99}, the authors proved that if the domain structure is restricted to be independent of $x_3$, then the scaling law of the minimum is different, precisely as $\ell^2(dQ^\frac12 t)^\frac12$. This coincides with the above discussion and strongly suggests that domain branching is required for energy minimisation.

This work extended previous results by Choksi and Kohn \cite{CK98}, where analogous result for the 2D version of the same problem were obtained, to three dimensions. Both articles were motivated by the highly influential work by Kohn and Müller \cites{KM92,KM94} concerning branching of twins near an austenite-twinned-martensite interface. In these papers, the authors identified the scaling law of the minimum energy for a certain non-convex and non-local variational problem regularised by small surface energy. Conti \cite{Con00} was able to go beyond the scaling law of the minimum energy. Notice that even though the bounds on the minimum energy provide insight about the shape of minimisers, they do not give precise information about the local behaviour of minimisers near the interfaces where branching is expected to occur. The purpose of Conti's work was to address this question, introducing new ideas which allowed for proving that minimisers are self-similar in a statistical sense (a more precise description is given later).

Conti's result was extended to 3D by Viehmann in his Ph.D.\ thesis \cite{Vie09}.\footnote{In fact, our initial motivation was to understand some of the results contained in the Ph.D.\ thesis of Viehmann.} 
In order to get to it, let us start by describing a $\Gamma$-type convergence established by Otto--Viehmann in \cite{OV10}, in which the limiting energy turns out to be the 3D generalisation of the 2D functional proposed by Kohn--M\"{u}ller \cites{KM92,KM94}. 

Given any $\sigma, T>0$, we let 
$$
\gamma'\coloneqq \frac{(\sigma T^2)^\frac13}{(dQ^\frac12 t^2)^\frac13}\quad \mbox{and}\quad \gamma_3\coloneqq \frac Tt. 
$$
We then perform the change of variables
$$
\hat x'=\gamma'x',\quad \hat x_3=\gamma_3 x_3,
$$
and define
$$
\varepsilon\coloneqq \left(\frac{T}{\sigma}\right)^\frac13 \left(\frac{dQ^\frac12}t\right)^\frac13=\frac{\gamma_3}{\gamma'},\quad \delta\coloneqq \left(\frac{T}{\sigma}\right)^\frac23\frac{dQ^{-\frac12}}{(dQ^\frac12t^2)^\frac13}=\frac{dQ^{-\frac12}}{\sigma}\gamma',\quad L\coloneqq \frac{(\sigma T^2)^\frac13}{(dQ^\frac12 t^2)^\frac13} \ell=\gamma'\ell.
$$
Notice in particular that $\delta\varepsilon^{-2}=Q^{-1}$. Denoting $\hat x= (\hat x',\hat x_3)$ and $\h=(\h',\h_3)$, by rescaling the stray field as
$$
\hat \h'(\hat x)=\frac1\varepsilon \h'(x), \quad \hat \h_3 (\hat x)=\h_3(x)
$$
we see that $\widehat \diver\;  \hat \h=\gamma_3^{-1} \diver \h$, where $\widehat \diver$ denotes the divergence with respect to the $\hat x$ coordinates. We also set $\hat \m(\hat x)=\m(x)$. Finally, we let 
$$
\hat{\E}(\hat \m)=\frac12\frac{L^4}{\ell^4}\frac{t}{T}\E(\m)=\frac12\frac{\gamma'^4}{\gamma_3}\E(\m).
$$
A direct computation then shows that
\begin{equation*}
\hat{\E}(\hat \m)=\frac{1}{2}\left(\sigma^2 \delta \int\limits_{(-L,L)^2 \times (0, T)}\left| \binom{\hat\nabla '}{\varepsilon \frac{\partial}{\partial \hat x_3}} \hat \m\right|^2\mathrm{d}\hat x +\frac1\delta \int\limits_{(-L,L)^2 \times (0, T)}|\hat \m'|^2\mathrm{d}\hat x+\int\limits_{(-L,L)^2\times \RR}\left|\binom{\hat \h'}{\frac1\varepsilon \hat \h_3}\right|^2\mathrm{d}\hat x\right)
\end{equation*}
and that in the sense of distributions there holds 
$$
\hat \nabla'\cdot \left(\hat \h+\frac1\varepsilon \hat \m'\right)+\frac{\partial}{\partial \hat x_3}(\hat \h_3+\hat \m_3)=0\quad \mbox{in }\RR^3.
$$
Here, $\hat \nabla'$ denotes the gradient with respect to the first two spaces variables in the $\hat x$ coordinates.

Let us observe that the regime \eqref{eq:parametersregime} is characterised in the parameters $\varepsilon,\delta,L,\sigma,T$ by the (equivalent) conditions
$$
\delta\varepsilon^{-2}\ll 1,\quad \frac{\sigma}{T}\varepsilon^3\ll 1,\quad \frac{L}{(\sigma T^2)^\frac13}\gg1,
$$
which, for fixed $\sigma,T$, can be conveniently written as $\delta\ll \varepsilon^2\ll 1\ll L$. 

A main result in \cite{OV10} is to establish a $\Gamma$-type convergence\footnote{For some technical reasons their result does not provide a full $\Gamma$-convergence result, see \cite{OV10}*{Theorem 2}.} for the energy functional $\hat \E(\hat \m)$. More precisely, for any $L\gg (\sigma T^2)^\frac13$, as $\delta\varepsilon^{-2}\to 0$ and $\varepsilon\to 0$ (where no order of the limits has to be imposed), the functional $\hat \E(\hat \m)$ converges to the \emph{sharp interface} micromagnetics functional, in the case $d=2$,
$$
\tilde E(m) \coloneqq \sigma\int_{Q_{L,T}} |\nabla'm| + \frac{1}{2}\int_{[-L,L]^d\times \RR} |h|^2\mathrm{d}x.
$$
Here, $m\in\{-1,1\}$ a.e.\ in $Q_{L,T}\coloneqq [-L,L]^d\times [0,T]$ denotes the magnetisation. It is $2L$-periodic in the first $d$ space coordinates and equal to $0$ elsewhere (i.e. for $x_{d+1}\leq 0$ and $x_{d+1}\geq T$). $\nabla '$ denotes the gradient with respect to the first $d$ space coordinates, and the stray field $h: \RR^{d+1} \to \RR^d$ induced by $m$ satisfies
\begin{equation}\label{eq:admissible}
\nabla' \cdot h + \partial_{d+1} m = 0\quad \text{and} \quad \nabla' \times h = 0 \quad \text{in }\RR^{d+1}
\end{equation}
in the sense of distributions, where $\nabla'\cdot h$ and $\nabla '\times h$ respectively denote the in-plane divergence and the in-plane curl of the vector field $h$. This energy functional comes supplemented with the (weak) boundary condition $m=0$ on $\{x_{d+1}=0\}$ and $\{x_{d+1}=T\}$ in the sense that 
\begin{align*}
m(\cdot,x_{d+1}) \overset{\ast}{\rightharpoonup} 0 \quad \text{weakly-* in } L^{\infty}(\RR^d) \quad \text{as } x_{d+1} \to 0,T,
\end{align*} 
that is, \emph{infinite branching}. Let us emphasise that this condition ensures finiteness of the anisotropic stray field energy. Moreover, it implies that the generated stray field $h$ vanishes outside $Q_{L,T}$. 

Since the magnetisation in $Q_{L,T}$ takes only the values $\pm 1$, the first contribution in the energy has to be understood in the sense of BV functions, 
\small
\begin{align*}
	\int_{Q_{L,T}} |\nabla'm|\,\mathrm{d}x
	&= \sup \Bigg\{ \int_{Q_{L,T}} m\, \nabla' \cdot \xi \,\mathrm{d}x :\,  \xi \in \mathcal{C}^{\infty}(\RR^{d+1}, \RR^{d}), \xi \text{ is } [-L,L)^d\text{-periodic in } x', |\xi|\leq 1 \text{ in } Q_{L,T} \Bigg\}
\end{align*}
\normalsize
and can be interpreted as the slice-wise measure of the set where $m$ changes sign. More precisely,
$$
\int_{Q_{L,T}} |\nabla'm| = 2 \int_{0}^{T} \mathcal{H}^1\big(\partial\{m(\cdot, x_{d+1})=1\}\big)\mathrm{d}x_{d+1},
$$
where $\mathcal{H}^1\big(\partial\{m(\cdot, x_{d+1})=1\}\big)$ denotes the Hausdorff measure of the set where $m$ takes the value $1$ on the slice $x_{d+1}=const$ in $Q_{L,T}$, corresponding to the usual geometric interpretation of the gradient of the characteristic function of a set as perimeter. 

The constant $\sigma$ represents the interface energy per cross section area. In particular, this constant ensures that both terms in the energy have the same units.

\subsection{Main results}\label{sec:main}
We consider the energy functional
\begin{equation}\label{eq:energy-nonconvex}
E_{Q_{L,T}}(m,h)\coloneqq \sigma\int_{Q_{L,T}} |\nabla'm| + \frac{1}{2}\int_{Q_{L,T}} |h|^2\mathrm{d}x
\end{equation}
on the set 
\begin{align}\label{eq:admissible-periodic}
	\mathcal{A}^{\mathrm{per}}_{Q_{L,T}} \coloneqq &\bigg\{(m,h): m \in L^1_{x_{d+1}}([0,T]; BV_{x'}(Q_{L}'; \{\pm 1\})), h\in  L^2(Q_{L,T}; \RR^d), \text{ such that }  \nonumber\\ 
	&\qquad \nabla'\cdot h + \partial_{d+1} m = 0 \text{ distributionally in } \RR^{d+1}, \, m\stackrel{*}{\rightharpoonup} 0 \text{ as } x_{d+1} \to 0,T, \nonumber \\
	&\qquad \text{ and periodic lateral boundary conditions} \bigg\},
\end{align}
of admissible periodic configurations, where $Q_{L}' \coloneqq [-L,L]^d$. 
Note that the stray field energy can also be expressed as\footnote{The curl-free condition on $h$ is enforced by the minimisation, see Remark~\ref{rem:gradient}.}
\begin{align}\label{eq:representation-strayfield}
	&\min_{(m,h)\in \mathcal{A}^{\mathrm{per}}_{Q_{L,T}}} E_{Q_{L,T}}(m,h) \nonumber \\
	&= \min_{m}\left( \sigma\int_{Q_{L,T}} |\nabla'm| + \min \left\{ \frac{1}{2}\int_{Q_{L,T}} |h|^2\mathrm{d}x: \, h \text{ is } [-\ell,\ell)^2\text{-periodic in } x', \, \nabla'\cdot h + \partial_{d+1} m = 0 \right\}\right) \nonumber \\
	&= \min_{m} \left( \sigma\int_{Q_{L,T}} |\nabla'm| + \frac{1}{2} \int_{Q_{L,T}} \left| |\nabla'|^{-1} \partial_{d+1}m \right|^2\,\mathrm{d}x \right);
\end{align}
see \cite{OV10}*{Appendix} for more details. This is an important ingredient in the derivation of an ansatz-free lower bound on the minimal energy based on interpolation between $BV$ and $\dot{H}^{-1}$ with respect to the horizontal variables $x'$, see Section~\ref{sec:lower-bound-global} based on \cite{CO16}. 

We will also consider the case of zero-flux boundary conditions 
\begin{align}\label{eq:admissible-zero}
	\mathcal{A}^{0}_{Q_{L,T}} \coloneqq &\bigg\{(m,h): m \in L^1_{x_{d+1}}([0,T]; BV_{x'}(Q_{L}'; \{\pm 1\})), h\in  L^2(Q_{L,T}; \RR^d), \text{ such that }  \nonumber\\ 
	&\qquad \nabla'\cdot h + \partial_{d+1} m = 0 \text{ distributionally in } Q_{L,T}, \, m\stackrel{*}{\rightharpoonup} 0 \text{ as } x_{d+1} \to 0,T, \nonumber \\
	&\qquad \qquad \quad h\cdot \nu' = 0 \text{ on } \Gamma_{L,T} \bigg\},
\end{align}
where $\Gamma_{L,T} = \partial Q_{L}' \times [0,T]$ and the equation and boundary conditions are understood in a weak sense (as described more precisely in \eqref{eq:admissibility-relaxed}). Let us remark that even though the reduced functional was derived using periodic lateral boundary conditions, from a physics point of view, the zero-flux lateral boundary conditions are the most natural to impose in this context, since they account for the situation of a finite sample where the stray field generated by the magnetisation naturally vanishes outside the sample (recall that $m\stackrel{*}{\rightharpoonup} 0 \text{ as } x_{d+1} \to 0,T$).

\medskip 
The energy functional $E$ has the following important scaling property: for $\lambda>0$ let $m_{\lambda}(x) \coloneqq m(\lambda^{\frac{2}{3}} x', \lambda x_{d+1})$ and $h_{\lambda}(x) \coloneqq \lambda^\frac{1}{3} h(\lambda^{\frac{2}{3}} x', \lambda x_{d+1})$. Then $(m_{\lambda},h_{\lambda})\in\mathcal{A}^{\mathrm{per}/0}_{Q_{\lambda^{-2/3}L, \lambda^{-1}T}}$ and 
\begin{align}\label{eq:scalingofenergy}
	E_{Q_{\lambda^{-2/3}L, \lambda^{-1}T}}(m_{\lambda},h_{\lambda}) = \lambda^{-d}\lambda^{-\frac{1}{3}} E_{Q_{L,T}}(m,h).
\end{align}

We finally have all the ingredients to present our first result, which provides global scaling laws for this functional.\footnote{In the case $d=2$ with periodic boundary conditions this can already be found in \cite{Vie09}, while an analogous statement for $d=1$ is contained in the classic work \cite{KM94}.}
\begin{theorem}[Global scaling laws]\label{thm:globalbound}
There exist universal constants\footnote{By a universal constant we always mean a constant that only depends on the dimension $d$, but not on any system parameter (like $L,T$).} $C_{LT}<\infty$ and $0<c_{S} \leq C_{S}<\infty$ such that if $L \geq C_{LT} \sigma^{\frac{1}{3}} T^{\frac{2}{3}}$, then the minimal energy with respect to periodic or zero-flux boundary conditions satisfies 
$$
c_{S} \sigma^\frac23 L^d T^\frac13 \leq\min_{(m,h)\in\mathcal{A}^{\mathrm{per}/0}_{Q_{L,T}}} E_{Q_{L,T}}(m,h)\leq C_{S} \sigma^\frac23 L^dT^\frac13.
$$
\end{theorem}
Note that this is in perfect agreement with \eqref{eq:globalscalingfullfunctional} in the case $d=2$, and therefore with the heuristic computations performed when branching is expected to occur.  The lower bound follows almost directly from an interpolation inequality, see \cite{Vie09}*{Lemma 13} and \cite{CO16}*{Theorem 1.1}, while the upper bound is obtained via an explicit construction, which is inspired by \cite{Vie09} (in the case $d=2$).

Our main result, which has partially been obtained by T.~Viehmann in his (unpublished) PhD thesis \cite{Vie09}, goes beyond the global scaling law, capturing the self-similar behaviour of minimisers near the top and bottom boundaries. It shows that the energy within any cuboid $Q_{\ell,t}(a)\coloneqq Q_{\ell,t}+a$ for $a\in \RR^{d+1}$ sitting at the top or bottom boundaries, for lengths $\ell\ll L$ and $t\ll T$ which respect the relation $\ell \gg (\sigma t^2)^\frac13$, satisfies the same scaling laws, i.e. it is of order $\sigma^\frac23\ell^dt^\frac13$. 

\begin{theorem}[Local scaling laws] \label{thm:localbounds} 
 There exists a universal constant $C_{\ell t}<\infty$ such that the following holds: if $L \geq C_{\ell t} \sigma^{\frac{1}{3}} T^{\frac{2}{3}}$ and $\ell \geq C_{\ell t}\sigma^{\frac{1}{3}}t^\frac23$, then there exist universal constants\footnote{To be precise, the constant $c_{s}$ only depends on the dimension $d$, while $C_{s}$ depends on $d$ and $C_{\ell t}$.} $0<c_{s} \leq C_{s}<\infty$ such that any $(m,h)$ minimising \eqref{eq:energy-nonconvex} in $Q_{L,T}$ with respect to periodic or zero-flux boundary conditions satisfies
$$
c_{s} \sigma^\frac23 \ell^d t^\frac13 \leq E_{Q_{\ell,t}(a)}(m,h)\leq C_{s} \sigma^\frac23 \ell^dt^\frac13
$$
for any cuboid $Q_{\ell,t}(a)$ with $a\in \RR^d\times\{0\}$ or $a\in \RR^d\times \{T-t\}$ such that $Q_{\ell,t}(a)\subseteq Q_{L,T}$.
\end{theorem}
\begin{remark} Our results in particular show that the energy density 
$$
e_{Q_{\ell,t}}(m,h)\coloneqq \frac{1}{|Q_{\ell,t}|}E_{Q_{\ell,t}}(m,h)
$$
of any minimising configuration $(m,h)$ scales like $(\sigma t^{-1})^\frac23$ and is therefore independent of the dimension $d$. It is worth mentioning that the minimal energy density $\min_{(m,h)} e_{Q_{L,T}}(m,h)$ has a thermodynamic limit as $L\to \infty$, as shown by Otto and Viehmann \cite{OV10} (for $d=2$), therefore proving that the energy scaling is asymptotically exact.
\end{remark}

\begin{remark}
	In the following, we will set $\sigma = 1$ for simplicity; the general case can then be recovered by a simple scaling argument. 	
\end{remark}

\begin{remark} \label{rem:almost-min}
The main results can be extended to \emph{almost-minimisers} in the following sense: 
$(m_{\mathrm{almost}} ,h_{\mathrm{almost}})\in\mathcal{A}^{\mathrm{per}/0}_{Q_{L,T}}$ is an almost-minimiser at scale $(L,T)$ if $h_{\mathrm{almost}}$ is curl-free (hence a gradient field) and there exists a constant $C<\infty$ such that 
		\begin{align*}
			\int_{Q_{L,T}} |\nabla'm_{\mathrm{almost}}| + \frac{1}{2} \int_{Q_{L,T}} h_{\mathrm{almost}}^2\,\mathrm{d}x \leq \int_{Q_{L,T}} |\nabla'm| + \frac{1}{2} \int_{Q_{L,T}} h^2\,\mathrm{d}x + C L^d T^{\frac{1}{3}}
		\end{align*}
for all competitors $(m,h) \in \mathcal{A}^{\mathrm{per}/0}_{Q_{L,T}}$. 

Theorem~\ref{thm:globalbound} also holds for almost-minimisers at scale $(L,T)$ for $L \gg T^{\frac{2}{3}}$. In this sense, almost-minimisers are ``low-energy configurations'' \cite{CKO99} because their energy is within a certain factor of the minimum energy. 

The extension of Theorem~\ref{thm:localbounds} requires $(m_{\mathrm{almost}} ,h_{\mathrm{almost}})$ to be almost-minimising at any scale $\ell \leq L$, $t\leq T$ with $\ell \gg t^{\frac{2}{3}}$.
\end{remark}

For $d=1$, Theorem \ref{thm:localbounds} was first proved in the seminal work by Conti \cite{Con00} in the context of twinning in shape-memory alloys, and was recently extended by Conti, Diermeier, Koser, and Zwicknagl \cite{CDKZ21} in several directions, including the case in which there are two phases with different volume fractions, i.e. $u\in\{-\theta,2-\theta\}$, where $\frac{\theta}2\in (0,\frac12]$ represents the volume fraction of the minority phase\footnote{Notice that when $\theta=1$ we recover the case analysed in this paper.}, in both the regime where the energy scales like $\sigma^{\frac{2}{3}}$ and the regime where it scales like $\sigma^{\frac{1}{2}}$.  

In his unpublished Ph.D.\ thesis, Viehmann \cite{Vie09} extended Conti's result to $d=2$. In the process of understanding it, we realised that we could obtain the same local energy bounds via a simplified proof, that we will present in a structured way. While at the macroscopic level, our proof follows the one of Viehmann, it differs substantially in its details. Moreover, it has the advantage of working in any dimension and therefore also reproduces the result contained in \cite{Con00}. Moreover, our approach, which is inspired by \cites{ACO09,BJO22}, makes a clear distinction between the detailed study of convex relaxation(s) of the problem and constructions to transfer properties of the relaxed problems to the non-convex one. In particular, we break down the rather complicated proof in \cite{Vie09} into its basic building blocks. On that level the intricacies of several of the constructions are clearly revealed. Finally, the Campanato-type iterations used to transfer the global scaling law to small cuboids at the sample boundaries allow for some flexibility, e.g.\ the extension to almost-minimisers (at every scale) and the treatment of both periodic and zero-flux lateral boundary conditions.

Let us give a rough idea of our proof. We start by introducing the convex relaxation of the minimisation problem, that is, we minimise among functions $m$ that satisfy $m\in [-1,1]$ instead of $m\in \{-1,1\}$. This is what we call relaxed problem, which is of course simpler due to its convexity. Moreover, we consider an over-relaxation, which corresponds to an $x_{d+1}$-averaged problem, that is trivial to solve. From the over-relaxed problem we are able to construct a competitor for the relaxed one, via an explicit boundary layer construction. 
Finally, from the competitor for the relaxed problem, we construct a competitor for the non-convex problem, via a re-distribution of mass, which is compatible with the hard constraint $m\in\{-1,1\}$. Of course, in this process, we have to modify the stray field accordingly to satisfy the differential constraint relating them.

These constructions are an essential ingredient for the core argument of our proof: regularity theory in the form of a Campanato type iteration, that allows us to transfer the global scaling law to a local one. More precisely, we do two iterations\footnote{In the case of zero-flux boundary conditions, an extra iteration for cuboids centred at the boundary $\Gamma_{L,T}$ is needed.}: an initial one (which is rather standard for this kind of problems) to reduce to the case for which $L=c_{LT} T^\frac23$, where $c_{LT}$ is a coupling constant that has to be chosen large enough. Then, by monitoring two quantities at the same time, we are able to approach the top or bottom boundary both horizontally and vertically in one go, which is a major advantage and one of the main novelties of our iteration scheme. The competitor construction plays a key role in proving a \emph{one-step improvement}, which is then fed into a Campanato iteration, to obtain the local scaling law. 

A different approach, based on duality estimates, inspired by the highly influential work of Alberti, Choksi, and Otto \cite{ACO09}, will be presented in a forthcoming article \cite{RR23}. In our view, the ideas introduced in \cite{ACO09} to analyse a sharp interface limit of a model of microphase separation in diblock copolymers, are more robust in proving screening properties. In fact, they have been successfully used to treat problems where screening is a driving principle; see for instance \cites{GO20,BJO22}.

In a similar model related to domain branching in type-I superconductors derived in \cite{CGOS18}, which is actually of branched transport type, it has recently been proved in \cite{DGR23} that optimal local energy estimates characterise the conjectured Hausdorff dimension of the irrigated measure in the cross-over regime between uniform and non-uniform branching. The irrigation problem from a Dirac mass to Lebesgue measure in a two-dimensional analogue of the energy functional derived in \cite{CGOS18} has been solved by Goldman \cite{Gol17}, who proves that the minimiser is a self-similar tree.

\medskip
The rest of the article is organised as follows. In Section \ref{sec:relaxedandover} we introduce the relaxed and over-relaxed problems and construct a competitor of the former from the (unique) solution to the latter. In Section \ref{sec:constructionnonconvex} we construct a competitor for the non-convex problem based on a building block construction. In Section \ref{sec:proofglobalscalinglaw} we provide a proof for Theorem \ref{thm:globalbound}. In Section \ref{sec:prooflocalscalinglaws} we give the proof of Theorem \ref{thm:localbounds}. Finally, in Appendix \ref{sec:appendix} we establish an elliptic estimate which is needed for the building block construction of Section \ref{sec:constructionnonconvex}.
\addtocontents{toc}{\protect\setcounter{tocdepth}{1}}
\section{Relaxed and over-relaxed problems}\label{sec:relaxedandover}

In the study of minimisers of \eqref{eq:energy-nonconvex}, it is beneficial to study its convex relaxation and allow for general lateral flux boundary conditions $g\in L^2(\Gamma_{L,T})$, as well as general top and bottom magnetisations $m^T, m^B \in L^{\infty}(Q_{L}'; [-1,1])$ defined on $Q_{L}' = [-L,L]^d$. We therefore consider the following relaxed minimisation problem 
\begin{align*}
	E_{Q_{L,T}}^{\mathrm{rel}}(g; m^{B,T}) \coloneqq \inf_{h\in \widetilde{\mathcal{A}}_{Q_{L,T}}^{\mathrm{rel}}(g; m^{B,T})}  \frac{1}{2} \int_{Q_{L,T}} |h|^2\,\mathrm{d}x,
\end{align*}
where the set of admissible functions $\widetilde{\mathcal{A}}_{Q_{L,T}}^{\mathrm{rel}}(g; m^{B,T})$ consists of those functions $h \in L^2(Q_{L,T}; \RR^d)$ satisfying for all test functions $\varphi \in \mathcal{C}^1(Q_{L,T})$\footnote{That is, they satisfy in the sense of distributions
	$\partial_{d+1} m + \nabla'\cdot h = 0$ in $Q_{L,T}$, with lateral flux boundary conditions $h\cdot \nu = g$ on $\Gamma_{L,T}$ and top/bottom boundary conditions $m(\cdot, x_{d+1}) \stackrel{*}{\rightharpoonup} m^{B,T}$ as $x_{d+1} \to 0,T$.}
\begin{align}\label{eq:admissibility-relaxed}
	\int_{Q_{L,T}} (h\cdot \nabla'\varphi + m \partial_{d+1}\varphi)\,\mathrm{d}x = \int_{\Gamma_{L,T}} g \varphi + \int_{Q_{L}'} (m^T \varphi(x',T) - m^B \varphi(x',0))\,\mathrm{d}x',
\end{align}
for some $m\in L^2(Q_{L,T}; [-1,1])$. We will write 
\begin{align*}
	\mathcal{A}_{Q_{L,T}}^{\mathrm{rel}}(g; m^{B,T})\coloneq \left\{ (m,h) \in  L^2(Q_{L,T}; [-1,1]) \times L^2(Q_{L,T}; \RR^d): \eqref{eq:admissibility-relaxed} \text{ holds for all } \varphi \in \mathcal{C}^1(Q_{L,T})\right\}
\end{align*} 
and remark that 
\begin{align*}
	E_{Q_{L,T}}^{\mathrm{rel}}(g; m^{B,T}) = \inf_{(h,m) \in \mathcal{A}_{Q_{L,T}}^{\mathrm{rel}}(g; m^{B,T})}  \frac{1}{2} \int_{Q_{L,T}} |h|^2\,\mathrm{d}x.
\end{align*}
Note that the set of admissible functions $\mathcal{A}_{Q_{L,T}}^{\mathrm{rel}}(g; m^{B,T})$ is non-empty only if the boundary conditions are compatible, i.e.\
\begin{align}\label{eq:compatibility}
	\int_{\Gamma_{L,T}} g \,\mathrm{d}\mathcal{H}^{d-1} = \int_{Q_{L}'}(m^B - m^T)\,\mathrm{d}x',
\end{align}
which we shall assume henceforth. In particular $\left|\int_{\Gamma_{L,T}} g \,\mathrm{d}\mathcal{H}^{d-1} \right| \leq 2 L^d$. 

Since \eqref{eq:admissibility-relaxed} implies that the vector field $(h,m) \in L^2(Q_{L,T};\RR^{d+1})$ is divergence-free in $Q_{L,T}$, $h\cdot \nu'$ has a well-defined lateral trace\footnote{With $\nu'\in\RR^d$ we will always denote the outer unit normal to $\Gamma_{L,T}$ (respectively $\Gamma_{\ell,t}(a)$).}  $\left.h\cdot \nu'\right|_{\Gamma_{\ell,t}(a)}$ and top/bottom trace $m(\cdot, a_{d+1})$ and $m(\cdot, a_{d+1}+t)$ in almost any sub-cube $Q_{\ell,t}(a) = Q_{\ell, t}+a$ contained in $Q_{L,T}$. By Fubini's theorem, $\left.h\cdot \nu'\right|_{\Gamma_{\ell,t}(a)} \in L^2(\Gamma_{\ell, t}(a))$ for almost every $\ell \in (0, L)$ and $t\in (0,T)$.

\begin{lemma}\label{lem:gradient}
	Let $(m,h) \in \mathcal{A}_{Q_{L,T}}^{\mathrm{rel}}(g; m^{B,T})$ be a minimiser of $E_{Q_{L,T}}^{\mathrm{rel}}(g; m^{B,T})$. Then $h$ is a gradient field, i.e., there exists a potential $u\in H^1(Q_{L,T})$ such that $h= - \nabla'u$. This potential is a solution of 
	\begin{align}\label{eq:relaxed-gradient-equation}
	\begin{array}{rcll}
		-\Delta'u &=& -\partial_{d+1} m&\text{in} \; Q_{L,T}\\
		-\nabla'u\cdot \nu &=& g& \text{on} \; \Gamma_{L,T}.
	\end{array}
	\end{align}
\end{lemma}
While this is rather standard, we give a proof of Lemma~\ref{lem:gradient} for the convenience of the reader. 
\begin{proof}
	Let $(m,h) \in \mathcal{A}_{Q_{L,T}}^{\mathrm{rel}}(g; m^{B,T})$ be a minimiser of $E_{Q_{L,T}}^{\mathrm{rel}}(g; m^{B,T})$ and let $u\in H^1(Q_{L,T})$ be a solution of \eqref{eq:relaxed-gradient-equation}. Then 
	\begin{align*}
		\int_{Q_{L,T}} h^2\,\mathrm{d}x 
		&= \int_{Q_{L,T}} |\nabla'u|^2 \,\mathrm{d}x + \int_{Q_{L,T}} |h+\nabla'u|^2 \,\mathrm{d}x - 2 \int_{Q_{L,T}} (h+ \nabla'u)\cdot \nabla'u \,\mathrm{d}x \\
		&\kern-1em\stackrel{\eqref{eq:admissibility-relaxed}\&\eqref{eq:relaxed-gradient-equation}}{=} \int_{Q_{L,T}} |\nabla'u|^2 \,\mathrm{d}x + \int_{Q_{L,T}} |h+\nabla'u|^2 \,\mathrm{d}x
		\geq \int_{Q_{L,T}} |\nabla'u|^2 \,\mathrm{d}x.
	\end{align*}
	Since $(m, -\nabla'u) \in \mathcal{A}_{Q_{L,T}}^{\mathrm{rel}}(g; m^{B,T})$ by \eqref{eq:relaxed-gradient-equation}, minimality of $h$ implies that 
	\begin{align*}
		\int_{Q_{L,T}} h^2 \,\mathrm{d}x \leq \int_{Q_{L,T}} |\nabla'u|^2 \,\mathrm{d}x,
	\end{align*}
	hence $h = -\nabla' u$. 
\end{proof}
\begin{remark}\label{rem:gradient} The same argument shows that if $(m,h)$ minimises the non-convex energy \eqref{eq:energy-nonconvex} over $\mathcal{A}^{\mathrm{per}/0}_{Q_{L,T}}$, then $h$ is a gradient field.
\end{remark}

\begin{remark}
Note that any minimiser $(m,h)\in \mathcal{A}_{Q_{L,T}}^{\mathrm{rel}}(g; m^{B,T})$ of $E_{Q_{L,T}}^{\mathrm{rel}}(g; m^{B,T})$ in a cube $Q_{L,T}$ is \emph{locally minimising} in any sub-cube $Q_{\ell, t}(a) \subset Q_{L,T}$, i.e.\ minimising given its own boundary conditions. 

Indeed, let 
$$
\widetilde{g} \coloneqq \left.h\cdot \nu'\right|_{\Gamma_{\ell,t}(a)},\quad  \widetilde{m}^B \coloneqq m(\cdot, a_{d+1}),\quad \widetilde{m}^T \coloneqq m(\cdot, a_{d+1} +t),
$$ 
and 
$$(\widetilde{m},\widetilde{h})\in \mathcal{A}_{Q_{\ell,t}}^{\mathrm{rel}}(\widetilde{g}; \widetilde{m}^{B,T}).
$$
Then \[(\widehat{m}, \widehat{h}) \coloneqq (\widetilde{m}1_{Q_{\ell,t}} + m 1_{Q_{L,T} \setminus Q_{\ell,t}}, \widetilde{h}1_{Q_{\ell,t}} + h 1_{Q_{L,T} \setminus Q_{\ell,t}}) \in \mathcal{A}_{Q_{L,T}}^{\mathrm{rel}}(g; m^{B,T}),\] and we have that 
\begin{align*}
	E_{Q_{L,T}}^{\mathrm{rel}}(g; m^{B,T}) 
	= \frac{1}{2} \int_{Q_{L,T}} |h|^2\,\mathrm{d}x 
	\leq \frac{1}{2} \int_{Q_{L,T}} |\widehat{h}|^2\,\mathrm{d}x 
	= \frac{1}{2} \int_{Q_{L,T}\setminus Q_{\ell,t}} |h|^2\,\mathrm{d}x + \frac{1}{2} \int_{Q_{\ell,t}} |\widetilde{h}|^2\,\mathrm{d}x ,
\end{align*}
hence 
\begin{align*}
	\frac{1}{2} \int_{Q_{\ell,t}} |h|^2\,\mathrm{d}x \leq  \frac{1}{2} \int_{Q_{\ell,t}} |\widetilde{h}|^2\,\mathrm{d}x
\end{align*}
for any $(\widetilde{h},\widetilde{m})\in \mathcal{A}_{Q_{\ell,t}}^{\mathrm{rel}}(\widetilde{g}; \widetilde{m}^{B,T})$. 
\end{remark}

Further, given top and bottom magnetisations $m^{B,T} \in L^{\infty}(Q_L')$, we define the (curl-free) fields generating them via 
\begin{align*}
	H^{B,T} \coloneqq -\nabla' u^{B,T},
\end{align*}
where the potentials $u^{B,T}$ solve $-\Delta' u^{B,T} = - m^{B,T}$ in $Q_{L}'$. By \eqref{eq:compatibility} they have to satisfy the boundary condition $(H^T -H^B) \cdot \nu' = \int_0^T g\,\mathrm{d}x_{d+1} \eqqcolon  T \overline{g}_T$ on $\partial Q_L'$.

\subsection{Over-relaxed problem}\label{sec:ORP}
We will also look at the $x_{d+1}$-averaged problem corresponding to the linear interpolation $m_0(x) \coloneqq \frac{x_{d+1}}{T} m^T + \left(1-\frac{x_{d+1}}{T} \right) m^B$ between $m^T$ and $m^B$, which can be thought of as an  over-relaxed problem given by the energy functional 
\begin{align*}
	E_{Q_{L,T}}^0(g;m^{B,T}) = \frac{1}{2} \int_{Q_{L,T}} |\nabla' v_0|^2 \mathrm{d}x = \frac{T}{2} \int_{Q_{L}'} |\nabla' v_0|^2\,\mathrm{d}x',
\end{align*}
where $v_0$ is the unique solution\footnote{Note that the equation is solvable by \eqref{eq:compatibility}.} of 
\begin{align*}
	\begin{array}{rcll}
	-\Delta' v_0 &=& -\dfrac{m^T - m^B}{T}&\text{in}\, Q_{L,T} \\
	-\nabla' v_0 \cdot \nu &=& \overline{g}_T&\text{on} \, \Gamma_{L,T},
	\end{array}
\end{align*}
with $\int_{Q_{L,T}} v_0\,\mathrm{d}x = 0$. Let us stress that $v_0$ is constant w.r.t.\ $x_{d+1}$. 

It is easy to see that if $(m,h)$ denotes a minimiser of the relaxed problem $E_{Q_{L,T}}^{\mathrm{rel}}(g; m^{B,T})$, then the height average $\overline{h}_{T} \coloneqq \dashint_0^T h\,\mathrm{d}x_{d+1}$ is a solution of the over-relaxed problem, i.e.\ $\overline{h}_{T} = -\nabla'v_0$ and we have $E_{Q_{L,T}}^0(g;m^{B,T}) = \frac{1}{2} \int_{Q_{L,T}} |\overline{h}_T|^2\,\mathrm{d}x$. More generally, we have: 

\begin{lemma}\label{lem:cumulated-field}
	Let $(m,h) \in \mathcal{A}_{Q_{L,T}}^{\mathrm{rel}}(g, m^{B,T})$ and assume that $h$ is a gradient field. Then 
	\begin{align*}
		\int_0^T h\,\mathrm{d}x_{d+1} = H^T - H^B 	\quad \text{on} \quad Q_{L}'.
	\end{align*}
\end{lemma}
\begin{proof}
	Let $h=-\nabla'u$ for some $u\in H^1(Q_{L,T})$, then by admissibility $H = -\nabla' U \coloneqq \int_0^T -\nabla'u\,\mathrm{d}x_{d+1}$ is a weak solution of 
	\begin{align*}
		\begin{array}{rcll}
		-\Delta' U &=&m^B - m^T & \text{in} \; Q_{L}'\\
		-\nabla' U \cdot \nu' &=& \int_0^T h\cdot \nu' \,\mathrm{d}x_{d+1} = \int_0^T g \,\mathrm{d}x_{d+1}& \text{on} \; \partial Q_{L}'.
		\end{array}
	\end{align*}
	By definition of $H^{B,T}=-\nabla' u^{B,T}$, the function $u^T-u^B$ satisfies the same PDE (including boundary conditions). By uniqueness, we must have $H = H^T - H^B$. 
\end{proof}

\begin{lemma}[Orthogonality]\label{lem:orthogonality}
	Let $h\in L^2(Q_{\ell, t})$. Then
	\begin{align*}
		\frac{1}{2} \int_{Q_{\ell,t}} |h - \overline{h}_t|^2\,\mathrm{d}x &= \frac{1}{2} \int_{Q_{\ell,t}} |h|^2\,\mathrm{d}x - \frac{1}{2} \int_{Q_{\ell,t}} |\overline{h}_t|^2\,\mathrm{d}x.
	\end{align*} 
\end{lemma}
\begin{proof}
	Let $h\in L^2(Q_{\ell, t})$. Then 
	\begin{align*}
		\frac{1}{2} \int_{Q_{\ell, t}} |h - \overline{h}_t|^2 \,\mathrm{d}x 
		&= \frac{1}{2} \int_{Q_{\ell, t}} |h|^2 \,\mathrm{d}x - \frac{1}{2} \int_{Q_{\ell, t}} |\overline{h}_t|^2 \,\mathrm{d}x + \int_{Q_{\ell, t}} \overline{h}_t \cdot (\overline{h}_t - h)\,\mathrm{d}x,
	\end{align*}
	and since
	\begin{align*}
		\int_{Q_{\ell, t}} \overline{h}_t \cdot (\overline{h}_t - h)\,\mathrm{d}x
		&= \int_{Q_{\ell}'} \overline{h}_t \cdot \int_0^t (\overline{h}_t - h)\,\mathrm{d}x_{d+1}\,\mathrm{d}x' 
		= 0,
	\end{align*}
	the claim follows.
\end{proof}

\subsection{A competitor for the relaxed problem} 
Given the minimiser of the over-relaxed problem, we can use it to construct a competitor for the relaxed problem by  correcting the boundary data. This will be done in a boundary layer of size $r>0$. 

For $r>0$ we define the set $A_r(Q_{L,T}) \coloneq Q_{L}' \setminus Q_{L-r}' \times [0,T]$.
Let $\mathcal{X}_{r}(Q_{L,T})$ be the set of all $(m,h)\in L^{\infty}(A_r(Q_{L,T});[-1,1])\times L^2(A_r(Q_{L,T});\RR^d)$ satisfying $\partial_{d+1} m + \nabla'\cdot h = 0$ in $A_r(Q_{L,T})$ with boundary conditions $m(\cdot, 0) = m(\cdot, T) = 0$ and $h\cdot \nu = g-\overline{g}_T$ on $\partial Q_{L}' \times [0,T]$, $h\cdot \nu = 0$ on $\partial Q_{L-r}' \times [0,T]$.

\begin{proposition}\label{prop:fine-properties-relaxed}
Let $g\in L^2(\Gamma_{L,T})$ and  $m^{B,T} \in L^{\infty}(Q_{L}'; [-1,1])$ be such that \eqref{eq:compatibility} holds. There exists a universal constant $C<\infty$ such that if\,\footnote{The assumption may appear asymmetric in terms of $H^B$ and $H^T$. However, in view of Lemma~\ref{lem:cumulated-field}, control on $H^B$ and the cumulated field implies control on $H^T$ as well.}
\begin{align}\label{eq:boundary-layer-size}
\frac{r}{L} \geq C \max\left\{ \left( \frac{T^2}{L^2} \dashint_{\Gamma_{L,T}} (g-\overline{g}_T)^2 \right)^{\frac{1}{d+1}}, \frac{T}{L} \sup_{Q_L'} \left|\dashint_0^T h\,\mathrm{d}x_{d+1}\right|, \frac{1}{L}\sup_{Q_L'}|H^B|\right\},
\end{align}
then there exists $(m_{r},h_{r}) \in \mathcal{X}_{r}(L,T)$ with the following properties\footnote{Hereafter we use the symbol $\lesssim$ to indicate that the inequality holds up to a dimensional constant.}:
	\begin{align}
		\frac{T^2}{L^2} \frac{1}{L^d T} \int_{A_{L,T}(r)} |h_{r}|^2 &\lesssim \frac{r}{L} \frac{T^2}{L^2}\dashint_{\Gamma_{L,T}} (g-\overline{g}_T)^2, \label{eq:bound-hell} \\
		T \|\partial_{d+1} m_{r}\|_{L^2_{x_{d+1}}L^{\infty}_{x'}(A_{L,T}(r))}^2 &\lesssim \left(\frac{r}{L}\right)^{-(d+1)} \frac{T^2}{L^2} \dashint_{\Gamma_{L,T}} (g-\overline{g}_T)^2. \label{eq:bound-d3mell}
	\end{align}
\end{proposition}

\begin{remark}\label{rem:bound-mb-cumulated-field}
	Assumption \eqref{eq:boundary-layer-size} in particular implies the following estimate (in a weak topology) on the linearly interpolated magnetisation $m_0(\cdot, x_{d+1}) = \frac{x_{d+1}}{T} m^T + \left(1-\frac{x_{d+1}}{T} \right) m^B$: 
	\begin{align*}
		\sup_{P_{r}'}\left|\dashint_{P_{r}'} m_0 \,\mathrm{d}x'\right| \leq \frac{1}{2} \quad \text{for all} \quad x_{d+1} \in [0,T],
	\end{align*}
	where the supremum is taken over all plaquettes $P_{r}'$ of size $r$ in $Q_{L}' \setminus Q_{L-r}'$.

	Indeed, since $\partial_{x_{d+1}} m_0 = \frac{m^T - m^B}{T} = \Delta' v_0$, we have that 
	\begin{align*}
		\dashint_{P_{r}'} m_0 \,\mathrm{d}x' 
		&= \dashint_{P_{r}'} m^B \,\mathrm{d}x' + \dashint_{P_{r}'} \int_0^{x_{d+1}} \partial_{x_{d+1}} m_0(x', \xi_{d+1}) \,\mathrm{d}\xi_{d+1} \,\mathrm{d}x' \\
		&= \dashint_{P_{r}'} \nabla'\cdot H^B \,\mathrm{d}x' + \int_0^{x_{d+1}} \dashint_{P_{r}'} \Delta'v_0 \,\mathrm{d}x' \,\mathrm{d}\xi_{d+1} \\
		&= \frac{1}{|P_r'|} \int_{\partial P_r'} H^B \cdot \nu' \,\mathrm{d}\mathcal{H}^{d-1} + \frac{1}{|P_r'|} \int_0^{x_{d+1}}\int_{\partial P_r'} \nabla'v_0\cdot \nu' \,\mathrm{d}\mathcal{H}^{d-1} \,\mathrm{d}\xi_{d+1}\\
		&= \frac{1}{|P_r'|} \int_{\partial P_r'} H^B \cdot \nu' \,\mathrm{d}\mathcal{H}^{d-1} + \frac{1}{|P_r'|} x_{d+1}\int_{\partial P_r'} \nabla' v_0\cdot \nu' \,\mathrm{d}\mathcal{H}^{d-1}.
	\end{align*}
	Note that since $m^B, m^T \in L^{\infty}(Q_L')$, and therefore  $\partial_{x_{d+1}} m_0\in L^{\infty}(Q_L')$, the functions $H^B = -\nabla'u^B$ and $-\nabla'v_0$ are continuous by elliptic regularity\footnote{This follows from maximal $L^p$ regularity, which implies that $H^B, -\nabla'v_0 \in W^{1,p}(Q_L')$ for any $p<\infty$, and therefore $H^B, -\nabla'v_0 \in C^{0,\alpha}(Q_{L}')$ for any $\alpha \in(0,1)$ by Morrey's inequality.}, hence we may bound  
	\begin{align*}
		\left| \dashint_{P_{r}'} m_0 \,\mathrm{d}x'  \right| 
		&\leq \frac{|\partial P_r'|}{|P_r'|} \sup_{P_r'} |H^B| + \frac{|\partial P_r'|}{|P_r'|} T \sup_{P_r'} |\nabla' v_0|
		\lesssim \frac{L}{r} \left( \frac{1}{L} \sup_{Q_L'} |H^B| + \frac{T}{L} \sup_{Q_L'} \left| \dashint_0^T h\,\mathrm{d}x_{d+1} \right| \right) 
		\stackrel{\eqref{eq:boundary-layer-size}}{\lesssim} \frac{1}{C},
	\end{align*}
	where in the last step we also used that $-\nabla'v_0= \overline{h}_T = \dashint_0^T h\,\mathrm{d}x_{d+1}$. In particular, $\left| \dashint_{P_{r}'} m_0 \,\mathrm{d}x'  \right| \leq \frac{1}{2}$ if $C$ is chosen large enough.
\end{remark}

The proof of Proposition~\ref{prop:fine-properties-relaxed}, inspired by \cite{Vie09}, proceeds via localisation of the problem in small boundary plaquettes $P_{r}'$ of size $r$ and a splitting of the boundary flux $f \coloneq g-\overline{g}_T$ into two components:
\begin{enumerate}
	\item the oscillatory part of the boundary data $f - \dashint_{\partial P_{r}' \cap \partial Q_L'} f$, giving rise to a divergence-free field $h_1$ (carrying no ``charges'', therefore not influencing the magnetisation), and
	\item the average boundary flux $\dashint_{\partial P_{r}' \cap \partial Q_L'} f$, giving rise to a \emph{bounded} field $h_2$ compatible with the magnetisation $m_{r}$.
\end{enumerate}

\begin{remark} Here one would be tempted to do a standard boundary layer construction based on screening properties, like, for instance, the one performed in \cite{GO20}*{Lemma 2.4} for optimal transportation. Nevertheless, such a strategy has a major problem in our situation: one needs that $m_0 + m_{\text{boundary-layer}} \in [-1,1]$, which is a hard constraint that is not accounted for in such a type of construction. We overcome this difficulty by the passage to a convex combination, which essentially converts boundedness by one in $L^{\infty}$ to requiring smallness in a weaker topology, see Remark~\ref{rem:bound-mb-cumulated-field}.
\end{remark}

\begin{proof}
The construction is done locally in each boundary plaquette $P_{r}'$. We first take out the oscillatory part with a divergence-free field $h_1 = -\nabla' v_1$ by solving the elliptic PDE 
\begin{align*}
	\begin{array}{rcll}
	-\Delta' v_1 &=& 0 & \text{in } P_{r}' \times \{x_{d+1}\} \\
	-\nabla' v_1 \cdot \nu &=& f - \dashint_{\partial P_{r}' \cap \partial Q_L'} f(\cdot, x_{d+1})  & \text{on } (\partial P_{r}' \cap \partial Q_{L}') \times \{x_{d+1}\}\\
	-\nabla' v_1 \cdot \nu &=& 0 & \text{on } (\partial P_{r}' \setminus \partial Q_{L}') \times \{x_{d+1}\}.
	\end{array}
\end{align*}
Note that the solvability condition is fulfilled and by elliptic regularity we have the bound\footnote{It is important in the following estimates on the $L^2$ norm of the fields to keep the r.h.s.\ localised to a plaquette, because we have to sum over all boundary plaquettes in the end, in order to not pick up a term that cancels the small factor in the final estimate.}
\begin{align*}
	\int_{P_{r}'\times\{x_{d+1}\}} |h_1|^2  \lesssim r \int_{\partial P_{r}' \cap \partial Q_L'} \left( f(\cdot, x_{d+1}) - \dashint_{\partial P_{r}' \cap \partial Q_L'} f(\cdot, x_{d+1})  \right)^2  \lesssim r \int_{\partial P_{r}' \cap \partial Q_L'} f^2(\cdot, x_{d+1}).
\end{align*}
In particular, summing over all the boundary plaquettes and integrating over $[0,T]$, we get
\begin{align}\label{eq:bound-h1}
	\frac{1}{L^d T} \int_{A_{L,T}(r)} |h_1|^2\,\mathrm{d}x 
	= \frac{1}{L^d T} \int_0^T \sum_{P_{r}'}  \int_{P_{r}'} |h_1|^2\,\mathrm{d}x' \,\mathrm{d}x_{d+1} 
	\lesssim \frac{r}{L} \, \dashint_{\Gamma_{L,T}} f^2.
\end{align}
 
Recall the linearly interpolated magnetisation $m_0(\cdot, x_{d+1}) = \frac{x_{d+1}}{T} m^T + \left(1-\frac{x_{d+1}}{T} \right) m^B$.
For the averaged boundary flux we make the ansatz that $m_{r} = \lambda (\overline{m}- m_0)$ for some well-chosen $\overline{m}$ taking values in $[-1,1]$, and we allow $\lambda \equiv \lambda_{P_{r}'}(x_{d+1}) \in[0,\frac{1}{2}]$ to depend on $x_{d+1}$ in the given plaquette $P_{r}'$. Then ${m}_{r}(x) \in [-1,1]$ for all $x\in P_{r}'\times [0,T]$, and therefore 
\begin{align*}
	m_0 + m_r = \lambda \overline{m} + (1-\lambda)m_0 \in [-1,1] \quad \text{on} \quad Q_{L,T}.  	
\end{align*}
This magnetisation gives rise to a field $h_2 = -\nabla' v_2$, which is defined slice-wise (for fixed $x_{d+1}$) as the solution to
\begin{align*}
	-\Delta' v_2 &= -\partial_{d+1}m_{r} & & \text{in } P_{r}' \times \{x_{d+1}\} \\
	-\nabla' v_2 \cdot \nu &= \dashint_{\partial P_{r}' \cap \partial Q_L'} f(\cdot, x_{d+1})  & & \text{on } (\partial P_{r}' \cap \partial Q_{L}') \times \{x_{d+1}\}\\
	-\nabla' v_2 \cdot \nu &= 0  & & \text{on } (\partial P_{r}' \setminus \partial Q_{L}') \times \{x_{d+1}\}.
\end{align*}
Note that this elliptic PDE has a solution if 
\begin{align*}
-\int_{\partial P_{r}' \cap \partial Q_L'} f = \int_{P_{r}'} \partial_{d+1} m_{r} = \partial_{d+1} \left[\lambda\left(\int_{P_{r}'}\overline{m}\,\mathrm{d}x'-\int_{P_{r}'}m_0\,\mathrm{d}x'\right)\right],
\end{align*}
which is ensured by the choice
\begin{align*}
	\lambda(x_{d+1}) = \frac{-\frac{1}{|P_{r}'|} \int_0^{x_{d+1}} \int_{\partial P_{r}' \cap \partial Q_L'} f}{\dashint_{P_{r}'} \overline{m}\,\mathrm{d}x' - \dashint_{P_{r}'} m_0\,\mathrm{d}x'}.
\end{align*}
Note that $\lambda(0) = \lambda(T) = 0$ since $\int_0^T f\,\mathrm{d}x_{d+1} = 0$ for almost every $x' \in \partial Q_L'$. Since we need to ensure that $\lambda \in [0,\frac{1}{2}]$, we select $\overline{m}\in[-1,1]$ in such a way that for a given sign of the enumerator the denominator is the largest possible, i.e.\footnote{In fact, we have that $\overline{m} = \argmin\left\{\frac{-\frac{1}{|P_{r}'|} \int_0^{x_{d+1}} \int_{\partial P_{r}' \cap \partial Q_L'} f}{\dashint_{P_{r}'} \widetilde{m}\,\mathrm{d}x' - \dashint_{P_{r}'} m_0\,\mathrm{d}x'}: \widetilde{m}\in[-1,1] \text{ such that } \lambda \geq 0\right\}$.}
\begin{align*}
	\overline{m}(x) \coloneq \sgn\left(- \frac{1}{|P_{r}'|}\int_0^{x_{d+1}} \int_{\partial P_{r}' \cap \partial Q_L'} f \right).
\end{align*}
Note that $\overline{m}$ only depends on $x_{d+1}$ in $P_r'\times[0,T]$.  
Indeed, with this choice, by Remark~\ref{rem:bound-mb-cumulated-field}, we can estimate 
\begin{align*}
	\left|\dashint_{P_{r}'} \overline{m}\,\mathrm{d}x' - \dashint_{P_{r}'} m_0\,\mathrm{d}x' \right| \geq 1 - \sup_{P_{r}'}\left|\dashint_{P_{r}'} m_0\,\mathrm{d}x'\right| \geq \frac{1}{2},
\end{align*}
so that $\lambda\leq \frac{1}{2}$ if $\left|-\frac{1}{|P_{r}'|} \int_0^{x_{d+1}} \int_{\partial P_{r}' \cap \partial Q_L'} f\right|\leq \frac{1}{4}$. We next argue that this is implied by our assumption \eqref{eq:boundary-layer-size} on the size of the boundary layer. Indeed, 
\begin{align*}
	\left|-\frac{1}{|P_{r}'|} \int_0^{x_{d+1}} \int_{\partial P_{r}' \cap \partial Q_L'} f\right| 
	&\lesssim \frac{|\partial P_{r}' \cap \partial Q_{L}'|^{\frac{1}{2}}}{|P_{r}'|} T^{\frac{1}{2}} \left( \int_0^{x_{d+1}} \int_{\partial P_{r}' \cap \partial Q_{L}'} f^2 \right)^{\frac{1}{2}} \\
	&\lesssim \left(\frac{r}{L}\right)^{-\frac{d+1}{2}}\left(\frac{T^2}{L^2} \dashint_{\Gamma_{L,T}} f^2 \right)^{\frac{1}{2}} 
	\stackrel{\eqref{eq:boundary-layer-size}}{\leq} \frac{1}{4},
\end{align*}
if the constant $C$ in \eqref{eq:boundary-layer-size} is chosen large enough.
Notice that we have the finer estimate
\begin{align}\label{eq:lambda-bound}
	\lambda \lesssim \left(\frac{r}{L}\right)^{-\frac{d+1}{2}}\left(\frac{T^2}{L^2} \frac{1}{L^{d-1} T} \int_0^T \int_{\partial P_{r}' \cap \partial Q_{L}'} f^2 \right)^{\frac{1}{2}},
\end{align}
which will be important when summing up the estimates over all plaquettes. 

Next, we show that $\partial_{d+1} m_{r} \in L^{\infty}$ and satisfies the estimate \eqref{eq:bound-d3mell}. To this end, notice that 
\begin{align*}
\partial_{d+1} m_{r} = \partial_{d+1}( \lambda (\overline{m} - m_0)) = \lambda' (\overline{m}-m_0) + \lambda (\partial_{d+1} \overline{m} - \partial_{d+1} m_0) = 	\lambda' (\overline{m}-m_0) - \lambda \partial_{d+1} m_0,
\end{align*}
where we used that $\lambda$ vanishes where $\overline{m}$ changes sign, so that $\lambda \partial_{d+1} \overline{m} = 0$.\footnote{This argument is a bit formal, since strictly speaking $\overline{m}$ is only a function of bounded variation. However, it can be made rigorous with little effort.}

Using that $|m^{B,T}|\leq 1$, the latter term is easily bounded by 
\begin{align*}
	|\lambda \partial_{d+1} m_0| \leq \lambda \left|\frac{m^T - m^B}{T} \right| \lesssim \frac{\lambda}{T} \stackrel{\eqref{eq:lambda-bound}}{\lesssim} \frac{1}{T}  \left(\frac{r}{L}\right)^{-\frac{d+1}{2}}\left(\frac{T^2}{L^2} \frac{1}{L^{d-1} T} \int_0^T \int_{\partial P_{r}' \cap \partial Q_{L}'} f^2 \right)^{\frac{1}{2}}.
\end{align*}
For the other term, we calculate
\begin{align*}
\lambda'(x_{d+1}) = \frac{- \frac{1}{|P_{r}'|} \int_{\partial P_{r}' \cap \partial Q_L'} f(\cdot, x_{d+1})}{\dashint_{P_{r}'} \overline{m}\,\mathrm{d}x' - \dashint_{P_{r}'} m_0\,\mathrm{d}x'} - \lambda \frac{\dashint_{P_{r}'} \partial_{d+1} m_0\,\mathrm{d}x'}{\dashint_{P_{r}'} \overline{m}\,\mathrm{d}x' - \dashint_{P_{r}'} m_0\,\mathrm{d}x'}.
\end{align*}
In particular, bounding the first term in a similar way as $\lambda$ in \eqref{eq:lambda-bound} (without the $x_{d+1}$-integral) and the second term as before, we obtain 
\begin{align*}
	|\lambda'(x_{d+1})| 
	\lesssim \left(\frac{r}{L}\right)^{-\frac{d+1}{2}}\left(\frac{1}{L^{d+1}} \int_{\partial P_{r}' \cap \partial Q_{L}'} f^2(\cdot, x_{d+1}) \right)^{\frac{1}{2}}  
	+  \left(\frac{r}{L}\right)^{-\frac{d+1}{2}}\left(\frac{1}{L^{d+1} T} \int_0^T \int_{\partial P_{r}' \cap \partial Q_{L}'} f^2 \right)^{\frac{1}{2}}.
\end{align*}
and so, since $\|\overline{m}(\cdot, x_{d+1}) - m_0(\cdot, x_{d+1})\|_{L^{\infty}(Q_{L}')} \leq 2$, it follows that 
\begin{align*}
	\sup_{P_{r}'} |\partial_{d+1}m_{r}(\cdot, x_{d+1})|^2 
	&\lesssim \left(\frac{r}{L}\right)^{-(d+1)} \frac{1}{L^{d+1}} \int_{\partial P_{r}' \cap \partial Q_{L}'} f^2(\cdot, x_{d+1}) \\
	&\quad +  \left(\frac{r}{L}\right)^{-(d+1)} \frac{1}{L^{d+1} T} \int_0^T \int_{\partial P_{r}' \cap \partial Q_{L}'} f^2.
\end{align*}
By standard elliptic regularity\footnote{Indeed, this follows from maximal $L^p$ regularity for $v_2$ using that $\partial_{d+1} m_{r}(\cdot, x_{d+1}) \in L^p(P_{r}')$ for a.e.\ $x_{d+1} \in [0,T]$ and $p\geq 1$, combined with the Sobolev embedding $h_2(\cdot, x_{d+1}) =-\nabla'v_2 (\cdot, x_{d+1}) \in W^{1,p}(P_{r}') \hookrightarrow L^{\infty}(P_{r}')$.}, we then have the following estimate for $h_2=-\nabla'v_2$:
\begin{align*}
	\sup_{P_{r}'} |h_2(\cdot, x_{d+1})|^2 
	&\lesssim r^2 \sup_{P_{r}'} |\partial_{d+1}m_{r}(\cdot, x_{d+1})|^2  + \dashint_{\partial P_{r}' \cap \partial Q_L'} f^2(\cdot, x_{d+1}) \\
	&\lesssim  \left(\frac{r}{L}\right)^{-(d-1)} \frac{1}{L^{d-1}} \int_{\partial P_{r}' \cap \partial Q_{L}'} f^2(\cdot, x_{d+1}) 
		+  \left(\frac{r}{L}\right)^{-(d-1)} \frac{1}{L^{d-1} T} \int_0^T \int_{\partial P_{r}' \cap \partial Q_{L}'} f^2.
\end{align*}
In particular, we get that 
\begin{align*}
	\frac{1}{L^d T} \int_0^T \int_{P_{r}'} |h_2|^2\,\mathrm{d}x 
	\lesssim \frac{1}{T}\int_0^T \left(\frac{r}{L}\right)^d \sup_{P_{r}'} |h_2|^2 \,\mathrm{d}x_{d+1}
	\lesssim \frac{r}{L} \frac{1}{L^{d-1}T} \int_0^T \int_{\partial P_{r}' \cap \partial Q_{L}'} f^2.
\end{align*}
Summing over all the boundary plaquettes we then obtain the bound
\begin{align}\label{eq:bound-h2}
	\frac{1}{L^d T} \int_{A_{L,T}(r)} |h_2|^2\,\mathrm{d}x 
	= \frac{1}{L^d T} \int_0^T \sum_{P_{r}'} \int_{P_{r}'} |h_2|^2\,\mathrm{d}x 
	\lesssim \frac{r}{L} \dashint_{\Gamma_{L,T}} f^2.
\end{align}

Putting the two contributions (i.e. $h_1$ and $h_2$) together, we have constructed an admissible pair $(m_{r}, h_{r}) \in \mathcal{X}_{r}(L,T)$, where $h_{r} = h_1 + h_2$, which by \eqref{eq:bound-h1} and \eqref{eq:bound-h2} satisfies the estimate
\begin{align*}
	\frac{1}{L^d T} \int_{A_{L,T}(r)} |h_{r}|^2\,\mathrm{d}x \lesssim \frac{r}{L} \dashint_{\Gamma_{L,T}} f^2.
\end{align*}
\end{proof}
\section{Construction: From relaxed to non-convex}\label{sec:constructionnonconvex}
In this section we construct a competitor for the non-convex energy \eqref{eq:energy-nonconvex} from a competitor for the relaxed problem. 
This construction will play an essential role in the proofs of our main results, as we shall see in Section~\ref{sec:proofglobalscalinglaw} and Section~\ref{sec:prooflocalscalinglaws}.
We start with a building block for the full branching construction.
\begin{lemma}[Building block for the construction]\label{lem:buildingblock}
Let $Q = [0,1]^{d+1}$. Given an admissible pair $(M,H) \in L^{\infty}(Q; [-1,1]) \times L^2(Q; \RR^d)$ such that $\partial_{d+1} M + \nabla'\cdot H = 0$ and 
$\|\partial_{d+1} M\|_{L^2_{y_{d+1}}L^{\infty}_{y'}(Q)} \lesssim 1$, there exist functions $\widetilde{M}\in L^{1}_{y_{d+1}}([0,1]; BV_{y'}([0,1]^d; \pm 1))$ and $\widetilde{H} \in L^2(Q; \RR^d)$ such that $\partial_{d+1} \widetilde{M} + \nabla'\cdot\widetilde{H} = 0$ in $Q$ and $\widetilde{H}\cdot \nu' = H\cdot\nu'$ on $\Gamma = \partial[0,1]^d \times [0,1]$, and satisfying the estimates 
\begin{align}
	\int_Q |\nabla' \widetilde{M}| &\lesssim 1, \label{eq:bound-mtilde} \\
	\int_Q |\widetilde{H}-H|^2\,\mathrm{d}y &\lesssim 1 \label{eq:bound-htilde}.
\end{align}
\end{lemma}

\begin{remark}
	Since the construction for $d\geq 2$ will be done in such a way that it is constant in the directions $y_3, \dots, y_d$, for simplicity and to not overburden the reader with notation, we will give the proof just for $d=2$. The generalisation to higher dimensions is then immediate. 
	
	Moreover, the proof can be adapted to the case $d=1$ (and simplified considerably), thus recovering a construction similar to \cite{Con00}.
\end{remark}

\begin{proof}[Proof for $d=2$]The construction of a $\pm1$-valued function $\widetilde{M}$ from $M$ is done in two steps: a \emph{geometric refinement} in the lower half of the cube and a \emph{local averaging} with respect to the refined subdomains. It is worth remarking that the geometric refinement will be crucial when constructing a competitor for the non-convex energy, in view of the weak boundary conditions that the magnetisation needs to satisfy on the top and bottom boundaries of the domain sample.

\begin{enumerate}[label=\textsc{\bf Step \arabic*},leftmargin=0pt,labelsep=*,itemindent=*,itemsep=10pt,topsep=10pt]
	\item 	\emph{(Averaging)}.
			Let $S_{ij} = \left[0,\frac{1}{2}\right]^2 + \left(\frac{i}{2},\frac{j}{2}\right)$ for $i,j \in\{0,1\}$, and define 
			\begin{align*}
				\overline{M}(y) \coloneqq \begin{cases}
 									\int_{[0,1]^2} M(x',y_3)\,\mathrm{d}x' & \text{for } y_3 \in\left[\frac{1}{2},1\right], \\	
 									(1-2y_3) \dashint_{S_{ij}} M(x',y_3)\,\mathrm{d}x' + 2y_3 \int_{[0,1]^2} M(x',y_3)\,\mathrm{d}x' & \text{for } y_3 \in\left[0,\frac{1}{2}\right) \text{ and } y' \in S_{ij}.
								  \end{cases}
			\end{align*}

			Let us observe that our construction refines from a single plaquette for $y_3 \in \left[\frac{1}{2},1\right]$ to four identical plaquettes for $y_3\in\left[0,\frac{1}{2}\right)$. The main idea here is that for $y_3\in \{0,1\}$, in each plaquette $\overline{M}$ is equal to the average of the magnetisation over the corresponding plaquette. 
			In the upper half of the cube, i.e. for $y_3\in\left[\frac12,1\right)$, since there is only one plaquette, we define $\overline{M}$ as the average of the magnetisation over the corresponding plaquette. 
			Then, in the lower half of the cube, i.e. for $y_3\in\left(0,\frac12\right)$, on each plaquette at height $y_3$ we take a convex combination between $\dashint_{S_{i,j}}M(y',y_3)\, \mathrm{d}y'$ and $\int_{[0,1]^2}M(y',y_3)\, \mathrm{d}y'$, where the coefficients of the convex combination are taken to be linear with respect to $y_3$ and so that 
			$\overline{M}$ does not jump at $y_3=\frac12$. 

			We remark that the constructed $\overline M$ is such that $\overline M\in [-1,1]$ in $Q$ and $\int_{[0,1]^2}M(y',y_3)\ \mathrm{d}y'=\int_{[0,1]^2}\overline{M}(y',y_3)\ \mathrm{d}y'$ for every $y_3\in[0,1]$, i.e. the average per $y_3$-slice is preserved. 
			
	\item 	\emph{(Construction of a $\pm1$-valued magnetisation)}. 
			With the locally averaged $\overline{M}$ we can now proceed to define a $\pm 1$-valued magnetisation $\widetilde M$ with the property that in each horizontal plaquette $S_{ij} \times \{y_3\}$ there is exactly one interface curve, 
			by shifting the magnetisation according to the volume fraction in each horizontal slice.
			
			For $y_3 \in \left[\frac{1}{2},1\right]$, we set 
			\begin{align*}
				\widetilde{M}(y) \coloneqq \begin{cases}
 									+ 1 & \text{if } y_1 \in \left[0, \eta_1(y_3) \right] \cup \left[ \frac{1}{2}, \eta_2(y_3)\right], \\
 									-1 & \text{else},
 								\end{cases}
			\end{align*}
			where
			\begin{align}\label{eq:eta-12}
				\eta_1(y_3) \coloneqq \min\left(y_3 \frac{1+\overline{M}(y)}{2}, \frac{1}{2}\right) \quad \text{and} \quad \eta_2(y_3) \coloneqq \frac{1}{2} + \max\left( (1-y_3)\frac{1+ \overline{M}(y)}{2}, \frac{\overline{M}(y)}{2}\right),
			\end{align}
			and we recall that $\overline{M}(y) = \int_{[0,1]^2} M(y)\,\mathrm{d}y' \in [-1,1]$ depends only on $y_3$. 
		
			For $y_3 \in \left[0, \frac{1}{2}\right)$ the local averaging is refined 
			to ensure that the construction is compatible with the top boundary conditions of four cubes with corresponding volume fractions: we define 
			\begin{align*}
				\widetilde{M}(y) \coloneqq \begin{cases}
 									+ 1 & \text{if } y' \in S_{ij} \text{ and } y_1 \in \left[\frac{i}{2}, \frac{i}{2} + \frac{1}{2} \frac{1+\overline{M}(y)}{2} \right], \\
 									-1 & \text{elsewhere on } S_{ij}.
 								\end{cases}
			\end{align*}
			Note that the magnetisation $\widetilde{M}$ is piecewise constant in $y_2$.
			
			In words, the main idea here is that for $y_3\in \{0,1\}$, all the ``positive'' mass of $M$ in each plaquette, that is the Lebesgue measure of the set where $M(\cdot,y_3)=+1$ in the corresponding plaquette, which is equal to $\frac{1+\overline M(y)}{2}$, is shifted to the left in the $y_1$-direction, which preserves the average. 
			
			In the lower half of the cube, we do exactly the same. However, recall that for $y_3\in\left(0,\frac12\right)$, $\overline M$ is not longer equal to the average of $M$ over the corresponding plaquette, but a convex combination between the average over each plaquette and the average over the four of them. This in particular ensures that for $y_3=\frac12$ the ``positive'' mass in $[0,1]^2\times \left\{\frac12\right\}$ is equally distributed (and shifted to the left) in the 2 plaquettes $\left[0,\frac12\right]\times [0,1]\times \left\{\frac12\right\}$ and $\left[\frac12,1\right]\times[0,1]\times \left\{\frac12\right\}$. Hence, $\widetilde M$ is constant in the $y_2$-direction at $y_3=\frac{1}{2}$.

			For $y_3\in\left(\frac12,1\right)$, the ``positive'' mass in $[0,1]^2\times \{y_3\}$ is unequally distributed (and shifted to the left) in the 2 plaquettes $\left[0,\frac12\right]\times [0,1]\times \{y_3\}$ and $\left[\frac12,1\right]\times[0,1]\times \{y_3\}$. More precisely, we put a fraction $y_3$ in the first plaquette, and the remaining $(1-y_3)$ fraction in the latter. A key observation is that when one reaches $y_3=1$, then $(1-y_3)$ vanishes and therefore all the ``positive'' mass is shifted to the left in the single plaquette $[0,1]^2\times \{1\}$.

			For the reader's convenience, in Figure~\ref{fig:example1-construction} and Figure~\ref{fig:example2-construction} we depict how $\widetilde M$ looks like in two simple situations.

			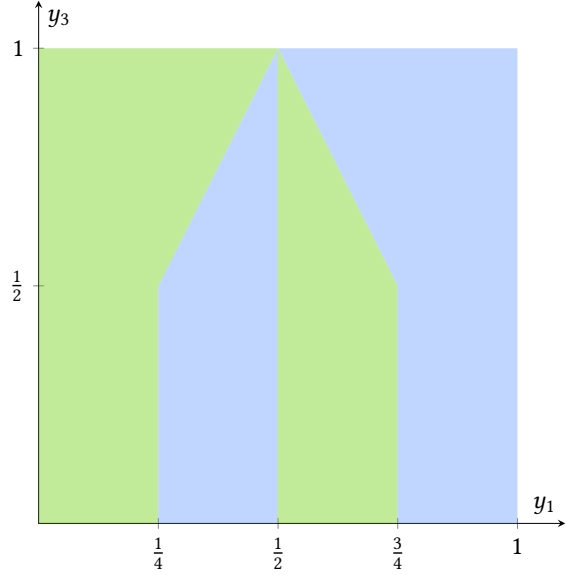
\begin{figure}[ht]
			\centering
			\scalebox{0.9}{\definecolor{ffffww}{rgb}{0.4,0.6,1.}
\definecolor{ffqqqq}{rgb}{0.4,0.8,0.}
\begin{tikzpicture}[line cap=round,line join=round,>=triangle 45,x=1cm,y=1cm]
\begin{axis}[
x=7cm,y=7cm,
axis lines=middle,
xlabel = \(y_1\),
ylabel = {\(y_3\)},
xmin=0,
xmax=1.1,
ymin=0,
ymax=1.1,
xtick={0,1/4,...,1},
xticklabels={0,$\frac14$,$\frac12$,$\frac34$,$1$},
ytick={0,0.5,...,1},
yticklabels={0,$\frac12$,$1$}]
\clip(-0.1,-0.1) rectangle (1.1,1.1);
\fill[line width=0pt,color=ffqqqq,fill=ffqqqq,fill opacity=0.4] (0,0) -- (0.25,0) -- (0.25,0.5) -- (0.5,1) -- (0,1) -- cycle;
\fill[line width=0pt,color=ffffww,fill=ffffww,fill opacity=0.4] (0.25,0) -- (0.5,0) -- (0.5,1) -- (0.25,0.5) -- cycle;
\fill[line width=0pt,color=ffqqqq,fill=ffqqqq,fill opacity=0.4] (0.5,0) -- (0.75,0) -- (0.75,0.5) -- (0.5,1) -- cycle;
\fill[line width=0pt,color=ffffww,fill=ffffww,fill opacity=0.4] (0.75,0) -- (1,0) -- (1,1) -- (0.5,1) -- (0.75,0.5) -- cycle;
\end{axis}
\end{tikzpicture}}
			\caption{Given $y_2\in[0,1]$, this depicts $\widetilde M(\cdot,y_2,\cdot)$ in the case when $\int_{[0,1]^2}M\,\mathrm{d}y'=0$ and $\overline{M}=0$ for every $y_3\in[0,1]$. The set $\{\widetilde M(\cdot,y_2,\cdot)=+1\}$ is coloured in green, whereas $\{\widetilde M(\cdot,y_2,\cdot)=-1\}$ is coloured in blue. Notice that in this special case $\widetilde{M}$ is constant in $y_2$ direction.}
			\label{fig:example1-construction}
			\end{figure}

			\begin{figure}
			\centering
			\begin{minipage}{.5\textwidth}
  			\centering
 			\scalebox{0.8}{\definecolor{ffffww}{rgb}{0.4,0.6,1.}
\definecolor{ffqqqq}{rgb}{0.4,0.8,0.}
\begin{tikzpicture}[line cap=round,line join=round,>=triangle 45,x=1cm,y=1cm]
\begin{axis}[
x=7cm,y=7cm,
axis lines=middle,
xlabel = \(y_1\),
ylabel = {\(y_3\)},
xmin=0,
xmax=1.1,
ymin=0,
ymax=1.1,
xtick={0,1/8,...,1},
xticklabels={0,$\frac18$,$\frac14$,$\frac38$,$\frac12$,$\frac58$,$\frac34$,$\frac78$,$1$},
ytick={0,0.5,...,1},
yticklabels={0,$\frac12$,$1$}]
\clip(-0.1,-0.1) rectangle (1.1,1.1);
\fill[line width=0pt,color=ffqqqq,fill=ffqqqq,fill opacity=0.4] (0,0) -- (0.5,0) -- (0.3125,0.5) -- (0.5,1) -- (0,1) -- cycle;
\fill[line width=0pt,color=ffffww,fill=ffffww,fill opacity=0.4] (0.5,0) -- (0.5,1) -- (0.3125,0.5) -- cycle;
\fill[line width=0pt,color=ffqqqq,fill=ffqqqq,fill opacity=0.4] (0.5,0) -- (0.6875,0) -- (0.8125,0.5) -- (0.625,1) -- (0.5,1) -- cycle;
\fill[line width=0pt,color=ffffww,fill=ffffww,fill opacity=0.4] (0.6875,0) -- (1,0) -- (1,1) -- (0.625,1) -- (0.8125,0.5) -- cycle;
\draw [dashed, opacity=0.7,line width=0.2mm] (5/16,0) -- (5/16,1/2);
\draw [dashed, opacity=0.7,line width=0.2mm] (5/8,0) -- (5/8,1);
\draw [dashed, opacity=0.7,line width=0.2mm] (13/16,0) -- (13/16,1/2);
\end{axis}
\end{tikzpicture}}
			\end{minipage}%
			\begin{minipage}{.5\textwidth}
  			\centering
			\scalebox{0.8}{\definecolor{ffffww}{rgb}{0.4,0.6,1.}
\definecolor{ffqqqq}{rgb}{0.4,0.8,0.}
\begin{tikzpicture}[line cap=round,line join=round,>=triangle 45,x=1cm,y=1cm]
\begin{axis}[
x=7cm,y=7cm,
axis lines=middle,
xlabel = \(y_1\),
ylabel = {\(y_3\)},
xmin=0,
xmax=1.1,
ymin=0,
ymax=1.1,
xtick={0,1/8,...,1},
xticklabels={0,$\frac18$,$\frac14$,$\frac38$,$\frac12$,$\frac58$,$\frac34$,$\frac78$,$1$},
ytick={0,0.5,...,1},
yticklabels={0,$\frac12$,$1$}]
\clip(-0.1,-0.1) rectangle (1.1,1.1);
\fill[line width=0pt,color=ffqqqq,fill=ffqqqq,fill opacity=0.4] (0.1875,0) -- (0.3125,0.5) -- (0.5,1) -- (0,1) -- (0,0) -- cycle;
\fill[line width=0pt,color=ffqqqq,fill=ffqqqq,fill opacity=0.4] (0.5,0) -- (0.875,0) -- (0.8125,0.5) -- (0.625,1) -- (0.5,1) -- cycle;
\fill[line width=0pt,color=ffffww,fill=ffffww,fill opacity=0.4] (0.1875,0) -- (0.5,0) -- (0.5,1) -- (0.3125,0.5) -- cycle;
\fill[line width=0pt,color=ffffww,fill=ffffww,fill opacity=0.4] (0.875,0) -- (1,0) -- (1,1) -- (0.625,1) -- (0.8125,0.5) -- cycle;
\draw [dashed, opacity=0.7,line width=0.2mm] (5/16,0) -- (5/16,1/2);
\draw [dashed, opacity=0.7,line width=0.2mm] (5/8,0) -- (5/8,1);
\draw [dashed, opacity=0.7,line width=0.2mm] (13/16,0) -- (13/16,1/2);
\end{axis}
\end{tikzpicture}}
			\end{minipage}
			\caption{The picture on the left depicts $\widetilde M(\cdot,y_2,\cdot)$ for $y_2\in\left[0,\frac12\right]$ and the one on the right for $y_2\in\left[\frac12,1\right]$, in the case  $\int_{[0,1]^2}M\,\mathrm{d}y'=\frac14$ for every $y_3\in[0,1]$, and $\dashint_{S_{0,0}}M\,\mathrm{d}y'=1$, $\dashint_{S_{1,0}}M\,\mathrm{d}y'=\dashint_{S_{0,1}}M\,\mathrm{d}y'=-\frac14$, and $\dashint_{S_{1,1}}M\,\mathrm{d}y'=\frac12$ for every $y_3\in\left[0,\frac12\right]$. The set $\{\widetilde M(\cdot,y_2,\cdot)=+1\}$ is coloured in green, whereas $\{\widetilde M(\cdot,y_2,\cdot)=-1\}$ is coloured in blue. Observe that $\widetilde M$ is not constant in $y_2$ direction for $y_3\in \left[0,\frac12\right)$.}
			\label{fig:example2-construction}
			\end{figure}
			
			We remark that the constructed $\widetilde M$ is such that $\widetilde M\in [-1,1]$ in $Q$ and $\int_{[0,1]^2}\widetilde M(y',y_3)\ \mathrm{d}y'=\int_{[0,1]^2}\overline{M}(y',y_3)\ \mathrm{d}y'=\int_{[0,1]^2} M(y',y_3)\ \mathrm{d}y'$ for every $y_3\in[0,1]$, i.e. the average per slice in $y_3$ direction is preserved.				 
			\begin{remark}\label{rem:continuity}
				The refinement is done in such a way that the constructed $\pm1$-valued magnetisation in Corollary~\ref{cor:competitor} is continuous in the vertical direction.
			\end{remark}
			Since this construction creates at most 3 interfaces for $y_3 \in \left[\frac{1}{2},1\right]$ and at most $8$ interfaces for $y_3 \in \left[0,\frac{1}{2}\right)$ (at most 6 in $y_1$ direction and at most 2 in $y_2$ direction), the total surface energy created is bounded by 
			\begin{align*}
				\int_{Q} |\nabla'\widetilde{M}| \lesssim 1.
			\end{align*}			
				
	\item	\emph{(Construction of corrector fields)}.
			Since the local rearrangement of the magnetisation creates ``excess charges'' $\partial_3 (\widetilde{M}-M)$, we have to define a corrector field $H_{\cor}$ such that 
			\begin{align}\label{eq:corrector-field}
			\begin{split}
				\partial_3(\widetilde{M} - M) + \nabla'\cdot H_{\cor} = 0 \quad \text{in} \quad Q,\\
				H_{\cor}\cdot \nu' = 0 \quad \text{on} \quad \Gamma,
			\end{split}
			\end{align}
			in order to obtain a competitor for the non-convex problem satisfying the right boundary conditions (since the corrector field satisfies Neumann boundary conditions). 
				
			\medskip We treat the cases $y_3 \in \left[0,\frac{1}{2}\right)$ and $y_3 \in \left[\frac{1}{2},1\right]$ separately. 

			\begin{enumerate}[label=\textsc{\bf Case \arabic*.},leftmargin=0pt,labelsep=*,itemindent=*,itemsep=10pt,topsep=10pt]
				\item For $y_3 \in \left[\frac{1}{2},1\right]$ we have 
					\begin{align*}
						\partial_3 \widetilde{M} = 2 \eta_1'(y_3) \mathcal{H}^1\lfloor_{\{y_1 = \eta_1(y_3)\}} + 2 \eta_2'(y_3) \mathcal{H}^1\lfloor_{\{y_1 = \eta_2(y_3)\}},
					\end{align*}
					with $\eta_{1,2}$ defined in \eqref{eq:eta-12}. There holds
					\begin{align*}
						\eta_1'(y_3) &= \1_{\left\{y_3 \frac{1+\overline{M}(y)}{2} \leq \frac{1}{2}\right\}} \left( \frac{1+\overline{M}(y)}{2} + \frac{y_3}{2} \partial_3 \overline{M}(y)\right), \\
						\eta_2'(y_3) &= \1_{\left\{(1-y_3) \frac{1+\overline{M}(y)}{2} \geq \frac{\overline{M}(y)}{2}\right\}} \left( -\frac{1+\overline{M}(y)}{2} + \frac{1-y_3}{2} \partial_3 \overline{M}(y)\right) + \1_{\left\{(1-y_3) \frac{1+\overline{M}(y)}{2} < \frac{\overline{M}(y)}{2}\right\}} \frac{\partial_3 \overline{M}(y)}{2} \\
						&= \1_{\left\{y_3 \frac{1+\overline{M}(y)}{2} \leq \frac{1}{2}\right\}} \left( - \frac{1+\overline{M}(y)}{2} - \frac{y_3}{2} \partial_3 \overline{M}(y) \right) + \frac{\partial_3 \overline{M}}{2},
					\end{align*}
					and we construct $H_{\cor}$ as superposition of (slice-wise defined) corrector fields corresponding to the decomposition\footnote{Notice that $\overline{M}$ does not depend on $y'$ for $y_3\in[\frac{1}{2},1]$.} 
					\begin{align*}
						\partial_3(\widetilde{M} - M)
						&= \1_{\left\{y_3 \frac{1+\overline{M}(y)}{2} \leq \frac{1}{2}\right\}} \left( (1+\overline{M})  \mathcal{H}^1\lfloor_{\{y_1 = \eta_1(y_3)\}} - (1+\overline{M})  \mathcal{H}^1\lfloor_{\{y_1 = \eta_2(y_3)\}} \right) \\
						&\quad+ \1_{\left\{y_3 \frac{1+\overline{M}(y)}{2} \leq \frac{1}{2}\right\}} \left( y_3 (\partial_3 \overline{M}) \mathcal{H}^1\lfloor_{\{y_1 = \eta_1(y_3)\}} - y_3 (\partial_3 \overline{M}) \mathcal{H}^1\lfloor_{\{y_1 = \eta_2(y_3)\}} \right) \\
						&\quad+ \left((\partial_3 \overline{M}) \mathcal{H}^1\lfloor_{\{y_1 = \eta_2(y_3)\}} - \partial_3 M \right).
					\end{align*}
 
					\begin{enumerate}[label=\textsc{\bf Corrector 1.\arabic*},leftmargin=0pt,labelsep=*,itemindent=*,itemsep=10pt,topsep=10pt]
						\item For $y_3 \frac{1+\overline{M}(y)}{2} \leq \frac{1}{2}$ we solve 
							\begin{align*}
								-\nabla' \cdot H_{\cor}^{(1)} = (1+\overline{M})  \mathcal{H}^1\lfloor_{\{y_1 = \eta_1(y_3)\}} - (1+\overline{M})  \mathcal{H}^1\lfloor_{\{y_1 = \eta_2(y_3)\}} \quad \text{in } [0,1]^2 \times \{y_3\}
							\end{align*}
							with boundary condition $H_{\cor}^{(1)} \cdot \nu' = 0$ on $\partial[0,1]^2 \times \{y_3\}$. A solution is easily found to be 
							\begin{align*}
								H_{\cor}^{(1)} = -(1+\overline{M}) e_1 \1_{\{y_1\geq \eta_1\}} + (1+\overline{M})e_1 \1_{\{y_1\geq \eta_2\}} = -(1+\overline{M})e_1 \1_{\{\eta_1 \leq y_1\leq \eta_2\}},
							\end{align*}
							where $e_1=(1,0)$ denotes the unit vector in $y_1$ direction. For $y_3 \frac{1+\overline{M}(y)}{2} > \frac{1}{2}$ we set $H_{\cor}^{(1)} = 0$. Note that this field satisfies $\|H_{\cor}^{(1)}\|_{L^{\infty}([0,1]^2\times \{y_3\})} \leq 2$, hence
							\begin{align}\label{eq:corrector-bound-1}
								\|H_{\cor}^{(1)}\|_{L^{2}([0,1]^2 \times [\frac{1}{2},1])} \leq 2.
							\end{align}
						\item For the next correction field we first note that 
							\begin{align*}
								\partial_3 \overline{M}
								= \int_{[0,1]^2} \partial_3 M \,\mathrm{d}y' 
							\end{align*}
							so that as for the first corrector we can set 
							\begin{align*}
								H_{\cor}^{(2)} = -y_3 (\partial_3 \overline{M}) e_1 \1_{\{\eta_1 \leq y_1\leq \eta_2\}}
							\end{align*}
							for $y_3 \frac{1+\overline{M}(y)}{2} \leq \frac{1}{2}$ and $H_{\cor}^{(2)}=0$ otherwise,  which solves 
							\begin{align*}
								-\nabla' \cdot H_{\cor}^{(2)} = y_3 (\partial_3 \overline{M}) \mathcal{H}^1\lfloor_{\{y_1 = \eta_1(y_3)\}} - y_3 (\partial_3 \overline{M}) \mathcal{H}^1\lfloor_{\{y_1 = \eta_2(y_3)\}}
							\end{align*}
							in $[0,1]^2 \times \{y_3\}$ with boundary condition $H_{\cor}^{(1)} \cdot \nu' = 0$ on $\partial[0,1]^2 \times \{y_3\}$ for $y_3 \frac{1+\overline{M}(y)}{2} \leq \frac{1}{2}$. 
							We then have the bound 
							\begin{align}\label{eq:corrector-bound-2}
								\|H_{\cor}^{(2)}\|_{L^{2}\left([0,1]^2\times\left[\frac{1}{2},1\right]\right)}^2 
								&\leq \int_{\frac{1}{2}}^1 \|H_{\cor}^{(2)}\|_{L^{\infty}([0,1]^2\times\{y_3\})}^2 \,\mathrm{d}y_3 
								\leq \int_0^1 \left( \partial_3 \overline{M} \right)^2 \,\mathrm{d}y_3 \nonumber\\
								&\leq \|\partial_3 M\|_{L^2_{y_3} L^{\infty}_{y'}(Q)}^2.
							\end{align}

						\item For the remaining correction we make the ansatz $H_{\cor}^{(3)}= \nabla' v_{\cor}^{(3)}$ where $v_{\cor}^{(3)}$ is a solution of the elliptic PDE 
							\begin{align*}
								\begin{array}{rcll}
								\Delta' v_{\cor}^{(3)} &=& \partial_3 M - (\partial_3 \overline{M}) \mathcal{H}^1\lfloor_{\{y_1 = \eta_2(y_3)\}} & \text{in }
								[0,1]^2 \times \{y_3\}\\
								\nabla'v_{\cor}^{(3)} \cdot \nu' &=& 0  & \text{on } \partial[0,1]^2 \times \{y_3\}
								\end{array}
							\end{align*}
							for  $y_3 \in \left[\frac{1}{2},1\right]$. 
							Note that by Lemma~\ref{lem:elliptic-line} (for $d=2$) this PDE is solvable and the solution satisfies the bound
							\begin{align*}
								\|H_{\cor}^{(3)}\|_{L^2([0,1]^2\times\{y_3\})} \lesssim \|\partial_3 M\|_{L^2([0,1]^2\times\{y_3\})} \lesssim \|\partial_3 M\|_{L^{\infty}([0,1]^2\times\{y_3\})},
							\end{align*}
							hence 
							\begin{align}\label{eq:corrector-bound-3}
								\|H_{\cor}^{(3)}\|_{L^2\left([0,1]^2\times\left[\frac{1}{2},1\right]\right)} \lesssim \|\partial_3 M\|_{L^2_{y_3} L^{\infty}_{y'}(Q)}.
							\end{align}
					\end{enumerate}
				\item For $y_3 \in \left[0,\frac{1}{2}\right)$, where the local averaging is refined, there is a change also in the $y_2$ direction. This is why we do the construction of a comparison field on the four plaquettes $S_{i,j}$, $i,j=0,1$. 
				
				To this end, define 
				 $\eta_{(i,j)}(y_3) = \frac{i}{2} + \frac{1}{2} \frac{1+\overline{M}(y)}{2}$ for $(i,j) \in \{0,1\}$, recalling that 
				 \begin{align*}
					\overline{M}(y) = (1-2y_3) \dashint_{S_{ij}} M(x',y_3)\,\mathrm{d}x' + 2y_3 \int_{[0,1]^2} M(x',y_3)\,\mathrm{d}x',\quad \mbox{for }y' \in S_{ij}.
				\end{align*} 
				Then
					\begin{align*}
						\partial_3 \widetilde{M} 
						&= 2 \eta_{(0,0)}'(y_3) \mathcal{H}^1\lfloor_{\{y_1 = \eta_{(0,0)}(y_3)\}\cap S_{00}}
						+ 2 \eta_{(1,0)}'(y_3) \mathcal{H}^1\lfloor_{\{y_1 = \eta_{(1,0)}(y_3)\}\cap S_{10}}\\
						&\quad + 2 \eta_{(0,1)}'(y_3) \mathcal{H}^1\lfloor_{\{y_1 = \eta_{(0,1)}(y_3)\}\cap S_{01}}
						+ 2 \eta_{(1,1)}'(y_3) \mathcal{H}^1\lfloor_{\{y_1 = \eta_{(1,1)}(y_3)\}\cap S_{11}},
					\end{align*}
					with 
					\begin{align*}
						\eta_{(i,j)}'(y_3) &= 
						\frac{1}{2} \left( \int_{[0,1]^2} M\,\mathrm{d}y' - \dashint_{S_{ij}} M\,\mathrm{d}y'\right) 
						 + \frac{y_3}{2} \left(\int_{[0,1]^2} \partial_3 M\,\mathrm{d}y' - \dashint_{S_{ij}} \partial_3 M\,\mathrm{d}y'\right)+\frac12  \dashint_{S_{ij}} \partial_3 M\,\mathrm{d}y'.
					\end{align*}
					
					\begin{enumerate}[label=\textsc{\bf Corrector 2.\arabic*},leftmargin=0pt,labelsep=*,itemindent=*,topsep=10pt]
						\item We first solve 
							\begin{align}\label{eq:cor-4}
								-\nabla'\cdot H_{\cor}^{(4)} 
								&= \sum_{i,j\in\{0,1\}} \left( \int_{[0,1]^2} M(y',y_3)\,\mathrm{d}y' - \dashint_{S_{ij}} M(y',y_3)\,\mathrm{d}y'\right) \mathcal{H}^1\lfloor_{\{y_1 = \eta_{(i,j)}(y_3)\}\cap S_{ij}}
							\end{align}
							in $[0,1]^2\times\{y_3\}$ with boundary condition $H_{\cor}^{(4)} \cdot \nu'=0$ on $\partial[0,1]^2\times\{y_3\}$. 
							We proceed in two steps: Consider the field 
							\begin{align*}
								H_{\cor}^{(4,\alpha)}(y) 
								&= -\sum_{j=0}^1 \left[ \left( \int_{[0,1]^2} M\,\mathrm{d}y' - \dashint_{S_{0j}} M\,\mathrm{d}y'\right) \1_{\{y_1 \geq \eta_{(0,j)}(y_3)\}\cap S_{0j}} \right.\\
								&\qquad\qquad- \left.\left( \int_{[0,1]^2} M\,\mathrm{d}y' - \dashint_{S_{1j}} M\,\mathrm{d}y'\right) \1_{\{y_1 \leq \eta_{(1,j)}(y_3)\}\cap S_{1j}} \right] e_1,
							\end{align*}
							which has zero normal flux through the boundary. 
							Its horizontal divergence has an extra contribution from surface charges at the interface $\left\{y_1 = \frac{1}{2}\right\}$, given by
							\begin{align*}
								\sum_{j=0}^1 \left(2 \int_{[0,1]^2} M\,\mathrm{d}y' - \dashint_{S_{0,j}}M \,\mathrm{d}y' - \dashint_{S_{1,j}}M \,\mathrm{d}y' \right) \mathcal{H}^1\lfloor_{\left\{y_1=\frac{1}{2}\right\}\cap(S_{0,j}\cup S_{1,j})},
							\end{align*}
							which we correct with the field
							\begin{align*}
								H_{\cor}^{(4,\beta)}(y) = \begin{cases}
																-\left(2 \int_{[0,1]^2} M\,\mathrm{d}y' - \dashint_{S_{0,0}}M \,\mathrm{d}y' - \dashint_{S_{1,0}}M \,\mathrm{d}y' \right) (e_1 + e_2) & \text{in } S_{10} \cap \left\{y_2 \geq y_1-\frac{1}{2}\right\}, \\
																-\left(2 \int_{[0,1]^2} M\,\mathrm{d}y' - \dashint_{S_{0,1}}M \,\mathrm{d}y' - \dashint_{S_{1,1}}M \,\mathrm{d}y' \right) (e_1 - e_2) & \text{in } S_{11} \cap \left\{y_2 \leq \frac{3}{2} - y_1\right\}, \\
																0 & \text{otherwise}.  
															\end{cases}
							\end{align*}
							Note that $H_{\cor}^{(4,\beta)}$ is continuous in $y_2$ across the sector boundary between $S_{10}$ and $S_{11}$ because 
							\begin{align}\label{eq:mass-balance}
								4 \int_{[0,1]^2} M \,\mathrm{d}y' = \sum_{i,j \in \{0,1\}} \dashint_{S_{ij}} M\,\mathrm{d}y',
							\end{align}
							and that $H_{\cor}^{(4,\beta)}\cdot \nu' = 0$ on $\partial[0,1]^2\times\{y_3\}$. 
							
							Hence, $H_{\cor}^{(4)} = H_{\cor}^{(4,\alpha)}+ H_{\cor}^{(4,\beta)}$ is a solution of \eqref{eq:cor-4} with the right boundary conditions. By construction, it satisfies $\|H_{\cor}^{(4)}\|_{L^{\infty}([0,1]^2\times\{y_3\})} \lesssim 1$ for $y_3 \in \left[0,\frac{1}{2}\right)$, hence
							\begin{align}\label{eq:corrector-bound-4}
								\|H_{\cor}^{(4)}\|_{L^{\infty}\left([0,1]^2\times\left[0,\frac{1}{2}\right)\right)} \lesssim 1.
							\end{align}
						
						\item Next, we solve 
							\begin{align*}
								-\nabla'\cdot H_{\cor}^{(5)} = \sum_{i,j \in \{0, 1\}} y_3 \left( \int_{[0,1]^2} \partial_3 M\,\mathrm{d}y' - \dashint_{S_{ij}} \partial_3 M\,\mathrm{d}y'\right) \mathcal{H}^1\lfloor_{\{y_1 = \eta_{(i,j)}(y_3)\}\cap S_{ij}}
							\end{align*}
							in $[0,1]^2$ with boundary condition $H_{\cor}^{(5)} \cdot \nu' = 0$ on $\partial [0,1]^2$ in an analogous way. More precisely, we set $H_{\cor}^{(5)} = H_{\cor}^{(5,\alpha)}+ H_{\cor}^{(5,\beta)}$ with 
							\begin{align*}
								H_{\cor}^{(5,\alpha)}(y) 
								&= -\sum_{j=0,1}  y_3 \left[\left(\int_{[0,1]^2} \partial_3 M\,\mathrm{d}y' - \dashint_{S_{0j}} \partial_3 M\,\mathrm{d}y'\right) \1_{\{y_1 \geq \eta_{(0,j)}(y_3)\}\cap S_{0j}} \right.\\
								&\qquad\qquad- \left.\left( \int_{[0,1]^2} \partial_3 M\,\mathrm{d}y' - \dashint_{S_{1j}} \partial_3 M\,\mathrm{d}y'\right) \1_{\{y_1 \leq \eta_{(1,j)}(y_3)\}\cap S_{1j}} \right] e_1,
							\end{align*}
							and a corresponding correction for the discontinuities at the interface $\{y_1 = \frac{1}{2}\}$ given by 
							\small
							\begin{align*}
								H_{\cor}^{(5,\beta)}(y) = \begin{cases}
																-y_3\left(2 \int_{[0,1]^2} \partial_3 M\,\mathrm{d}y' - \dashint_{S_{0,0}} \partial_3M \,\mathrm{d}y' - \dashint_{S_{1,0}} \partial_3M \,\mathrm{d}y' \right) (e_1 + e_2) & \text{in } S_{10} \cap \left\{y_2 \geq y_1-\frac{1}{2}\right\}, \\
																-y_3\left(2 \int_{[0,1]^2} \partial_3 M\,\mathrm{d}y' - \dashint_{S_{0,1}}\partial_3M \,\mathrm{d}y' - \dashint_{S_{1,1}} \partial_3M \,\mathrm{d}y' \right) (e_1 - e_2) & \text{in } S_{11} \cap \left\{y_2 \leq \frac{3}{2} - y_1\right\}, \\
																0 & \text{otherwise}.  
															\end{cases}
							\end{align*}
							\normalsize
							In particular, we have the bound $\|H_{\cor}^{(5)}\|_{L^{\infty}([0,1]^2\times \{y_3\})} \lesssim \|\partial_3 M\|_{L^{\infty}([0,1]^2 \times \{y_3\})}$ for $y_3 \in \left[0,\frac{1}{2}\right)$, hence
							\begin{align}\label{eq:corrector-bound-5}
								\|H_{\cor}^{(5)}\|_{L^2\left([0,1]^2\times\left[0,\frac{1}{2}\right)\right)} \lesssim \|\partial_3 M\|_{L^2_{y_3}L^{\infty}_{y'}(Q)}.
							\end{align}
						\item Finally, we solve the elliptic PDE
							\begin{align*}
								\begin{array}{rcll}
								\Delta' v_{\cor}^{(6)} &=& \partial_3 M - \sum_{i,j\in\{0,1\}} \left(\frac{1}{2} \dashint_{S_{ij}} \partial_3 M \,\mathrm{d}y'\right) \mathcal{H}^1\lfloor_{\{y_1 = \eta_{(i,j)}(y_3)\}\cap S_{ij}} & \text{in } [0,1]^2 \times \{y_3\}, \\
								\nabla'  v_{\cor}^{(6)} \cdot \nu' & =& 0 &\text{on } \partial [0,1]^2 \times \{y_3\}.
								\end{array} 
							\end{align*}
							Note that by \eqref{eq:mass-balance},
							\begin{align*}
								\int_{[0,1]^2} \partial_3 M \,\mathrm{d}y' - \frac{1}{2} \sum_{i,j\in\{0,1\}} \left(\frac{1}{2} \dashint_{S_{ij}} \partial_3 M \,\mathrm{d}y'\right) 
								= \partial_3 \left( \int_{[0,1]^2} M  \,\mathrm{d}y' - \frac{1}{4} \sum_{i,j\in\{0,1\}} \dashint_{S_{ij}} M \,\mathrm{d}y' \right) = 0.
							\end{align*}
							Hence the equation is solvable and with Lemma~\ref{lem:elliptic-line} (for $d=2$) we obtain a correction field $H_{\cor}^{(6)} = \nabla' v_{\cor}^{(6)}$ satisfying $\|H_{\cor}^{(6)}\|_{L^2([0,1]^2 \times \{y_3\})} \lesssim \|\partial_3 M\|_{L^2([0,1]^2\times\{y_3\})} \lesssim \|\partial_3 M\|_{L^{\infty}([0,1]^2\times\{y_3\})}$,
							which gives the bound 
							\begin{align}\label{eq:corrector-bound-6}
								\|H_{\cor}^{(6)}\|_{L^2\left([0,1]^2\times\left[0,\frac{1}{2}\right)\right)} \lesssim \|\partial_3 M\|_{L^2_{y_3} L^{\infty}_{y'}(Q)}.
							\end{align}
					\end{enumerate}
			Putting together the different parts of the corrector field, i.e.\ 
			\begin{align*}
				H_{\cor} = \begin{cases}
 							H_{\cor}^{(1)} + H_{\cor}^{(2)} + H_{\cor}^{(3)} & \text{for } y_3 \in \left[\frac{1}{2},1\right], \\
 							H_{\cor}^{(4)} + H_{\cor}^{(5)} + H_{\cor}^{(6)} & \text{for } y_3 \in \left[0, \frac{1}{2}\right), 
 						\end{cases}
			\end{align*}
			we obtain a field that solves \eqref{eq:corrector-field}.
			\end{enumerate}
		\item \emph{(Bound on the energy of the corrector field)}.
			Combining the estimates \eqref{eq:corrector-bound-1}, \eqref{eq:corrector-bound-2}, \eqref{eq:corrector-bound-3}, \eqref{eq:corrector-bound-4}, \eqref{eq:corrector-bound-5}, and \eqref{eq:corrector-bound-6}, we get, by the assumption $\|\partial_{3} M\|_{L^2_{y_{3}}L^{\infty}_{y'}(Q)} \lesssim 1$, that
			\begin{align*}
				\int_Q |H_{\cor}|^2\,\mathrm{d}y \lesssim 1 + \|\partial_3 M\|_{L^2_{y_3} L^{\infty}_{y'}(Q)} \lesssim 1.
			\end{align*}			
	\end{enumerate}
\end{proof}

With the building block at hand, we can now proceed with the construction of a competitor for the non-convex problem.
\begin{corollary}[A competitor for the non-convex problem] \label{cor:competitor}
Let $(m_{\rel},h_{\rel}) \in \mathcal{A}_{Q_{L,T}}^{\rel}(g, m^{B,T})$ and assume that $\|\partial_{d+1} m_{\rel}\|_{L^2_{x_{d+1}}L^{\infty}_{x'}(Q_{L,T})}^2\lesssim T^{-1}$. 
Then for any $N\in\NN$ there exist functions $\widetilde{m}\in L^1_{x_{d+1}}([0,T];BV_{x'}(Q_{L}'; \{\pm 1\}))$ and $\widetilde{h} \in L^2(Q_{L,T}; \RR^d)$ such that $\partial_{d+1} \widetilde{m} + \nabla'\cdot\widetilde{h} = 0$ in $Q_{L,T}$, $\widetilde{h}\cdot \nu' = g$ on $\Gamma_{L,T}$, $\widetilde{m}\stackrel{*}{\rightharpoonup} m^{B,T}$ as $x_{d+1} \to 0,T$, and with the property that
\begin{align}
	\dashint_{Q_{L,T}} |\nabla' \widetilde{m}| &\lesssim \left(\frac{L}{N}\right)^{-1} , \label{eq:bound-mtilde-competitor} \\
	\dashint_{Q_{L,T}} |\widetilde{h}-h_{\rel}|^2\,\mathrm{d}x &\lesssim \frac{L^2}{N^2} \frac{1}{T^2}. \label{eq:bound-htilde-competitor}
\end{align}
\end{corollary}

\begin{remark}\label{rem:competitornonconvex}
	At this stage, including the additional parameter $N\in\NN$ seems artificial. However, it will play an important role in balancing the two energy terms. Indeed, the right-hand sides of \eqref{eq:bound-mtilde-competitor} and \eqref{eq:bound-htilde-competitor} are of the same order if and only if 
	\begin{align*}
		\frac{N}{L} \sim \frac{L^2}{N^2} \frac{1}{T^2}, \quad \text{that is,} \quad N \sim \frac{L}{T^{\frac{2}{3}}},
	\end{align*}
	in which case 
	\begin{align*}
		\dashint_{Q_{L,T}} |\nabla' \widetilde{m}| + \dashint_{Q_{L,T}} |\widetilde{h}-h_{\rel}|^2\,\mathrm{d}x \lesssim T^{-\frac{2}{3}}.
	\end{align*}
	Of course, since $N$ has to be an integer, this is only possible if $L\gtrsim T^{\frac{2}{3}}$.
\end{remark}

\begin{proof} We proceed in several steps.

\begin{enumerate}[label=\textsc{\bf Step \arabic*},leftmargin=0pt,labelsep=*,itemindent=*,itemsep=10pt,topsep=10pt]
	\item 	\emph{Decomposition of $Q_{L,T}$.}
			We divide the cuboid $Q_{L,T}$ into smaller cuboids, where in order to treat the top and bottom boundary condition $m \weaklystar m^{B,T}$ we do this in such a way that the $x_{d+1}$-scale gets finer and finer towards the boundaries $x_{d+1}=0,T$.
			To this end, we dyadically decompose $Q_{L,T}$, taking into account the natural anisotropic scaling between horizontal and vertical directions. More precisely, fix an integer $N\in\NN$, let $k \in \NN$ and split
			\begin{align*}
				Q_{L,T} = \bigcup_{k\in\NN} \bigcup_{i_1, \dots, i_d \in \{-2^k N, \dots, 2^k N-1\}} Q_{i_1, \dots, i_d}^k,
			\end{align*}
			where $Q_{i_1, \dots, i_d}^k \coloneq P_{i_1, \dots, i_d}^k \times I_k$ with dyadic horizontal plaquettes 
			\begin{align*}
				\textstyle P_{i_1, \dots, i_d}^k = \bigtimes_{n=1}^d \left[i_n 2^{-k} \frac{L}{N}, (i_n+1) 2^{-k}\frac{L}{N}\right],
			\end{align*}
			and dyadic vertical intervals $I_k = I_k^0 \cup I_k^T$, 
			\begin{align*}\textstyle
				I_k^0 = \left[(2^{\frac{3}{2}})^{-k}\frac{T}{2},(2^{\frac{3}{2}})^{-(k-1)}\frac{T}{2}\right], \quad
				I_k^T = \left[T-(2^{\frac{3}{2}})^{-(k-1)}\frac{T}{2}, T-(2^{\frac{3}{2}})^{-k}\frac{T}{2}\right].
			\end{align*}
			Note that $\bigcup_{k\in\NN} I_k^0 = \left[0, \frac{T}{2}\right]$ and $\bigcup_{k\in\NN} I_k^T =  \left[\frac{T}{2},T\right]$, so we are refining towards the top and bottom boundaries. 
	\item 	\emph{Reduction to building blocks.}
			According to the decomposition of $Q_{L,T}$ we decompose 
			\begin{align*}
				m = \sum_{k} \sum_{i_1, \dots, i_d} m_{i_1 \dots i_d}^k, \quad h_{\rel} = \sum_{k} \sum_{i_1, \dots, i_d} h_{i_1 \dots i_d}^k,
			\end{align*}
			where $m_{i_1 \dots i_d}^k = m \1_{Q_{i_1 \dots i_d}^k}$, $h_{i_1 \dots i_d}^k = h \1_{Q_{i_1 \dots i_d}^k}$.
			We then rescale each of the cubes $Q_{i_1, \dots, i_d}^k$ with the matrix
			\begin{align*}
			S_k = \mathrm{diag}\left(\sigma', \dots, \sigma', \sigma_{d+1}\right), \quad \text{with} \quad  \sigma'=2^{-k}\frac{L}{N}, \, \sigma_{d+1}=(2^{\frac{3}{2}})^{-k} \frac{T}{2} (2^{\frac{3}{2}}-1),
			\end{align*} 
			and shift by the vectors $b=b_{i_1 \dots i_d} = (i_1, \dots, i_d, (2^{\frac{3}{2}}-1)^{-1})\in\RR^{d+1}$, to obtain 
			\begin{align*}
				Q^k_{i_1 \dots i_d} = S_k  \left( [0,1]^{d+1} + b \right).
			\end{align*}
			Based on this transformation, we define the new variables $y = (S_k )^{-1}x - b \in [0,1]^{d+1}$ and the transformed functions 
			\begin{align*}
				M_{i_1 \dots i_d}^k(y) &= m_{i_1 \dots i_d}^k(S_k (y + b)), \quad \text{and} \quad
				H_{i_1 \dots i_d}^k(y) = \frac{\sigma_{d+1}}{\sigma'} h_{i_1 \dots i_d}^k(S_k (y + b)).
			\end{align*}
			Note that $M_{i_1 \dots i_d}^k \in L^{\infty}(Q;[-1,1])$, $H_{i_1 \dots i_d}^k \in L^2(Q; \RR^d)$, and $\partial_{d+1} M_{i_1 \dots i_d}^k + \nabla'\cdot H_{i_1 \dots i_d}^k = 0$. 
			Furthermore,
			\begin{align}\label{eq:d3mtilde-rescaled}
				\|\partial_{d+1} M_{i_1 \dots i_d}^k\|_{L^2_{y_{d+1}}L^{\infty}_{y'}(Q)}^2 = \sigma_{d+1} \|\partial_{d+1} m_{i_1 \dots i_d}^k\|_{L^2_{x_{d+1}}L^{\infty}_{x'}\left(Q_{i_1 \dots i_d}^k\right)}^2.
			\end{align}
	\item 	\emph{(Putting together the building blocks).}
			Appealing to Lemma~\ref{lem:buildingblock}, there exist functions $\widetilde{M}_{i_1 \dots i_d}^k\in BV(Q; \pm 1)$ and $\widetilde{H}_{i_1 \dots i_d}^k \in L^2(Q; \RR^d)$ such that 
			\begin{align*}
				\begin{array}{rcll}
				\partial_{d+1} \widetilde{M}_{i_1 \dots i_d}^k + \nabla'\cdot\widetilde{H}_{i_1 \dots i_d}^k &=& 0 & \text{in } Q \\
				\widetilde{H}_{i_1 \dots i_d}^k\cdot \nu' &=& H_{i_1 \dots i_d}^k\cdot\nu' & \text{on } \Gamma,
				\end{array}
			\end{align*}
			such that
			\begin{align*}
				\int_Q |\nabla' \widetilde{M}_{i_1 \dots i_d}^k| &\lesssim 1, \\
				\int_Q |\widetilde{H}_{i_1 \dots i_d}^k-H_{i_1 \dots i_d}^k|^2\,\mathrm{d}x &\lesssim 1 + \|\partial_{d+1} M_{i_1 \dots i_d}^k\|_{L^2_{x_{d+1}}L^{\infty}_{x'}(Q)}^2.
			\end{align*}
			Define 
			\begin{align*}
				\widetilde{m} &= \sum_{k} \sum_{i_1 \dots i_d} \widetilde{m}_{i_1 \dots i_d}^k \1_{Q_{i_1 \dots i_d}^k}, \quad \text{with} \quad \widetilde{m}_{i_1 \dots i_d}^k(x) = \widetilde{M}_{i_1 \dots i_d}^k((S_k )^{-1}x - b), \\
				\widetilde{h} &= \sum_{k} \sum_{i_1 \dots i_d} \widetilde{h}_{i_1 \dots i_d}^k \1_{Q_{i_1 \dots i_d}^k}, \quad \text{with} \quad \widetilde{h}_{i_1 \dots i_d}^k(x) = \frac{\sigma'}{\sigma_{d+1}} \widetilde{H}_{i_1 \dots i_d}^k((S_k )^{-1}x - b).
			\end{align*}
			Then $\widetilde{m} \in BV(Q_{L,T}; \pm 1)$, $\widetilde{h} \in L^2(Q_{L,T}; \RR^d)$, $\partial_{d+1}\widetilde{m} + \nabla'\cdot \widetilde{h} = 0$, and $\partial_{d+1} \widetilde{m} \in L^2_{x_{d+1}}L^{\infty}_{x'}$ since the constructed $\widetilde{m}$ is continuous across the top and bottom boundaries of $Q_{i_1 \dots i_d}^k$ by construction, see Remark~\ref{rem:continuity}.

	\item	\emph{Estimate on the energies.}
			For the energies, we have 
			\begin{align*}
				\int_{Q_{L,T}} |h_{\rel} - \widetilde{h}|^2\,\mathrm{d}x 
				&= \sum_k \sum_{i_1, \dots, i_d} \int_{Q_{i_1 \dots i_d}^k} |h_{i_1 \dots i_d}^k(x) - \widetilde{h}_{i_1 \dots i_d}^k(x)|^2\,\mathrm{d}x \\
				&= \sum_k \sum_{i_1, \dots, i_d} \sigma'^d \sigma_{d+1} \int_{Q} |h_{i_1 \dots i_d}^k(S_k (y + b)) - \widetilde{h}_{i_1 \dots i_d}^k(S_k (y + b))|^2\,\mathrm{d}y \\
				&= \sum_k \sum_{i_1, \dots, i_d} \frac{\sigma'^{d+2}}{\sigma_{d+1}} \int_Q |H_{i_1 \dots i_d}^k(y) - \widetilde{H}_{i_1 \dots i_d}^k(y)|^2\,\mathrm{d}y.
			\end{align*}
			It follows with $\|\partial_{d+1} m\|_{L^2_{x_{d+1}}L^{\infty}_{x'}(Q_{L,T})}^2 \lesssim \frac{1}{T}$ and \eqref{eq:d3mtilde-rescaled}, that $\|\partial_{d+1} M_{i_1 \dots i_d}^k\|_{L^2_{y_{d+1}}L^{\infty}_{y'}(Q)}^2 \lesssim \frac{\sigma_{d+1}}{T} \lesssim 1$ for any $k \in \NN$, hence by \eqref{eq:bound-htilde}, 
			\begin{align*}
				\int_{Q_{L,T}} |h_{\rel} - \widetilde{h}|^2\,\mathrm{d}x 
				&\lesssim \sum_k \sum_{i_1, \dots, i_d} \frac{\sigma'^{d+2}}{\sigma_{d+1}} 
				\lesssim \sum_k (2^k N)^d \left(2^{-k} \frac{L}{N}\right)^{d+2} \frac{(2^k)^{\frac{3}{2}}}{T} \\
				&\lesssim \left(\frac{L}{N}\right)^2 \frac{L^d}{T} \sum_k \left(\frac{1}{\sqrt{2}}\right)^k 
				\lesssim \left(\frac{L}{N}\right)^2 \frac{L^d}{T}.
			\end{align*}
			For the surface energy created in the construction of $\widetilde{m}$ we obtain
			\begin{align*}
				\int_{Q_{L,T}} |\nabla'\widetilde{m}| 
				\leq \sum_{k} \sum_{i_1 \dots i_d} \sigma'^{d-1} \sigma_{d+1} \left(\int_{Q} |\nabla' \widetilde{M}_{i_1 \dots i_d}^k| + 8 \right),
			\end{align*}
			where the additional constant term accounts for the jumps between neighbouring building blocks. 
					
			Hence, by \eqref{eq:bound-mtilde}, it follows that 
			\begin{align*}
				\int_{Q_{L,T}} |\nabla'\widetilde{m}| 
				&\lesssim \sum_{k} \sum_{i_1 \dots i_d} \sigma'^{d-1} \sigma_{d+1} 
				\lesssim \sum_k (2^k N)^d \left(2^{-k} \frac{L}{N}\right)^{d-1} (2^{-k})^{\frac{3}{2}} T \\
				&\lesssim \left(\frac{L}{N}\right)^{-1} L^{d} T \sum_k 2^{-\frac{k}{2}} 
				\lesssim \left(\frac{L}{N}\right)^{-1} L^{d} T. \qedhere
			\end{align*}
\end{enumerate}
\end{proof}
\section{Global energy bounds: Proof of Theorem~\ref{thm:globalbound}}\label{sec:proofglobalscalinglaw}
\subsection{Lower bound on the global energy}\label{sec:lower-bound-global}
The lower bound on the global scaling of the energy of a periodic minimising configuration has been proved in \cite{CO16} based on the interpolation inequality 
		\begin{align}\label{eq:CO-interpolation}
			\|u\|_{L^{\frac{4}{3}}} \lesssim \|\nabla'u\|_{L^1}^{\frac{1}{2}} \||\nabla'|^{-1} u\|_{L^2}^{\frac{1}{2}}
		\end{align}
		for any periodic function $u: [0,\Lambda]^d \to \RR$ with $\int_{[0,\Lambda]^d} u \,\mathrm{d}x' = 0$, with an implicit constant that only depends on $d$. A similar proof based on a weak-$L^{\frac{4}{3}}$ interpolation inequality can also be found in \cite{Vie09}. 

\begin{proof}[Proof of Theorem \ref{thm:globalbound} (Lower bound)]
	Note that if $(m,h) \in \mathcal{A}^{\mathrm{per}/0}_{Q_{L,T}}$, in particular $m^{B,T} \equiv 0$, then $\int_{Q_L'} m(x', x_{d+1})\,\mathrm{d}x' = 0$ for any $x_{d+1} \in [0,T]$. Indeed, by the first equation in \eqref{eq:admissible} and integration by parts (using periodic boundary conditions or $h \cdot \nu' = 0$ on $\partial Q_{L}'$), it follows that 
		\begin{align*}
			\int_{Q_L'} m(x', x_{d+1})\,\mathrm{d}x' 
			&= \int_{Q_L'} \int_0^{x_{d+1}} \partial_{x_{d+1}} m(x', y_{d+1})\,\mathrm{d}y_{d+1}\,\mathrm{d}x'
			= -\int_0^{x_{d+1}} \int_{Q_L'} \nabla'\cdot h\,\mathrm{d}x' \,\mathrm{d}y_{d+1} \\
			&= -\int_0^{x_{d+1}} \int_{\partial Q_L'} h\cdot\nu' \,\mathrm{d}\mathcal{H}^{d-1}\,\mathrm{d}y_{d+1}
			=0.
		\end{align*}
		By admissibility, we can bound 
		\begin{align*}
			E_{Q_{L,T}}(m,h) \geq E_{Q_{L,T}}(m^*, h^*)
		\end{align*}
		for any minimiser $(m^*, h^*)$ in $\mathcal{A}^{\mathrm{per}}_{Q_{L,T}}$ and $\mathcal{A}^{0}_{Q_{L,T}}$, respectively. 
		By \eqref{eq:representation-strayfield} and Poincaré's inequality (in $x_{d+1}$; recall that $m^{B,T}\equiv 0$) we can further bound 
		\begin{align*}
			\int_{Q_{L,T}} |h^*|^2 \,\mathrm{d}x 
			&= \int_{Q_L'}\int_0^T |\partial_{x_{d+1}}|\nabla'|^{-1}m^*|^2 \,\mathrm{d}x_{d+1}\,\mathrm{d}x'
			\geq \frac{1}{T^2} \int_{Q_L'}\int_0^T ||\nabla'|^{-1}m^*|^2 \,\mathrm{d}x_{d+1}\,\mathrm{d}x'.
		\end{align*}

		In the periodic case, we can now apply Young's inequality and \eqref{eq:CO-interpolation} to estimate
		\begin{align}
			E_{Q_{L,T}}(m^*,h^*) &= \int_{Q_{L,T}} |\nabla' m^*|\,\mathrm{d}x + \frac{1}{2}\int_{Q_{L,T}} |h^*|^2 \,\mathrm{d}x \nonumber \\ 
			&\geq \int_0^T \left(\int_{Q_L'} |\nabla' m^*|\,\mathrm{d}x' + \frac{1}{2T^2} \int_{Q_L'} ||\nabla'|^{-1}m^*|^2\,\mathrm{d}x' \right)\,\mathrm{d}x_{d+1} \nonumber \\
			&\gtrsim \int_0^T \left(\int_{Q_L'} |\nabla' m^*|\,\mathrm{d}x' \right)^{\frac{2}{3}} \left( \frac{1}{T^2} \int_{Q_L'} ||\nabla'|^{-1}m^*|^2\,\mathrm{d}x' \right)^{\frac{1}{3}} \,\mathrm{d}x_{d+1} \nonumber \\
			&\gtrsim T^{-\frac{2}{3}} \int_0^T \int_{Q_L'} |m^*|^{\frac{4}{3}}\,\mathrm{d}x'\,\mathrm{d}x_{d+1} 
			\gtrsim L^d T^{\frac{1}{3}}, \label{eq:lower-bound-global}
		\end{align}
		since $|m^*| = 1$. 
		
		\medskip
		For the case of zero-flux boundary conditions on $\Gamma_{L,T}$ we extend $(m,h)$ to $[-L, 3L]^d \times [0,T]$ by a sequence of $d$ even reflections of $m$ and corresponding reflections of $h$ across each face of $\Gamma_{L,T}$ to obtain an admissible configuration in $\mathcal{A}^{\mathrm{per}}_{Q_{[-L,3L]^d\times [0,T]}}$. More precisely, given $(m,h)$ on $Q_{L,T}$ we first define 
		\small
		\begin{align*}
			m^{(1)}(x) &\coloneqq \begin{cases}
				m(x), & x\in [-L,L] \times [-L,L]^{d-1}\times[0,T] \\
				m(2L-x_1, x_2, \dots, x_{d+1}), & x \in [L,3L]\times [-L,L]^{d-1} \times [0,T],
			\end{cases}\\ 
			h^{(1)}(x) &\coloneqq \begin{cases}
				h(x), & x\in [-L,L] \times [-L,L]^{d-1}\times[0,T] \\
				(-h_1, h_2, \dots, h_d)(2L-x_1, x_2, \dots, x_{d+1}), & x \in [L,3L]\times [-L,L]^{d-1} \times [0,T],
			\end{cases}
		\end{align*}
		\normalsize
		so that $\partial_{d+1} m^{(1)} + \nabla' \cdot h^{(1)} = 0$ distributionally in $[-L, 3L] \times [-L,L]^{d-1}\times[0,T]$ with lateral boundary conditions $h^{(1)}\cdot \nu' = 0$ on $\partial ([-L, 3L] \times [-L,L]^{d-1}) \times [0,T]$. Note that thanks to the zero-flux boundary conditions satisfied by $h$, the field $h^{(1)}$ has no jump across the surface $\{L\}\times[-L,L]^{d-1}\times [0,T]$; similarly, by construction, the magnetisation $m^{(1)}$ has no additional jumps on the reflection surface $\{L\}\times[-L,L]^{d-1}\times [0,T]$. 
		We then set 
		\footnotesize
		\begin{align*}
			m^{(2)}(x) &\coloneqq \begin{cases}
				m^{(1)}(x), & x\in [-L, 3L] \times [-L,L] \times [-L,L]^{d-2}\times[0,T] \\
				m^{(1)}(x_1, 2L-x_2,x_3, \dots, x_{d+1}), & x \in [-L,3L]\times [L,3L] \times[-L,L]^{d-2} \times [0,T],
			\end{cases}\\ 
			h^{(2)}(x) &\coloneqq \begin{cases}
				h^{(1)}(x), & x\in [-L, 3L] \times [-L,L] \times [-L,L]^{d-2}\times[0,T] \\
				\left(h_1^{(1)}, -h_2^{(1)}, h_3^{(1)},\dots, h_d^{(1)}\right)(x_1, 2L-x_2, x_3,\dots, x_{d+1}), & x \in [-L,3L]\times [L,3L] \times[-L,L]^{d-2} \times [0,T],
			\end{cases}
		\end{align*}
		\normalsize
		which satisfies $\partial_{d+1} m^{(2)} + \nabla' \cdot h^{(2)} = 0$ distributionally in $[-L, 3L]^2 \times [-L,L]^{d-2}\times[0,T]$ with lateral boundary conditions $h^{(2)}\cdot \nu' = 0$ on $\partial ([-L, 3L]^2 \times [-L,L]^{d-2}) \times [0,T]$, and proceed inductively to define 
		\footnotesize
		\begin{align*}
			m^{(d)}(x) &\coloneqq \begin{cases}
				m^{(d-1)}(x), & x\in [-L, 3L]^{d-1} \times [-L,L]\times[0,T] \\
				m^{(d-1)}(x_1,\dots, x_{d-1}, 2L-x_d, x_{d+1}), & x \in  [-L, 3L]^{d-1} \times [L,3L]\times[0,T],
			\end{cases}\\ 
			h^{(d)}(x) &\coloneqq \begin{cases}
				h^{(d-1)}(x), & x\in [-L, 3L]^{d-1} \times [-L,L]\times[0,T] \\
				\left(h_1^{(d-1)}, \dots, h_{d-1}^{(d-1)}, -h_d^{(d-1)}\right)(x_1,\dots, x_{d-1}, 2L-x_d, x_{d+1}), & x \in  [-L, 3L]^{d-1} \times [L,3L]\times[0,T].
			\end{cases}
		\end{align*}
		\normalsize
		By construction, $(m^{(d)}, h^{(d)}) \in \mathcal{A}^{\mathrm{per}}_{[-L,3L]^d \times [0,T]}$, and since we have not added any extra inter-facial energy, we have 
		\begin{align*}
			E_{[-L,3L]^d \times [0,T]}(m^{(d)}, h^{(d)}) = 2^d E_{Q_{L,T}}(m,h).
		\end{align*}
		Applying the lower bound \eqref{eq:lower-bound-global}, which is valid in the periodic case, to $E_{[-L,3L]^d \times [0,T]}(m^{(d)}, h^{(d)})$, it follows that there exists a universal constant $c_s<\infty$ such that 
		\begin{align*}
			E_{Q_{L,T}}(m,h) 
			= 2^{-d} E_{[-L,3L]^d \times [0,T]}(m^{(d)}, h^{(d)})
			\geq c_s L^d T^{\frac{1}{3}}.
		\end{align*}
\end{proof}

\subsection{Upper bound on the global energy} The upper bound is mainly a consequence of the construction in Corollary~\ref{cor:competitor}.

\begin{proof}[Proof of Theorem \ref{thm:globalbound} (Upper bound)]
Since $m^{B,T}=0$, in the case of zero-flux boundary conditions we have that $(m_{\rel},h_{\rel})\equiv 0$ is a solution of the relaxed problem. Therefore, Corollary~\ref{cor:competitor} provides the existence of a competitor for the non-convex problem, that is a pair $(\widetilde m,\widetilde h)\in \mathcal{A}^{0}_{Q_{L,T}}$ (recall \eqref{eq:admissible-zero}). In particular, $\widetilde h$ satisfies zero-flux boundary conditions on $\Gamma_{L,T}$ and  $\widetilde{m}\stackrel{*}{\rightharpoonup}0$ as $x_{d+1} \to 0,T$. Moreover, $(\widetilde{m},\widetilde{h})$ satisfies \eqref{eq:bound-mtilde-competitor} and \eqref{eq:bound-htilde-competitor} (with $h_{\rel}\equiv 0$), which by Remark \ref{rem:competitornonconvex}\footnote{Recall that this requires the constant $C_{LT}$ in the relation $L \geq C_{LT} T^{\frac{2}{3}}$ to be large enough.} lead us to
$$
\min_{(m,h)\in \mathcal{A}^{0}_{Q_{L,T}}} E_{Q_{L,T}}(m,h)\leq |Q_{L,T}|\left(\dashint_{Q_{L,T}} |\nabla' \widetilde{m}| + \dashint_{Q_{L,T}} |\widetilde{h}-h_{\rel}|^2\,\mathrm{d}x\right) \lesssim  |Q_{L,T}|T^{-\frac{2}{3}}\lesssim L^dT^\frac13.
$$ 

Let us remark that in this case the competitor is built out of building blocks looking exactly as in Figure~\ref{fig:example1-construction} (since $M=\overline M=0$ everywhere). 
Also, at this point one sees the importance of the refinement in our construction, since it is the key point that allows for $\widetilde m$ to satisfy the weak boundary condition at $x_{d+1}=0,T$.

\medskip 
In the case of periodic boundary conditions, we do exactly the same, except that this time the constructed pair $(\widetilde m,\widetilde h)$ does not satisfy the correct boundary conditions. Arguing as in the proof of the lower bound, we can extend $(\widetilde m,\widetilde h)$ to  $[-L,3L]^d\times [0,T]$ by a sequence of $d$ even reflections of $\widetilde m$ and corresponding reflections of $\widetilde h$ across each face of $\Gamma_{L,T}$ to obtain a periodic configuration $(\widetilde m^{(d)},\widetilde h^{(d)})\in \mathcal{A}^{\mathrm{per}}_{Q_{[-L,3L]^d\times [0,T]}}$ that satisfies 
$$
E_{[-L,3L]^d \times [0,T]}(\widetilde m^{(d)}, \widetilde h^{(d)})=2^d E_{Q_{L,T}}(\widetilde m,\widetilde h) \lesssim L^dT^\frac13.
$$
Then, by suitably contracting this pair in the horizontal directions (keeping the height fixed), we obtain a competitor for the non-convex energy \eqref{eq:energy-nonconvex} in $\mathcal{A}^{\mathrm{per}}_{Q_{L,T}}$ (recall \eqref{eq:admissible-periodic}), with the same energy (up to a larger dimensional factor). More precisely, setting
\begin{align*}
\widetilde{\widetilde{m}}(x) \coloneqq \widetilde{m}^{(d)}\left(2 x' +(\ell, \dots, \ell) , x_{d+1} \right), \quad 
\widetilde{\widetilde{h}}(x) \coloneqq \frac12 \widetilde{h}^{(d)}\left( 2 x' +(\ell, \dots, \ell) , x_{d+1} \right), \quad  \text{for } x\in  Q_{L,T},
\end{align*}
we get the desired periodic competitor, which satisfies 
$$
E_{Q_{L,T}}(\widetilde{\widetilde{m}},\widetilde{\widetilde{h}})\lesssim E_{[-L,3L]^d \times [0,T]}(\widetilde m^{(d)}, \widetilde h^{(d)})\lesssim L^dT^\frac13,
$$
hence
$$
\min_{(m,h)\in\mathcal{A}^{\mathrm{per}}_{Q_{L,T}}} E_{Q_{L,T}}(m,h)\lesssim L^dT^\frac13.
$$ 
This concludes the proof.
\end{proof}
\section{Local energy bounds: Proof of Theorem~\ref{thm:localbounds}}\label{sec:prooflocalscalinglaws}
We are now ready to present the proofs of the local energy bounds. Without loss of generality, we will prove these estimates for small cuboids at the bottom boundary of the sample. The proofs can easily be adjusted for cuboids at the top boundary.

\subsection{Lower bound on the local energy}
The lower bound on the local energy in cuboids at the bottom sample boundary is obtained by a suitable extension of the localised minimiser $(m_{\ell,t}, h_{\ell,t}) \coloneqq (m|_{Q_{\ell,t}(a)}, h|_{Q_{\ell,t}(a)})$ and rescaling to obtain an admissible pair $(\tilde{m}, \tilde{h})$ in a large cuboid with periodic lateral boundary conditions and zero top/bottom magnetisation, for which the global lower bound can be applied. 

In contrast to the periodic extension in the global case (see Section~\ref{sec:lower-bound-global}), we now have to take into account that the normal component of $h$ along $\Gamma_{\ell,t}(a)$ is non-vanishing. Since jumps of $h$ at the boundary have to be compensated by an according change in $m$ in order to preserve admissibility, the extension will be performed differently. 

\begin{proof}[Proof of Theorem~\ref{thm:localbounds} (Lower bound)]
Let us assume w.l.o.g.\ that $a=0$. We proceed in several steps.

\begin{enumerate}[label=\textsc{\bf Step \arabic*},leftmargin=0pt,labelsep=*,itemindent=*,itemsep=10pt,topsep=10pt]
	\item \emph{(Extension to 	$Q_{\ell, 2t}$)}. We first extend $m_{\ell,t}$ to a magnetisation in $Q_{\ell, 2t}$ by an even reflection across the top surface $[-\ell, \ell]^d \times\{t\}$,
	\begin{align*}
		\tilde{m}^{(0)}(x) \coloneqq	 \begin{cases}
			m_{\ell,t}(x), & x \in [-\ell, \ell]^d \times [0,t] \\
			m_{\ell,t}(x', 2t-x_{d+1}), & x\in [-\ell, \ell]^d \times [t,2t].
		\end{cases}
	\end{align*}
	Note that the extension $\tilde{m}^{(0)}$ is continuous along $[-\ell, \ell]^d \times\{t\}$. Since
	\begin{align*}
		\partial_{d+1}\tilde{m}^{(0)}(x) = \begin{cases}
			\partial_{d+1} m_{\ell,t}(x), & x \in [-\ell, \ell]^d \times [0,t] \\
			- \partial_{d+1} m_{\ell,t}(x', 2t-x_{d+1}), & x\in [-\ell, \ell]^d \times (t,2t],
		\end{cases}
	\end{align*}
	we define the extension of $h_{\ell,t}$ to $Q_{\ell, 2t}$ via odd reflection, 
	\begin{align*}
		\tilde{h}^{(0)}(x) \coloneqq	 \begin{cases}
			h_{\ell,t}(x), & x \in [-\ell, \ell]^d \times [0,t] \\
			-h_{\ell,t}(x', 2t-x_{d+1}), & x\in [-\ell, \ell]^d \times (t,2t],
		\end{cases}		
	\end{align*}
	so that 
	\begin{align*}
		\nabla'\cdot \tilde{h}^{(0)}(x) = \begin{cases}
			\nabla' \cdot h_{\ell,t}(x), & x \in [-\ell, \ell]^d \times [0,t] \\
			-\nabla'\cdot h_{\ell,t}(x', 2t-x_{d+1}), & x\in [-\ell, \ell]^d \times (t,2t].
		\end{cases}	
	\end{align*}
	In particular, $\partial_{d+1} \tilde{m}^{(0)} + \nabla'\cdot \tilde{h}^{(0)} = 0$ in $Q_{\ell, 2t}$, with top/bottom boundary conditions $\tilde{m}^{(0)}(\cdot, 0) = \tilde{m}^{(0)}(\cdot, 2t) = 0$ and lateral boundary conditions $\tilde{h}^{(0)}\cdot\nu' = h_{\ell,t}\cdot \nu'$ on $\Gamma_{\ell, t}$ and $\tilde{h}^{(0)}\cdot\nu' = -h_{\ell,t}\cdot \nu'$ on $\partial[-\ell, \ell]^d \times (t, 2t]$.
	
	\item\emph{(Extension to $[-\ell, 3\ell]^d \times [0,2t]$)}. Now that the magnetisation has been extended in such a way that the top boundary condition is zero, we extend $(\tilde{m}^{(0)}, \tilde{h}^{(0)})$ to a periodic (w.r.t.\ $x'$) configuration in $[-\ell, 3\ell]^d \times [0,2t]$. To this end, define for $j=1, \dots, d$\footnote{With the convention that $[a,b]^0=\emptyset$.}
	\begin{align*}
		\tilde{h}^{(j)}(x) \coloneqq \tilde{h}^{(j-1)}(x), \quad \text{if} \quad x \in [-\ell, 3\ell]^{j-1} \times [-\ell, \ell] \times [-\ell, \ell]^{d-j} \times [0,2t]
	\end{align*}
	and
	\small 
	\begin{align*}
		\tilde{h}^{(j)}(x) \coloneqq \left(-\tilde{h}^{(j-1)}_1, \dots, -\tilde{h}^{(j-1)}_{j-1}, \tilde{h}^{(j-1)}_j, -\tilde{h}^{(j-1)}_{j+1}, \dots, -\tilde{h}^{(j-1)}_d\right)(x_1,\dots, x_{j-1}, 2\ell-x_j, x_{j+1}, \dots, x_d, x_{d+1}),
	\end{align*}
	\normalsize
	if $x \in [-\ell, 3\ell]^{j-1} \times [\ell, 3\ell] \times [-\ell, \ell]^{d-j} \times [0,2t]$. Then $\nabla'\cdot \tilde{h}^{(j)}(x) = \nabla'\cdot \tilde{h}^{(j-1)}(x)$ for $x \in [-\ell, 3\ell]^{j-1} \times [-\ell, \ell] \times [-\ell, \ell]^{d-j} \times [0,2t]$,  
	$$\nabla'\cdot \tilde{h}^{(j)}(x)=-\nabla'\cdot \tilde{h}^{(j-1)}(x_1,\dots,x_{j-1}, 2\ell-x_j,x_{j+1}, \dots, x_d, x_{d+1})$$
	for $x \in [-\ell, 3\ell]^{j-1} \times [\ell, 3\ell] \times [-\ell, \ell]^{d-j} \times [0,2t]$, and $\tilde{h}^{(j)}\cdot \nu'$ is continuous across the $j$\textsuperscript{th} reflection surface $[-\ell, 3\ell]^{j-1} \times \{\ell\} \times [-\ell, \ell]^{d-j} \times [0,2t]$. Hence, setting 
	\begin{align*}
		\tilde{m}^{(j)}(x) \coloneqq \tilde{m}^{(j-1)}(x), \quad \text{if} \quad x \in [-\ell, 3\ell]^{j-1} \times [-\ell, \ell] \times [-\ell, \ell]^{d-j} \times [0,2t]
	\end{align*}
	and 
	\begin{align*}
		\tilde{m}^{(j)}(x) \coloneqq  -\tilde{m}^{(j-1)}(x_1,\dots, x_{j-1}, 2\ell-x_j, x_{j+1}, \dots, x_d, x_{d+1}),
	\end{align*}
	if $x \in [-\ell, 3\ell]^{j-1} \times [\ell, 3\ell] \times [-\ell, \ell]^{d-j} \times [0,2t]$, we have that $\partial_{d+1} \tilde{m}^{(j)} + \nabla'\cdot \tilde{h}^{(j)} = 0$ in $[-\ell, 3\ell]^{j} \times [-\ell, \ell]^{d-j} \times [0,2t]$ for all $j=0, \dots, d$. 
	Since we flipped the sign of the magnetisation in the extension, we also have that 
	\begin{align*}
		\int_{[-\ell, 3\ell]^j\times [-\ell, \ell]^{d-j}} \tilde{m}^{(j)}(x',x_{d+1}) \,\mathrm{d}x'=0 \quad \text{for all} \quad x_{d+1}\in[0,t], \quad  j=1, \dots, d.	
	\end{align*}
	
	\item \emph{(Rescaling and global lower bound)}. We now shift and rescale $(\tilde{m}^{(d)},  \tilde{h}^{(d)})$ to obtain an admissible configuration in $Q_{\tilde{L},T}$ with $\tilde{L} = \left(\frac{T}{2t}\right)^{\frac{2}{3}} 2\ell$. To this end,
	set 
	\begin{align*}
		\tilde{m}(x) &\coloneqq \tilde{m}^{(d)}\left( \left(\frac{2t}{T}\right)^{\frac{2}{3}} \left(x' + (\ell, \dots, \ell) \right), \frac{2t}{T} x_{d+1} \right) \quad \text{and} \\
		\tilde{h}(x) &\coloneqq \left(\frac{2t}{T}\right)^{\frac{1}{3}} \tilde{h}^{(d)}\left( \left(\frac{2t}{T}\right)^{\frac{2}{3}} \left(x' + (\ell, \dots, \ell) \right), \frac{2t}{T} x_{d+1} \right), \quad \text{for} \quad  x \in Q_{\tilde{L}, T}.
	\end{align*}
	Then $(\tilde{m}, \tilde{h}) \in \mathcal{A}^{\mathrm{per}}_{Q_{\tilde{L},T}}$ is admissible and by the global lower bound \eqref{eq:lower-bound-global} there exists a universal constant $c_{S}>0$ such that  
	\begin{align}\label{eq:lower-bound-local-rescaled}
		E_{Q_{\tilde{L},T}}(\tilde{m}, \tilde{h}) \geq c_{S} \tilde{L}^d T^{\frac{1}{3}}. 
	\end{align}
	
	\item \emph{(Energy bound)}. 
	Note that in the construction of $(\tilde{m}^{(d)}, \tilde{h}^{(d)})$ we changed the sign of the magnetisation along the reflection surface and therefore pick up an extra interfacial energy of twice the area of the surface along which we reflect, i.e.\ a total energy of $C_d \ell^{d-1} t$ for some dimensional constant $C_d<\infty$. 
	It follows that 
	\begin{align*}
		E_{Q_{\tilde{L},T}}(\tilde{m},\tilde{h}) 
		&\stackrel{\eqref{eq:scalingofenergy}}{=} \left(\frac{2t}{T}\right)^{-\frac{2}{3}d-\frac{1}{3}} E_{[-\ell, 3\ell]^d \times [0,2t]}(\tilde{m}^{(d)}, \tilde{h}^{(d)}) \\
		&= \left(\frac{2t}{T}\right)^{-\frac{2}{3}d-\frac{1}{3}} \left( 2^d E_{Q_{\ell,t}}(m_{\ell,t}, h_{\ell,t}) + C_d \ell^{d-1} t \right) \\
		&= \left(\frac{2t}{T}\right)^{-\frac{2}{3}d-\frac{1}{3}} \left( 2^d \ell^d t e_{Q_{\ell,t}}(m_{\ell,t}, h_{\ell,t}) + C_d \ell^{d-1} t \right),
	\end{align*}
	hence
	\begin{align*}
		e_{Q_{\ell,t}}(m_{\ell,t}, h_{\ell,t}) 
		&= 2^{-d} \ell^{-d} t^{-1} \left(\frac{2t}{T}\right)^{\frac{2}{3}d+\frac{1}{3}} E_{Q_{\tilde{L},T}}(\tilde{m},\tilde{h}) -  2^{-d} C_d \ell^{-1}\\
		&\stackrel{\eqref{eq:lower-bound-local-rescaled}}{\geq} 2^{-d} \ell^{-d} t^{-1} \left(\frac{2t}{T}\right)^{\frac{2}{3}d+\frac{1}{3}} c_{S} \tilde{L}^d T^{\frac{1}{3}} - 2^{-d} C_d \ell^{-1} 
		= 2^{\frac{1}{3}} c_{S} t^{-\frac{2}{3}} - 2^{-d} C_d \ell^{-1}.
	\end{align*}
	Therefore, there exists a universal constant $c_{\ell t} <\infty$ such that if $\ell \geq c_{\ell t} t^{\frac{2}{3}}$, then 
	\begin{align*}
		e_{Q_{\ell,t}}(m_{\ell,t}, h_{\ell,t}) \geq c_{s} t^{-\frac{2}{3}}
	\end{align*}
	with $c_s = 2^{\frac{1}{3}} c_{S} - \frac{C_d}{2^d c_{\ell t}} >0$. \qedhere
\end{enumerate}
\end{proof}

\subsection{Upper bound on the local energy}
As in \cites{Con00,Vie09}, instead of directly looking at the energy density in a smaller cuboid $Q_{\ell,t}(a) \coloneqq Q_{\ell,t} + a$ for $a \in \RR^d \times \{0\}$,
\begin{align*}
	e(Q_{\ell,t}(a)) \coloneqq \dashint_{Q_{\ell,t}(a)} |\nabla'm| + \dashint_{Q_{\ell,t}(a)} \frac{1}{2} |h|^2\,\mathrm{d}x,
\end{align*}
we will keep track of the local energy with field shifted by its height average\footnote{Recall that $\overline{h}_t$ is the solution of the over-relaxed problem introduced in Section~\ref{sec:ORP}.} (corresponding to a linear approximation of the magnetisation), i.e.\
\begin{align*}
	\dashint_{Q_{\ell,t}(a)} |\nabla'm| + \dashint_{Q_{\ell,t}(a)} \frac{1}{2} |h-\overline{h}_{t}|^2\,\mathrm{d}x.
\end{align*}
In order to obtain the desired local scaling law of the energy of a minimiser, we are therefore led to also study the decay of the cumulated field strength 
\begin{align*}
	n(Q_{\ell,t}(a)) \coloneqq \frac{1}{\ell} \sup_{Q_{\ell}'(a)} \left| \int_0^t h\,\mathrm{d}x_{d+1} \right| = \frac{t}{\ell} \sup_{Q_{\ell}'(a)} |\overline{h}_t|,
\end{align*}
which, combined with the orthogonality relation (Lemma~\ref{lem:orthogonality}), allows us to pass the local energy bound from the energy density shifted by $\overline{h}_t$ to the unshifted energy density
\begin{align}\label{eq:boundenergywrtfandn}
	e(Q_{\ell, t}(a))
	&= \dashint_{Q_{\ell,t}(a)} |\nabla'm| + \dashint_{Q_{\ell,t}(a)} \frac{1}{2} |h-\overline{h}_{t}|^2 \,\mathrm{d}x + \dashint_{Q_{\ell,t}(a)} \frac{1}{2}|\overline{h}_t|^2\,\mathrm{d}x \nonumber\\
	&\leq \dashint_{Q_{\ell,t}(a)} |\nabla'm| + \dashint_{Q_{\ell,t}(a)} \frac{1}{2} |h-\overline{h}_{t}|^2 \,\mathrm{d}x + \frac{1}{2}\frac{\ell^2}{t^2} n(Q_{\ell,t}(a))^2.
\end{align}

We start with the outer loop of the argument, which proves the upper bound on the local energy. The core ingredients are the one-step improvement results presented in Section~\ref{sec:onestep-inner} (in the interior) and Section~\ref{sec:onestep-boundary} (at the boundary), which drive the Campanato-type iterations in Section~\ref{sec:iterations}.
\begin{proof}[Proof of Theorem~\ref{thm:localbounds} (Upper bound)]
Let $C_{LT}< \infty$ (to be chosen later). 
We start in a cuboid $Q_{L,T}$ with $L \geq C_{LT} T^{\frac{2}{3}}$ and let $(m,h)\in \mathcal{A}^{\mathrm{per}/0}_{Q_{L,T}}$ be a minimiser of \eqref{eq:energy-nonconvex} with top and bottom boundary conditions $m^B = m^T \equiv 0$. By Theorem~\ref{thm:globalbound} there exists a constant $C_{S}<\infty$ such that the energy density satisfies
\begin{align*}
	e(Q_{L,T}) \leq C_{S} T^{-\frac{2}{3}}.
\end{align*}

In order to transfer the energy scaling from the large scale to smaller cuboids at the sample boundary, we run an iteration based on a one-step improvement for the renormalised
energy density 
\begin{align*}
	f(Q_{\ell,t}(a)) \coloneqq \frac{t^2}{\ell^2}\left( \dashint_{Q_{\ell,t}(a)} |\nabla'm| + \frac{1}{2} \dashint_{Q_{\ell,t}(a)} (h-\overline{h}_t)^2\,\mathrm{d}x \right).
\end{align*}
We will also denote $f_0(Q_{\ell,t}(a)) \coloneqq \frac{t^2}{\ell^2} e(Q_{\ell,t}(a))$. 
Note that if $e(Q_{L,T}) \leq C_{S} T^{-\frac{2}{3}}$, then 
\begin{equation}\label{eq:globalboundenergyf_0}
f_0(Q_{L,T}) \leq C_{S} \frac{T^{\frac{4}{3}}}{L^2} \leq C_S C_{LT}^{-2} \leq \epsilon
\end{equation}
for any $\epsilon >0$ if $C_{LT}$ is chosen large enough. 
Therefore $f_0$ (and $f$, as we will see in the next step) provides a good (non-dimensional) small quantity which is suitable for an iteration down to smaller scales. For this iteration it is convenient to work with cuboids $Q_{\tilde{L},T}$ such that $\tilde{L} = c_{LT} T^{\frac{2}{3}}$ with $\frac{C_{LT}}{8} \leq c_{LT} \leq \frac{C_{LT}}{4}$.
Thus, the first step in our proof consists of transferring the energy scaling from $Q_{L,T}$ to $Q_{\tilde{L},T}$ (note that in this step the height of the cuboid remains fixed).

\begin{enumerate}[label=\textsc{\bf Step \arabic*},leftmargin=0pt,labelsep=*,itemindent=*,itemsep=10pt,topsep=10pt]
	\item\label{step:local-strategy-step-1} Choose $C_{LT}$ so large that the assumption of Proposition~\ref{prop:iteration-A} holds. Then there exists a constant $c_{LT}$ with the property that $\frac{C_{LT}}{8} \leq c_{LT} \leq \frac{C_{LT}}{4}$ and such that for $\tilde{L} = c_{LT} T^{\frac{2}{3}}$ there holds
	\begin{align*}
		f_0(Q_{\tilde{L},T}(a))  \lesssim f_0(Q_{L,T}) + \frac{T^{\frac{4}{3}}}{\tilde{L}^2},
	\end{align*}
	for any $a \in \RR^d \times\{0\}$ satisfying $Q_{\tilde{L},T}(a) \subseteq Q_{\frac{3}{4}L,T}$. 
	
	\item\label{step:local-strategy-step-2} 
	We now consider the cumulated field $H(x')\coloneqq \int_0^T h\,\mathrm{d}x_{d+1}$, which by the first equation in \eqref{eq:admissible} satisfies $\nabla' \cdot H = -\int_0^T \partial_{d+1} m \,\mathrm{d}x_{d+1} = - (m^T-m^B) = 0$ with periodic or zero-flux boundary conditions on $\partial Q_{L}'$. 
	Since an energy-minimising $h$ is a gradient field, $h=-\nabla' u$ (see Remark~\ref{rem:gradient}), we have $H(x') = -\nabla' \left(\int_0^T u \,\mathrm{d}x_{d+1}\right) = -T \nabla' \overline{u}_T$, hence the function $\overline{u}_T$ solves $-\Delta' \overline{u}_T = 0$ on $Q_{L}'$ with periodic or zero-flux boundary conditions. 
	But that means $\overline{u}_T$ is constant, implying that the cumulated field $H \equiv 0$ on $Q_{L}'$, in particular $n(Q_{L,T}) = 0$ and $n(Q_{\tilde{L},T}(a)) = 0$. 

	Hence, we have that 
	\begin{align}\label{eq:local-strategy-step-2f}
		f(Q_{\tilde{L},T}(a)) = f_0(Q_{\tilde{L},T}(a)) \lesssim f_0(Q_{L,T}) + \frac{T^{\frac{4}{3}}}{\tilde{L}^2},
	\end{align}
	and therefore, inserting \eqref{eq:globalboundenergyf_0}, we get $f(Q_{\tilde{L},T}(a)) \leq \epsilon$ for any $\epsilon >0$ if $C_{LT}$ is chosen large enough.
	Due to good screening properties, the quantity $f$ has the right decay on smaller length scales: for any $\tilde{\ell} \leq \frac{\tilde{L}}{8}$ and $t\leq T$ such that $\tilde{\ell} = c_{LT} t^{\frac{2}{3}}$, and $b$ such that $Q_{\tilde{\ell},t}(b) \subseteq Q_{\frac{3}{4}\tilde{L},T}(a)$ there holds
	\begin{align*}
		f(Q_{\tilde{\ell},t}(b)) 
		&\lesssim f(Q_{\tilde{L},T}(a)) + \frac{T^{\frac{4}{3}}}{\tilde{L}^2} \quad \text{and} \\ 
		n(Q_{\tilde{\ell},t}(b)) 
		&\lesssim n(Q_{\tilde{L},T}(a)) + \left(f(Q_{\tilde{L},T}(a)) + \frac{T^{\frac{4}{3}}}{\tilde{L}^2}\right)^{\frac{1}{d+2}} 
		\lesssim \left(f(Q_{\tilde{L},T}(a)) + \frac{T^{\frac{4}{3}}}{\tilde{L}^2}\right)^{\frac{1}{d+2}};
	\end{align*}
	see Proposition~\ref{prop:iteration-B} (increasing $C_{LT}$, hence $c_{LT}$, if necessary).
	
	Combining this with \eqref{eq:local-strategy-step-2f}, recalling that $f(Q_{L,T})=f_0(Q_{L,T})$, we are led to
	\begin{align*}
		f(Q_{\tilde{\ell},t}(b)) \lesssim f(Q_{L,T}) + \frac{T^{\frac{4}{3}}}{\tilde{L}^2}, \quad \text{and} \quad 
		n(Q_{\tilde{\ell},t}(b)) \lesssim \left( f(Q_{L,T}) + \frac{T^{\frac{4}{3}}}{\tilde{L}^2} \right)^{\frac{1}{d+2}},
	\end{align*}
	for any $b$ such that $Q_{\tilde{\ell},t}(b) \subseteq Q_{\frac{3}{4}\tilde{L},T}(a)$.

	\item
	By a covering argument, we can then extend the estimate from \ref{step:local-strategy-step-2} to cuboids $Q_{\ell,t}(b)$ with $\ell \geq c_{LT} t^{\frac{2}{3}}$:
	we divide $Q_{\ell, t}(b)$ into $J=\left(\frac{\ell}{\tilde{\ell}}\right)^d$ disjoint cuboids $Q_{\tilde{\ell},t}(b_j)$, $j=1, \dots, J$ of lateral side length $\tilde{\ell} = c_{LT} t^{\frac{2}{3}}$ centred at points $b_j$ on the bottom boundary of $Q_{L,T}$ (w.l.o.g.\ the constant $c_{LT}$ is chosen such that $J\in\NN$); see Figure~\ref{picture-strategy}. 	
	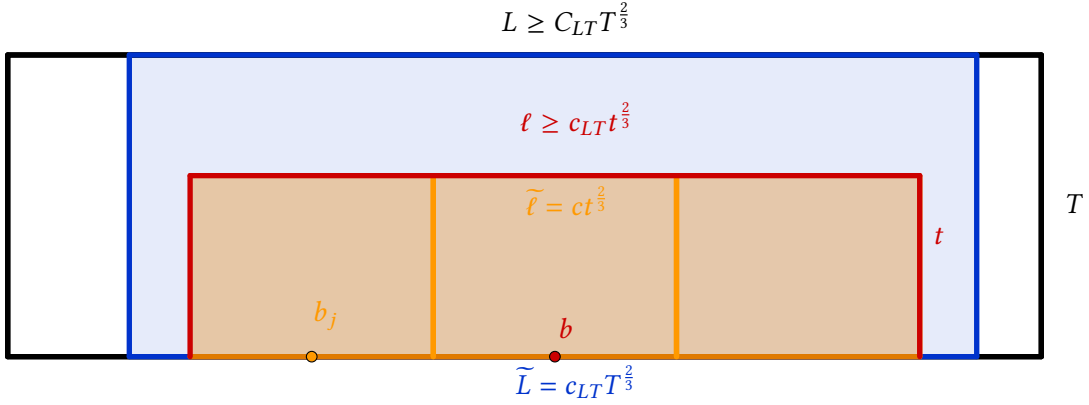
\begin{figure}[ht]
	\centering
	\definecolor{ffzzqq}{rgb}{1,0.6,0}
\definecolor{ccqqqq}{rgb}{0.8,0,0}
\definecolor{zzttqq}{rgb}{0.6,0.2,0}
\definecolor{qqttcc}{rgb}{0,0.2,0.8}
\definecolor{cqcqcq}{rgb}{0.7529411764705882,0.7529411764705882,0.7529411764705882}
\begin{tikzpicture}[line cap=round,line join=round,>=triangle 45,scale=.8]
\draw[line width=2pt] (-9,0) -- (8,0) -- (8,5) -- (-9,5) -- cycle;
\fill[line width=2pt,color=qqttcc,fill=qqttcc,fill opacity=0.1] (-7,5) -- (-7,0) -- (6.9375,0) -- (6.9375,5) -- cycle;
\fill[line width=2pt,color=zzttqq,fill=zzttqq,fill opacity=0.10000000149011612] (-6,3) -- (-6,0) -- (6,0) -- (6,3) -- cycle;
\fill[line width=2pt,color=ffzzqq,fill=ffzzqq,fill opacity=0.26] (-6,3) -- (-6,0) -- (-2,0) -- (-2,3) -- cycle;
\fill[line width=2pt,color=ffzzqq,fill=ffzzqq,fill opacity=0.25] (-2,3) -- (-2,0) -- (2,0) -- (2,3) -- cycle;
\fill[line width=2pt,color=ffzzqq,fill=ffzzqq,fill opacity=0.24] (2,0) -- (6,0) -- (6,3) -- (2,3) -- cycle;
\draw [line width=2pt] (-9,0)-- (8,0);
\draw [line width=2pt] (8,0)-- (8,5);
\draw [line width=2pt] (8,5)-- (-9,5);
\draw [line width=2pt] (-9,5)-- (-9,0);
\draw [line width=2pt,color=qqttcc] (-7,5)-- (-7,0);
\draw [line width=2pt,color=qqttcc] (-7,0)-- (6.9375,0);
\draw [line width=2pt,color=qqttcc] (6.9375,0)-- (6.9375,5);
\draw [line width=2pt,color=qqttcc] (6.9375,5)-- (-7,5);
\draw [line width=2pt,color=zzttqq] (-6,0)-- (6,0);
\draw [line width=2pt,color=ffzzqq,opacity=.7] (-6,0)-- (6,0);
\draw [line width=2pt,color=ffzzqq] (-2,0)-- (-2,3);
\draw [line width=2pt,color=ffzzqq] (-2,3)-- (-2,0);
\draw [line width=2pt,color=ffzzqq] (2,0)-- (2,3);
\draw [line width=2pt,color=ffzzqq] (6,0)-- (6,3);
\draw [line width=2pt,color=ffzzqq] (2,3)-- (2,0);
\draw [line width=2pt,color=ccqqqq] (-6,3)-- (-6,0);
\draw [line width=2pt,color=ccqqqq] (6,0)-- (6,3);
\draw [line width=2pt,color=ccqqqq] (6,3)-- (-6,3);
\draw (-1.0380305529055522,6.047232765719198) node[anchor=north west] {$L \geq C_{LT} T^{\frac{2}{3}}$};
\draw (8.240639511623115,2.8711598829954084) node[anchor=north west] {$T$};
\draw (6.062760963469688,2.304004011080446) node[anchor=north west] {{\color{ccqqqq}$t$}};
\draw (-0.81116820413957,0.03538052342059554) node[anchor=north west] {{\color{qqttcc}$\widetilde{L} = c_{LT} T^{\frac{2}{3}}$}};
\draw (-0.7431094995097755,4.391137619727508) node[anchor=north west] {{\color{ccqqqq}$\ell \geq c_{LT} t^{\frac{2}{3}}$}};
\draw (-0.65,3) node[anchor=north west] {{\color{ffzzqq}$\widetilde{\ell}= c t^{\frac{2}{3}}$}};
\draw [fill=ffzzqq] (-4,0) circle (2.5pt);
\draw[color=ffzzqq] (-3.7603787380973377,0.6819382174036528) node {$b_j$};
\draw [fill=ccqqqq] (0,0) circle (2.5pt);
\draw[color=ccqqqq] (0.18702613043075117,0.47776210351426635) node {$b$};
\end{tikzpicture}
	\caption{Depiction of the different cuboids and scales used throughout the proof.}
	\label{picture-strategy}
	\end{figure}

	In the boxes $Q_{\tilde{\ell},t}(b_j)$ we can apply \ref{step:local-strategy-step-2}, from which it follows that	
	\begin{align*}
		f(Q_{\ell, t}(b)) 
		&= \frac{t^2}{\ell^2} \frac{1}{\ell^d t} \sum_{j=1}^J \left( \int_{Q_{\tilde{\ell},t}(b_j)} |\nabla'm| + \frac{1}{2} \int_{Q_{\tilde{\ell},t}(b_j)} |h-\overline{h}_t|^2\,\mathrm{d}x \right) 
		= \left(\frac{\tilde{\ell}}{\ell} \right)^{d+2} \sum_{j=1}^J f(Q_{\tilde{\ell},t}(b_j)),
	\end{align*}
	and using that $f(Q_{\tilde{\ell},t}(b_j)) \lesssim f(Q_{L,T}) + \frac{T^{\frac{4}{3}}}{\tilde{L}^2}$, we obtain 
	\begin{align}\label{eq:localbound-f}
		f(Q_{\ell, t}(b)) 
		\lesssim \left(\frac{\tilde{\ell}}{\ell} \right)^{d+2} J \left( f(Q_{L,T}) + \frac{T^{\frac{4}{3}}}{\tilde{L}^2} \right)
		\lesssim \left(\frac{\tilde{\ell}}{\ell} \right)^{2} \left( f(Q_{L,T}) + \frac{T^{\frac{4}{3}}}{\tilde{L}^2} \right).
	\end{align}

	Similarly, we may estimate
	\begin{align}\label{eq:localbound-n}
		n(Q_{\ell, t}(b)) 
		= \frac{t}{\ell} \sup_{j} \sup_{Q_{\tilde{\ell}}'(b_j)} | \overline{h}_t| 
		= \frac{\tilde{\ell}}{\ell} \sup_j n(Q_{\tilde{\ell},t}(b_j)) 
		\lesssim \frac{\tilde{\ell}}{\ell} \left( f(Q_{L,T}) + \frac{T^{\frac{4}{3}}}{\tilde{L}^2} \right)^{\frac{1}{d+2}}.
	\end{align}

	\item Finally, recalling \eqref{eq:boundenergywrtfandn}, the local energy bound follows for any cuboid $Q_{\ell,t}(b)\subseteq Q_{\frac34 \tilde L,T}(a)$ (with $\ell \geq c_{LT} t^{\frac{2}{3}}$ and $a$ such that $Q_{\tilde L,T}(a)\subseteq Q_{\frac34 L,T}$), since 
	\begin{align*}
		e(Q_{\ell,t}(b)) 
		\leq \frac{\ell^2}{t^2} f(Q_{\ell,t}(b)) +  \frac{\ell^2}{t^2} n(Q_{\ell,t}(b))^2,
	\end{align*}
	so using \eqref{eq:localbound-f}, \eqref{eq:localbound-n}, and the fact that $f(Q_{L,T}) = f_0(Q_{L,T}) \leq C_S \frac{T^{\frac{4}{3}}}{L^2}$,
	we obtain with $\tilde{\ell} = c_{LT} t^{\frac{2}{3}}$ and $\tilde{L} = c_{LT} T^{\frac{2}{3}}$, that 
	\begin{align*}
		e(Q_{\ell,t}(b)) 
		&\lesssim \frac{\tilde{\ell}^2}{t^2} \left( f(Q_{L,T}) + \frac{T^{\frac{4}{3}}}{\tilde{L}^2} \right) + \frac{\tilde{\ell}^2}{t^2} \left(f(Q_{\tilde{L},T}) + \frac{T^{\frac{4}{3}}}{\tilde{L}^2} \right)^{\frac{2}{d+2}} 
		\lesssim c_{LT}^2 t^{-\frac{2}{3}} \left( \frac{T^{\frac{4}{3}}}{\tilde{L}^2} \right)^{\frac{2}{d+2}} 
		\lesssim c_{LT}^{\frac{2d}{d+2}} t^{-\frac{2}{3}}.
	\end{align*}
	
	\item \emph{(The role of the lateral boundary conditions).} This proves the upper bound in the theorem for small cuboids contained in $Q_{\frac{3}{4} \tilde{L},T}$. We now give an argument how to extend this to any cuboid contained in $Q_{L,T}$ sitting at the bottom boundary.

	If $(m,h)$ satisfies periodic boundary conditions, then it also minimises the energy in $Q_{2L,T}$. Therefore, applying \ref{step:local-strategy-step-1} of the proof, we obtain 
	$$
	f_0(Q_{\tilde{L},T}(a))  \lesssim f_0(Q_{2L,T}) + \frac{T^{\frac{4}{3}}}{\tilde{L}^2}\lesssim f_0(Q_{L,T}) + \frac{T^{\frac{4}{3}}}{\tilde{L}^2},
	$$
	for any $a \in \RR^d \times\{0\}$ such that $Q_{\tilde{L},T}(a) \subseteq Q_{\frac{3}{2}L,T}$ (and in particular for any $a$ such that $Q_{\tilde{L},T}(a)\subseteq Q_{L,T}$). We can then apply \ref{step:local-strategy-step-2}, which holds for any $b$ such that $Q_{\tilde \ell,t}(b)\subseteq Q_{\frac34\tilde L,T}(a)$. Combining the two, we deduce that \ref{step:local-strategy-step-2} holds for any $b$ such that $Q_{\tilde \ell,t}(b)\subseteq Q_{L,T}$. Hence, the local energy upper bound is extended to any cuboid contained in $Q_{L,T}$ sitting at the bottom boundary.

	\medskip
	If $(m,h)$ satisfies zero-flux boundary conditions, then by arguing as in the proof of the global scaling laws, we can obtain a new configuration $(\widetilde m,\widetilde h)$, which is periodic in $Q_{2L,T}$ and that has the same energy up to a dimensional constant. More precisely, we can extend $(m,h)$ to  $[-L,3L]^d\times [0,T]$ by a sequence of $d$ even reflections of $m$ and corresponding reflections of $h$ across each face of $\Gamma_{L,T}$ to obtain an admissible configuration in $\mathcal{A}^{\mathrm{per}}_{Q_{[-L,3L]^d\times [0,T]}}$. Finally, by translating the extended configuration to $Q_{[-2L,2L]^d\times [0,T]}$, we get a periodic pair $(\widetilde m,\widetilde h)$ in $\mathcal{A}^{\mathrm{per}}_{Q_{2L,T}}$ that satisfies 
	$$
	E_{Q_{2L,T}}(\widetilde{m},\widetilde{h})=2^d E_{Q_{L,T}}(m,h) \lesssim L^dT^\frac13.
	$$
	Since $h$ is a gradient field, so is $\widetilde{h}$, therefore we deduce that $(\widetilde m,\widetilde h)$ is an almost-minimising configuration at scale $(2L,T)$. 
	Hence, \ref{step:local-strategy-step-1} (which still holds true for almost-minimisers at scale $(2L,T)$) implies that for $C_{LT}$ large enough we can find a constant $\frac{C_{LT}}{8} \leq c_{LT} \leq \frac{C_{LT}}{4}$ such that for $\tilde{L} = c_{LT} T^{\frac{2}{3}}$ we have 
	\begin{align*}
		f_0(Q_{\tilde{L},T}(a)) \lesssim \widetilde{f}_0(Q_{2L,T}) + \frac{T^{\frac{4}{3}}}{\tilde{L}^2} \lesssim f_0(Q_{L,T}) + \frac{T^{\frac{4}{3}}}{\tilde{L}^2},
	\end{align*}
	for any $a \in Q_{L}' \times\{0\}$ such that $Q_{\tilde{L},T}(a) \subseteq Q_{\frac32 L,T}$. 
	However, at this stage we do not know whether $(\widetilde m,\widetilde h)$ is almost-minimising on smaller scales, so that when iterating down to smaller cuboids with Proposition~\ref{prop:iteration-B}, we have to leave some space to the lateral boundary $\Gamma_{L,T}$. Therefore, in contrast to the periodic case, we can benefit from this estimate only for $a$'s such that $Q_{\tilde{L},T}(a) \subseteq Q_{L,T}$.

	\medskip 
	Applying \ref{step:local-strategy-step-2}, we infer that the outcome of that step holds for any $b$ such that $Q_{\tilde \ell,t}(b)\subseteq Q_{\frac34\tilde L,T}(a)$, for any $a$ such that $Q_{\tilde{L},T}(a)\subseteq Q_{L,T}$. Moreover, as a careful look at the proof reveals, the factor $\frac{3}{4}$ was arbitrarily chosen (for concreteness purposes), and can actually be replaced by any factor $\gamma<1$. This reflects the fact that the argument holds for any cuboid located at a positive distance from $\Gamma_{L,T}$, leaving enough space to benefit from interior regularity.

	Hence, it only remains to extend the previous local energy bounds to cuboids centred at the lateral boundary $\Gamma_{L,T}$. In order to do so, we appeal to an iteration at the boundary. Letting $Q^*_{\ell,t}(b) \coloneqq Q_{\ell,t}(b) \cap Q_{L,T}$ be a cuboid at the boundary of $Q_{L,T}$, centred in a point $b\in \Gamma_{L,T} \cap \{x_{d+1}=0\}$, and arguing almost verbatim as in \ref{step:local-strategy-step-2}, replacing the use of Proposition \ref{prop:iteration-B} by Proposition \ref{prop:iteration-C}, we deduce that 
	\begin{align*}
		f(Q^*_{\tilde{\ell},t}(b)) \lesssim f(Q_{L,T}) + \frac{T^{\frac{4}{3}}}{\tilde{L}^2}, \quad \text{and} \quad 
		n(Q^*_{\tilde{\ell},t}(b)) \lesssim \left( f(Q_{L,T}) + \frac{T^{\frac{4}{3}}}{\tilde{L}^2} \right)^{\frac{1}{d+2}},
	\end{align*}
	for any $b\in \Gamma_{L,T} \cap \{x_{d+1}=0\}$. 

	Thus, combining the estimates on cuboids located in the bulk of $Q_{L,T}$ with the ones centred on $\Gamma_{L,T}\cap \{x_{d+1}=0\}$, we deduce that \ref{step:local-strategy-step-2} holds for any $b$ such that $Q_{\tilde \ell,t}(b)\subseteq Q_{L,T}$. Hence, the local energy upper bound is extended to any cuboid contained in $Q_{L,T}$ sitting at the bottom boundary.

	This concludes the proof. \qedhere
\end{enumerate}
\end{proof}

The rest of this section is devoted to the proof of the one-step improvement results and the iteration to smaller scales.  
In our forthcoming article \cite{RR23} we will give a simpler proof by a closer study of the convex relaxation of the problem, which is based on duality and avoids some of the more complicated constructions (especially the boundary layer construction). It also shows from a somewhat different angle that the shift of the energy by $\overline{h}_t$ is quite natural in view of the convex relaxation.\footnote{In fact, $\overline{h}_t$ is the solution to the \emph{over-relaxed} problem introduced in Section~\ref{sec:ORP}, which makes this approach similar to \cite{BJO22}.}

\subsubsection{One-step improvements in the interior}\label{sec:onestep-inner}
We start with a first one-step improvement result, which allows us to go to smaller widths $\ell \leq L$, while keeping the height $T$ fixed. This one-step improvement result is somewhat simpler, because it does not need decay of the cumulated field encoded in the quantity $n$ defined above. However, this will be a crucial ingredient in the second one-step improvement result.

\begin{lemma}[One-step improvement -- Version A for fixed height]\label{lem:onestep-1}
	Let $(m,h)$ be a minimiser of the non-convex energy \eqref{eq:energy-nonconvex} in $Q_{L,T}$ with periodic or zero-flux lateral boundary conditions and $m^{B} = m^T \equiv 0$. 
	
	For any $\theta \in (0,\frac{1}{2}]$ there exists an $\epsilon>0$ such that the following holds: 
	If $\ell\leq L$ is such that $f_0(Q_{\ell, T})\leq \epsilon$,
	then 
	\begin{align}
		f_0(Q_{\theta\ell, T}(a)) - \theta f_0(Q_{\ell, T}) &\lesssim_{\theta} \frac{T^{\frac{4}{3}}}{\ell^2} \label{eq:osi-f-1}
	\end{align}
	for any $a \in \RR^d \times \{0\}$ with $|a|\leq (1-2\theta) \ell$ (so that $Q_{2\theta\ell, T}(a) \subseteq Q_{\ell,T}$).
\end{lemma}

\begin{remark}
We have the followings observations: 

\begin{enumerate}
	\item In order to iterate this one-step improvement, \eqref{eq:osi-f-1} shows that we need $\ell \gg T^{\frac{2}{3}}$, which is consistent with the prediction that branching should be observed on cuboids at the sample boundary that are wide enough to capture the oscillations of the magnetisation. 	
	More precisely, fix $\theta\in(0,\frac{1}{2}]$ and let $f_0(Q_{L,T}) \leq \epsilon$ with the $\epsilon$ given by Lemma~\ref{lem:onestep-1}. Further, let $c_{LT} \geq (2 C_{\theta}\epsilon^{-1})^{\frac{1}{2}}$, where $C_{\theta}$ is the implicit constant in \eqref{eq:osi-f-1}. Then  
	\begin{align*}
		f_0(Q_{\theta\ell, T}(a)) 
		\leq \theta f_0(Q_{L,T}) + C_{\theta} \frac{T^{\frac{4}{3}}}{\ell^2} 
		\leq \theta \epsilon + C_{\theta} c_{LT}^{-2}
		\leq \frac{\epsilon}{2} + \frac{\epsilon}{2} \leq \epsilon,
	\end{align*} 
	hence we can iterate the one-step improvement. 
	\item As the proof of Lemma~\ref{lem:onestep-1} reveals, the assumption $m^T \equiv 0$ can be replaced by the assumption $\dashint_0^T h\,\mathrm{d}x_{d+1} \lesssim T^{-\frac{1}{3}}$ on the cumulated field.
	\item Since $a\in\RR^d \times \{0\}$ is required to be such that $Q_{2\theta \ell,T}(a) \subseteq Q_{L,T}$, this implies that $Q_{\theta \ell, T}(a) \subseteq Q_{(1-\theta)L,T}$, in particular, the smaller cuboid has to stay away from the boundary $\Gamma_{L,T}$.
\end{enumerate}
\end{remark}

\begin{lemma}[One-step improvement -- Version B for fixed aspect ratio]\label{lem:onestep-2}
	Let $(m,h)$ be a minimiser of the non-convex energy \eqref{eq:energy-nonconvex} in $Q_{L,T}$ with periodic or zero-flux lateral boundary conditions,  $m^{B} \equiv 0$, and $|m^{T}|\leq 1$. 
	
	For any $\theta \in (0, \frac{1}{3}]$ and $A<\infty$, there exist constants $\delta, \epsilon \in (0,1)$ with the property that
	\begin{align}\label{eq:epsilon-delta-relation}
		\epsilon \leq \frac{\delta^{d+2}}{A},
	\end{align}
	such that the following holds: 

	\noindent If $\ell \leq L$ and $t \leq T$ are such that
	\begin{align*}
		f(Q_{\ell, t})\leq \epsilon 
		\quad \text{and} \quad
		n(Q_{\ell, t})\leq \delta,
	\end{align*}
	then for any $a \in \RR^d \times \{0\}$ with $|a|\leq (1-3\theta) \ell$ (so that $Q_{3\theta\ell, t}(a) \subseteq Q_{\ell,t}$) there holds
	\begin{align}
		f(Q_{\theta \ell, \theta^{\frac{3}{2}} t}(a)) - \theta f(Q_{\ell, t}) &\lesssim_\theta \frac{t^{\frac{4}{3}}}{\ell^2}, \label{eq:osi-f}\\
		n(Q_{\theta \ell, \theta^{\frac{3}{2}} t}(a)) - \theta^{\frac{1}{2}} n(Q_{\ell, t}) &\lesssim_\theta f(Q_{\ell, t})^{\frac{1}{d+2}}. \label{eq:osi-n}
	\end{align} 
\end{lemma}

Note that we are going down with a different exponent in the $x'$ and $x_{d+1}$ directions. This ensures that the ratio $t^{\frac{2}{3}} \ell^{-1}$ stays invariant when going down to smaller scales and later ensures the right scaling relation on the smaller scales. 
\begin{remark}\label{rem:onestep-2}
	Fixing $\theta \in (0, \frac{1}{4}]$, we can iterate this one-step improvement. Indeed, let $f(Q_{\ell, t})\leq \epsilon$ and $n(Q_{\ell, t})\leq \delta$ with $\epsilon$ and $\delta$ from Lemma~\ref{lem:onestep-2} (which are fixed by the choice of $\theta$). If $C_n, C_f <\infty$ denote the implicit constants in \eqref{eq:osi-n} and \eqref{eq:osi-f}, then
	\begin{align*}
		f(Q_{\theta \ell, \theta^{\frac{3}{2}} t}(a)) 
		\leq \theta f(Q_{\ell, t}) + C_f \frac{t^{\frac{4}{3}}}{\ell^2}
		\leq \frac{\epsilon}{4} + C_f \frac{t^{\frac{4}{3}}}{\ell^2},
	\end{align*}
	and
	\begin{align*}
		n(Q_{\theta \ell, \theta^{\frac{3}{2}} t}(a)) 
		\leq \theta^{\frac{1}{2}} n(Q_{\ell, t}) + C_n f(Q_{\ell, t})^{\frac{1}{d+2}} 
		\leq \frac{\delta}{2} + C_n \epsilon^{\frac{1}{d+2}},
	\end{align*}
	as long as $a \in \RR^d \times \{0\}$ is such that $Q_{3\theta\ell, t}(a) \subseteq Q_{\ell,t}$. 
	By \eqref{eq:epsilon-delta-relation} for the choice $A = 2 C_n$ it then follows that $n(Q_{\theta \ell, \theta^{\frac{3}{2}} t}(a)) \leq \delta$. Hence, if $C_f \frac{t^{\frac{4}{3}}}{\ell^2} \leq \frac{3\epsilon}{4}$, which is guaranteed if $\ell \geq C_{\ell t} t^{\frac{2}{3}}$ with $C_{\ell t} \geq \left(\frac{4 C_f}{3\epsilon}\right)^{\frac{1}{2}}$, we also have $f(Q_{\theta \ell, \theta^{\frac{3}{2}} t}(a)) \leq \epsilon$. 
	Let us mention that the condition on $a$ implies that $Q_{\theta \ell, \theta^{\frac{3}{2}}t}(a) \subseteq Q_{(1-2\theta)\ell, \theta^{\frac{3}{2}} t}$, so that the iteration (Proposition~\ref{prop:iteration-B}) based on Lemma~\ref{lem:onestep-2} only gives information on $f$ and $n$ on smaller cuboids at the top/bottom boundary that have positive distance from $\Gamma_{L,T}$.
\end{remark}

Before giving the proof of Lemma~\ref{lem:onestep-1} and Lemma~\ref{lem:onestep-2}, let us start with some technical preliminaries that will be useful. The first one concerns the choice of a good cuboid with $\ell$ on which the boundary trace is well-behaved. The second one concerns a monotonicity formula for the renormalised energy density $f$, which seems to be new in this context and plays an important role in Lemma~\ref{lem:onestep-2}.

\begin{lemma}[Choice of good width]\label{lem:good} Let $h\in L^2(Q_{L}')$ for some $L>0$. Then the set of $\ell \in  \left[\frac{L}{2},L\right]$ such that 
	\begin{align*}
		\dashint_{\partial Q_{\ell}'}(h\cdot\nu)^2\,\mathrm{d}\mathcal{H}^{d-1} \lesssim \dashint_{Q_{L}'} |h|^2 \,\mathrm{d}x'
	\end{align*}
	has positive Lebesgue measure.
\end{lemma}

\begin{proof} 
	By Fubini's theorem, 
	\begin{align*}
		\int_{\frac{L}{2}}^{L} \int_{\partial Q_{\ell}'} (h\cdot\nu)^2\,\mathrm{d}\mathcal{H}^{d-1}\,\mathrm{d}\ell 
		= \int_{\left[\frac{L}{2},L\right]^{d}} (h\cdot \nu)^2\,\mathrm{d}x' 
		\leq \int_{Q_L} |h|^2\,\mathrm{d}x' < \infty,
	\end{align*}
 	so that $\int_{\partial Q_{\ell}'} (h\cdot\nu)^2\,\mathrm{d}\mathcal{H}^{d-1}$ is finite for Lebesgue-a.e.\ $\ell\in\left[\frac{L}{2},L\right]$. 
 	Given $A>1$, we let 
 	\begin{align*}
 		\Lambda \coloneq \left\{ {\textstyle \ell\in\left[\frac{L}{2}, L\right]:} \int_{\partial Q_{\ell}'} (h\cdot\nu)^2\,\mathrm{d}\mathcal{H}^{d-1} \leq \frac{2A}{L} \int_{Q_{L}'} |h|^2 \,\mathrm{d}x'\right\}.
 	\end{align*} 
	By Markov's inequality, for any $A>1$ it holds that $\left|\Lambda^c\right|
		\leq \frac{L}{2A}$, 
	hence 
	$\left|\Lambda \right| \geq \frac{L}{2}\left(1-\frac{1}{A}\right)>0$,
	and for all $\ell \in \Lambda$ we have 
	\begin{align*}
		\dashint_{\partial Q_{\ell}'}(h\cdot\nu)^2\,\mathrm{d}\mathcal{H}^{d-1} 
		\leq 2 A \left(\frac{L}{\ell}\right)^{d-1} \dashint_{Q_{L}'} |h|^2 \,\mathrm{d}x'	
		\leq 2^d A \dashint_{Q_{L}'} |h|^2 \,\mathrm{d}x'.
	\end{align*}
	The results thus follows.
\end{proof}

\begin{lemma}[Monotonicity formula]\label{lem:monotonicity}
	Let $h \in L^2(Q_{\ell,T})$ for some $\ell, T>0$. Then the map
	\begin{align*}
		t \mapsto \frac{t^2}{\ell^2} \dashint_{Q_{\ell,t}} (h-\overline{h}_t)^2\,\mathrm{d}x
	\end{align*}
	is non-decreasing in $[0,T]$.
\end{lemma}

\begin{proof}
	Since $h\in L^2(Q_{\ell,T})$, by Fubini's theorem $h(\cdot,t)$ is well-defined and finite for Lebesgue-a.e.\ $t>0$. 
	We calculate
	\begin{align*}
		\frac{\mathrm{d}}{\mathrm{d}t} \overline{h}_t 
		= \frac{\mathrm{d}}{\mathrm{d}t} \frac{1}{t} \int_0^t h\,\mathrm{d}x_{d+1}
		= \frac{1}{t}\left(h(\cdot, t) - \overline{h}_t\right),
	\end{align*}
	and obtain with this
	\begin{align*}
		&\frac{\mathrm{d}}{\mathrm{d}t} \frac{t^2}{\ell^2} \dashint_{Q_{\ell,t}} (h-\overline{h}_t)^2\,\mathrm{d}x 
		= \frac{1}{\ell^2} \frac{\mathrm{d}}{\mathrm{d}t} \left( t \int_0^t \dashint_{Q_{\ell}'} (h-\overline{h}_t)^2\,\mathrm{d}x'\,\mathrm{d}x_{d+1}\right)\\
		&\quad= \frac{1}{\ell^2} \left( \int_0^t \dashint_{Q_{\ell}'} (h-\overline{h}_t)^2\,\mathrm{d}x + t \dashint_{Q_{\ell}'} (h(\cdot, t)-\overline{h}_t)^2\,\mathrm{d}x' - 2 t^2 \dashint_{Q_{\ell,t}} (h-\overline{h}_t)\cdot \frac{\mathrm{d}}{\mathrm{d}t} \overline{h}_t\,\mathrm{d}x \right)\\
		&\quad= \frac{1}{\ell^2} \left( t \dashint_{Q_{\ell,t}} (h-\overline{h}_t)^2\,\mathrm{d}x + t \dashint_{Q_{\ell,t}} (h(\cdot, t)-\overline{h}_t)^2\,\mathrm{d}x - 2 t \dashint_{Q_{\ell,t}} (h-\overline{h}_t)\cdot (h(\cdot, t)-\overline{h}_t)\,\mathrm{d}x \right) \\
		&\quad= \frac{t}{\ell^2} \dashint_{Q_{\ell,t}} (h-h(\cdot,t))^2\,\mathrm{d}x 
		\geq 0,
	\end{align*}
	which proves the claim.
\end{proof}

The proofs of both Lemma~\ref{lem:onestep-1} and Lemma~\ref{lem:onestep-2} have the same first steps and only deviate at the end when it comes to estimating the cumulated field strength $n$. We therefore start with the common part first and conclude the proofs separately.
\begin{proof}[Proof of Lemma~\ref{lem:onestep-1} and Lemma~\ref{lem:onestep-2} (common part)]
Let $\theta$ and $a \in \RR^d \times \{0\}$ be as in the assumptions of Lemma~\ref{lem:onestep-1} and Lemma~\ref{lem:onestep-2}, respectively, and denote $\ltheta = \theta \ell$, $\ttau = \tau t$ for $\tau\leq 1$.
\begin{enumerate}[label=\textsc{\bf Step \arabic*},leftmargin=0pt,labelsep=*,itemindent=*,itemsep=10pt,topsep=10pt]
	\item \emph{(Choice of good width)}.
		Let $\rho\in [\theta\ell, 2\theta\ell]$ be a good width in the sense of Lemma~\ref{lem:good} (in an $x_{d+1}$-integrated version). Then for $g \coloneqq h\cdot \nu |_{\Gamma_{\rho,\ttau}(a)}$ and its height average $\overline{g}_{\ttau} = \int_0^{\ttau} g\,\mathrm{d}x_{d+1}$ there holds
		\begin{align*}
			\dashint_{\Gamma_{\rho,\ttau}(a)} (g-\overline{g}_{\ttau})^2 \,\mathrm{d}\mathcal{H}^{d-1}\,\mathrm{d}x_{d+1} 
			&\lesssim \dashint_{Q_{2\theta\ell, \ttau}(a)} (h-\overline{h}_{\ttau})^2 \,\mathrm{d}x 
			\lesssim \theta^{-d} \dashint_{Q_{\ell, \ttau}} (h-\overline{h}_{\ttau})^2 \,\mathrm{d}x,
		\end{align*}
		and applying the monotonicity formula from Lemma~\ref{lem:monotonicity}, we obtain  
		\begin{align*}
			\frac{\ttau^2}{\ell^2} \dashint_{Q_{\ell, \ttau}} (h-\overline{h}_{\ttau})^2 \,\mathrm{d}x
			\leq \frac{t^2}{\ell^2} \dashint_{Q_{\ell, t}} (h-\overline{h}_t)^2\,\mathrm{d}x 
			\leq f(Q_{\ell, t}),
		\end{align*}
		hence
		\begin{align}\label{eq:good-width-height}
			\frac{\ttau^2}{\rho^2} \dashint_{\Gamma_{\rho,\ttau}(a)} (g-\overline{g}_{\ttau})^2 \,\mathrm{d}\mathcal{H}^{d-1} \, \mathrm{d}x_{d+1}
			\lesssim \theta^{-(d+2)} f(Q_{\ell, t}).
		\end{align}

	\item \emph{(Construction of a competitor)}. 
		Let $(m_r,h_r)\in\mathcal{X}_r\left(Q_{\rho, \ttau}(a)\right)$ be the admissible pair constructed in Proposition~\ref{prop:fine-properties-relaxed} with $g = h\cdot \nu$, i.e., $h_r \cdot \nu = (h - \overline{h}_{\ttau})\cdot \nu$ on $\Gamma_{\rho, \ttau}(a)$ (extended by zero outside $A_r(Q_{\rho,\ttau}(a))\coloneqq (Q'_{\rho}(a)\setminus Q'_{\rho-r}(a)) \times[0,\ttau]$). Further, set $m^{\ttau} = -\nabla' \cdot \int_0^{\ttau} h\,\mathrm{d}x_{d+1}$. Then $|m^{\ttau}| \leq 1$ in $Q_{\ell}'$, and  $(m_{\mathrm{rel}}, h_{\mathrm{rel}})\coloneqq(\frac{x_{d+1}}{\ttau} m^{\ttau} + m_r, \overline{h}_{\ttau} + h_r) \in \mathcal{A}^{\mathrm{rel}}_{Q_{\rho,\ttau}(a)}$ is admissible for the relaxed problem in $Q_{\rho,\ttau}(a)$ with boundary conditions given by the non-convex minimiser $(m,h)$.
		Let us observe that 
		$$
		\partial_{d+1}m_{\mathrm{rel}}=\frac1{\ttau}m^{\ttau}+\partial_{d+1}m_r
		$$
		and therefore
		$$
		\ttau
		\|\partial_{d+1} m_{\mathrm{rel}}\|^2_{L^2_{x_{d+1}}L^{\infty}_{x'}(Q_{\rho,\ttau}(a))}\leq 1+\ttau\|\partial_{d+1} m_r\|^2_{L^2_{x_{d+1}}L^{\infty}_{x'}(A_{\rho,\ttau}(r;a))}
		\stackrel{\eqref{eq:boundary-layer-size}\&\eqref{eq:bound-d3mell}}{\lesssim}1.
		$$
		We can thus apply Corollary~\ref{cor:competitor}, which provides the existence of an admissible pair $(\widetilde{m}, \widetilde{h}) \in \mathcal{A}_{Q_{\rho,\ttau}(a)}(g;m^{B,T})$ such that \eqref{eq:bound-mtilde-competitor} and \eqref{eq:bound-htilde-competitor} hold in $Q_{\rho,\ttau}(a)$.
		
		We claim that
		\begin{align}\label{eq:htilde-average}
			\overline{\widetilde h}_{\ttau}=\dashint_{0}^{\ttau} \widetilde{h}\,\mathrm{d}x_{d+1} = \dashint_{0}^{\ttau} h\,\mathrm{d}x_{d+1}=\overline{h}_{\ttau}.
		\end{align}
		In fact, let us observe that 
		$$
		-\nabla'\cdot \overline{\widetilde h}_{\ttau}=-\nabla '\cdot \dashint_{0}^{\ttau} \widetilde{h}\,\mathrm{d}x_{d+1} =\dashint_{0}^{\ttau} \partial_{d+1}\widetilde m\,\mathrm{d}x_{d+1}=\frac{m^{\ttau}-m^B}{\ttau}.
		$$
		Moreover, on $\Gamma_{\rho,\ttau}(a)$ it holds that
		$$
		\dashint_{0}^{\ttau}\widetilde h\cdot \nu \,\mathrm{d}x_{d+1}=\dashint_{0}^{\ttau} g \,\mathrm{d}x_{d+1}=\dashint_{0}^{\ttau} h\cdot \nu \,\mathrm{d}x_{d+1}.
		$$
		By construction, $\widetilde h$ is a gradient field, hence $\widetilde h=-\nabla'\widetilde u$, where $\widetilde u$ is a solution to
		\begin{align*}
		\begin{array}{rll}
		-\Delta' u &= -\frac{m^{\ttau} - m^B}{\ttau} &\text{in} \quad  Q_{\rho,{\ttau}}(a), \\
		-\nabla' u \cdot \nu &= \overline{g}_{\ttau} &\text{on} \quad \Gamma_{\rho,\ttau}(a).
		\end{array}
		\end{align*}
		Finally, by observing that $u_0=-\nabla '\overline{h}_{\ttau}$ is also a solution, by uniqueness we deduce the validity of the claim.

		On the other hand, by (local) minimality of $(m,h)$ in $Q_{\rho,\ttau}(a)$, we have 
		\begin{align}\label{eq:competitor}
			\int_{Q_{\rho,\ttau}(a)} |\nabla' m| + \frac{1}{2} \int_{Q_{\rho,\ttau}(a)} h^2\,\mathrm{d}x 
			\leq \int_{Q_{\rho,\ttau}(a)} |\nabla' \widetilde{m}| + \frac{1}{2} \int_{Q_{\rho,\ttau}(a)} \widetilde{h}^2\,\mathrm{d}x.
		\end{align}
		
	\item \emph{(Orthogonality and optimality)}. 
		To bound $f(Q_{\ltheta, \ttau}(a))$ we start by estimating 
			\begin{align*}
				f(Q_{\ltheta, \ttau}(a)) 
				&= \frac{\ttau^2}{\ltheta^2} \frac{1}{\ltheta^d \ttau}  \left( \int_{Q_{\ltheta, \ttau}(a)} |\nabla'm| + \frac{1}{2} \int_{Q_{\ltheta, \ttau}(a)} (h-\overline{h}_{\ttau})^2 \right)\\
				&\leq \frac{\ttau^2}{\ltheta^2} \frac{1}{\ltheta^d \ttau}  \left( \int_{Q_{\rho,\ttau}(a)} |\nabla'm| + \frac{1}{2} \int_{Q_{\rho,\ttau}(a)} (h-\overline{h}_{\ttau})^2 \right).
		\end{align*}
		By Lemma~\ref{lem:orthogonality} we further have that 
		\begin{align*}
			\int_{Q_{\rho,\ttau}(a)} |\nabla'm| + \frac{1}{2} \int_{Q_{\rho,\ttau}(a)} (h-\overline{h}_{\ttau})^2
			&= \int_{Q_{\rho,\ttau}(a)} |\nabla'm| + \frac{1}{2} \int_{Q_{\rho,\ttau}(a)} h^2 - \frac{1}{2} \int_{Q_{\rho,\ttau}(a)} \overline{h}_{\ttau}^2 \,\mathrm{d}x \\
			&\stackrel{\eqref{eq:competitor}}{\leq} \int_{Q_{\rho,\ttau}(a)} |\nabla' \widetilde{m}| + \frac{1}{2} \int_{Q_{\rho,\ttau}(a)} \widetilde{h}^2\,\mathrm{d}x - \frac{1}{2} \int_{Q_{\rho,\ttau}(a)} \overline{h}_{\ttau}^2 \,\mathrm{d}x \\
			&\stackrel{\eqref{eq:htilde-average}}{=} \int_{Q_{\rho,\ttau}(a)} |\nabla' \widetilde{m}| + \frac{1}{2} \int_{Q_{\rho,\ttau}(a)} (\widetilde{h} - \overline{\widetilde h}_{\ttau})^2\,\mathrm{d}x,
		\end{align*}
		hence 
		\begin{align}\label{eq:boundf}
			f(Q_{\ltheta, \ttau}(a)) 
			&\leq \frac{\ttau^2}{\ltheta^2} \frac{1}{\ltheta^d \ttau} \left( \int_{Q_{\rho,\ttau}(a)} |\nabla' \widetilde{m}| + \frac{1}{2} \int_{Q_{\rho,\ttau}(a)} (\widetilde{h} - \overline{\widetilde h}_{\ttau})^2\,\mathrm{d}x \right) \nonumber\\
			&\stackrel{\eqref{eq:bound-mtilde-competitor}}{\lesssim} \frac{\ttau^2}{\rho^2} \left( \frac{N}{\rho} + \dashint_{Q_{\rho,\ttau}(a)} (\widetilde{h} - \overline{\widetilde h}_{\ttau})^2\,\mathrm{d}x \right).
		\end{align}
	\item \emph{(Estimate of the correction fields in the interior and in the boundary layer)}.
		Note that since $m^B \equiv 0$, we have $H^B \equiv 0$, so that by Proposition~\ref{prop:fine-properties-relaxed} we may bound
		\begin{align*}
			&\frac{\ttau^2}{\rho^2} \dashint_{Q_{\rho,\ttau}(a)} (\widetilde{h} - \overline{\widetilde h}_{\ttau})^2\,\mathrm{d}x 
			\lesssim \frac{\ttau^2}{\rho^2} \dashint_{Q_{\rho,\ttau}(a)} (\widetilde{h} - h_{\mathrm{rel}})^2\,\mathrm{d}x +  \frac{\ttau^2}{\rho^2}  \dashint_{Q_{\rho,\ttau}(a)} (h_{\mathrm{rel}} - \overline{\widetilde h}_{\ttau})^2\,\mathrm{d}x \\
			&\quad\stackrel{\eqref{eq:bound-htilde-competitor}}{\lesssim} \frac{\ttau^2}{\rho^2} \frac{\rho^2}{N^2 \ttau^2} + \frac{\ttau^2}{\rho^2} \frac{1}{\rho^d \ttau} \int_{A_{\rho,\ttau}(r;a)} h_r^2\,\mathrm{d}x \\
			&\quad\stackrel{\eqref{eq:bound-hell}\&\eqref{eq:boundary-layer-size}}{\lesssim} \frac{1}{N^2} + \left(\left(\frac{\ttau^2}{\rho^2} \dashint_{\Gamma_{\rho,\ttau}(a)} (g-\overline{g}_{\ttau})^2 \right)^{\frac{1}{d+1}} + \frac{\ttau}{\rho} \sup_{Q_{\rho}'(a)} \left|\dashint_0^{\ttau} h\,\mathrm{d}x_{d+1}\right| \right) \frac{\ttau^2}{\rho^2} \dashint_{\Gamma_{\rho,\ttau}(a)} (g-\overline{g}_{\ttau})^2 \\
			&\quad\stackrel{\eqref{eq:good-width-height}}{\lesssim} \frac{1}{N^2} + \left(\theta^{-(d+2)} f(Q_{\ell, t})\right)^{\frac{d+2}{d+1}} + n(Q_{\rho,\ttau}(a)) \theta^{-(d+2)} f(Q_{\ell, t}).
		\end{align*}
		Inserting this in \eqref{eq:boundf}, we are led to
		\begin{align*}
			f(Q_{\ltheta, \ttau}(a)) \lesssim \left( \frac{N \ttau^2}{\rho^3} + \frac{1}{N^2} + \left(\theta^{-(d+2)} f(Q_{\ell, t})\right)^{\frac{d+2}{d+1}} + n(Q_{\rho,\ttau}(a)) \theta^{-(d+2)} f(Q_{\ell, t}) \right).
		\end{align*}
		Recalling that $\ltheta = \theta \ell$, optimising in $N\in\NN$ (i.e. choosing $N=\lceil\ttau^{-\frac23}\rho\rceil$), we find with $\ltheta\leq \rho \leq 2 \ltheta$,
		\begin{align}\label{eq:estimateonf}
			f(Q_{\ltheta, \ttau}(a)) \lesssim \frac{\ttau^{\frac{4}{3}}}{\ltheta^2} + \left( \theta^{-(d+2)} f(Q_{\ell, t})\right)^{\frac{d+2}{d+1}} + n(Q_{\rho,\ttau}(a)) \theta^{-(d+2)} f(Q_{\ell, t}).
		\end{align}
\end{enumerate}
\begin{remark}
	Since minimality of $(m,h)$ in $Q_{L,T}$ is only used in \eqref{eq:competitor} (in the form of local minimality of $(m,h)$ in $Q_{\rho,\ttau}(a)$, it can be replaced by almost-minimality in $Q_{\ell,t}$ in the sense of Remark~\ref{rem:almost-min}. This leads to an extra term $C \frac{t^{\frac{4}{3}}}{\ell^2}$ in the estimate, and can therefore be absorbed in the right-hand side of \eqref{eq:osi-f-1} and \eqref{eq:osi-f}.
\end{remark}
At this point, the proofs of Lemma~\ref{lem:onestep-1} and Lemma~\ref{lem:onestep-2} deviate: to bound the cumulated field strength $n(Q_{\rho, \ttau}(a))$, we have to distinguish the cases $\tau = 1$ (for $t=T$) and $\tau = \theta^{\frac{3}{2}}$ (for $t<T$).

\begin{proof}[Proof of Lemma~\ref{lem:onestep-1}]
We proceed in two steps.

\begin{enumerate}[label=\textsc{\bf Step \arabic*.A},start=5,leftmargin=0pt,labelsep=*,itemindent=*,itemsep=10pt,topsep=10pt]
\item \emph{(Treatment of the cumulated field)}. 
	Note that for $t=T$ and $\tau = 1$, we have $\int_0^T h\,\mathrm{d}x_{d+1} = 0$, c.f.\ \ref{step:local-strategy-step-2} of the proof of Theorem~\ref{thm:localbounds}.
	In particular, $f(Q_{\ell, T})= f_0(Q_{\ell, T})$ and 
	\begin{align*}
		n(Q_{\rho,T}) = \frac{1}{\rho} \sup_{Q_{\rho}'(a)} \left|\int_0^T h\,\mathrm{d}x_{d+1}\right| = 0,
	\end{align*}
	which simplifies estimate \eqref{eq:estimateonf} to 
	\begin{align}\label{eq:onestep-1-final}
		f_0(Q_{\ltheta, T}(a)) 
		\lesssim \frac{T^{\frac{4}{3}}}{\ltheta^2} + \theta^{-(d+2)} \left(\theta^{-(d+2)}f_0(Q_{\ell, T})\right)^{\frac{1}{d+1}} f_0(Q_{\ell,T}).
	\end{align}
\item 
\emph{(Choice of $\epsilon$)}. Let $C>0$ be the implicit constant in \eqref{eq:onestep-1-final}. In view of \eqref{eq:osi-f-1} we choose 
\begin{align*}
	\epsilon = \theta^{d+2} \left(\frac{\theta^{d+3}}{C}\right)^{d+1}.
\end{align*} 
Then 
	\begin{align*}
		C \theta^{-(d+2)} \left(\theta^{-(d+2)} f_0(Q_{\ell,T})\right)^{\frac{1}{d+1}} 
		\leq C \theta^{-(d+2)} \left( \theta^{-(d+2)} \epsilon\right)^{\frac{1}{d+1}} 
		\leq \theta,
	\end{align*}
	from which it follows that  
	\begin{align*}
		f(Q_{\ltheta, T}(a)) - \theta f(Q_{\ell,T}) \lesssim \frac{T^{\frac{4}{3}}}{\ltheta^2}. 
	\end{align*}
	\qedhere
\end{enumerate}
\end{proof}
\begin{proof}[Proof of Lemma~\ref{lem:onestep-2}] We proceed in two steps.

\begin{enumerate}[label=\textsc{\bf Step \arabic*.B},start=5,leftmargin=0pt,labelsep=*,itemindent=*,itemsep=10pt,topsep=10pt]
\item\label{step:one-step-improvement5B} \emph{(Estimate on the cumulated field strength $n$)}. As we shall next show, for $\tau<1$, the cumulated field strength $n(Q_{\rho,\ttau}(a))$ decays w.r.t.\ the vertical direction. Let us observe that
	\begin{align}\label{eq:n-initial}
		n(Q_{\ltheta, \ttau}(a)) \leq \frac{\rho}{\ltheta} n(Q_{\rho, \ttau}(a)), 
	\end{align}
	so it suffices to estimate $n(Q_{\rho, \ttau}(a))$. 
		
	By the triangle inequality and since $Q_{\rho}'(a) \subseteq Q_{2\theta\ell}'(a) \subseteq Q_{\ell}'$, we have 
	\begin{align*}
		n(Q_{\rho, \ttau}(a))
		&\leq \frac{\ttau}{\rho} \sup_{Q_{\rho}'(a)} \left| \dashint_{0}^{t} h\,\mathrm{d}x_{d+1} \right| 
		+ \frac{\ttau}{\rho} \sup_{Q_{\rho}'(a)} \left| \dashint_{0}^{\ttau} h\,\mathrm{d}x_{d+1} - \dashint_{0}^{t} h\,\mathrm{d}x_{d+1} \right| \\
		&\leq \tau \frac{\ell}{\rho} n(Q_{\ell, t}) + \frac{\ttau}{\rho} \sup_{Q_{\rho}'(a)} \left| \dashint_{0}^{\ttau} h\,\mathrm{d}x_{d+1} - \dashint_{0}^{t} h\,\mathrm{d}x_{d+1} \right|.
	\end{align*}
	We claim that the latter term can be estimated by 
	\begin{align}\label{eq:CS-elliptic}
		\frac{\ttau}{\rho} \sup_{Q_{\rho}'(a)} \left| \dashint_{0}^{\ttau} h\,\mathrm{d}x_{d+1} - \dashint_{0}^{t} h\,\mathrm{d}x_{d+1} \right|
		\lesssim \frac{\ell}{\rho} (\tau f(Q_{\ell, t}))^{\frac{1}{d+2}},
	\end{align}
	thus
	\begin{align}\label{eq:n-intermidiate}
		n(Q_{\rho, \ttau}(a)) - \frac{\ell}{\rho} \tau n(Q_{\ell, t}) \lesssim \frac{\ell}{\rho}  (\tau f(Q_{\ell, t}))^{\frac{1}{d+2}}.
	\end{align}
    Before providing a proof for \eqref{eq:CS-elliptic}, let us observe that, by inserting \eqref{eq:n-intermidiate} in \eqref{eq:n-initial}, we find
	\begin{align}\label{eq:n-onestep-1}
		n(Q_{\ltheta, \ttau}(a)) - \frac{\tau}{\theta} n(Q_{\ell, t}) \lesssim \theta^{-1} (\tau f(Q_{\ell, t}))^{\frac{1}{d+2}}
	\end{align}
	and that, by plugging \eqref{eq:n-intermidiate} into \eqref{eq:estimateonf}, we obtain
	\begin{align}\label{eq:f-onestep-1}
		f(Q_{\ltheta, \ttau}(a)) &\lesssim \frac{\ttau^{\frac{4}{3}}}{\ltheta^2} + \theta^{-(d+2)} \left( (\theta^{-(d+2)} f(Q_{\ell, t}))^{\frac{1}{d+1}} + \frac{\tau}{\theta} n(Q_{\ell,t})  + \frac{1}{\theta} (\tau f(Q_{\ell, t}))^{\frac{1}{d+2}}\right) f(Q_{\ell, t}).
	\end{align}

	In order to prove \eqref{eq:CS-elliptic}, let us consider the function 
\begin{align*}
	H(x') \coloneqq \int_{0}^{\ttau} h\,\mathrm{d}x_{d+1} - \frac{\ttau}{t} \int_{0}^{t} h\,\mathrm{d}x_{d+1}
	= \int_0^{\ttau} ( h - \overline{h}_t)\,\mathrm{d}x_{d+1}, \quad x' \in Q_{\rho}'(a).
\end{align*}
By the Cauchy-Schwarz inequality, we have that 
\begin{align}\label{eq:boundH}
	|H(x')|^2 \leq \ttau \int_0^{\ttau} ( h - \overline{h}_t)^2\,\mathrm{d}x_{d+1} 
	\leq \ttau t \dashint_0^{t} ( h - \overline{h}_t)^2\,\mathrm{d}x_{d+1}.
\end{align}
We now apply interior elliptic regularity to estimate $\sup_{x' \in Q_{\rho}'} |H(x')|^2$. To this end, note that $H$ solves the equation 
\begin{align*}
	\nabla'\cdot H 
	= \nabla'\cdot \int_{0}^{\ttau}  h\,\mathrm{d}x_{d+1} - \frac{\ttau}{t} \int_{0}^{t} \nabla'\cdot h\,\mathrm{d}x_{d+1} 
	= -m^{\ttau} + \frac{\ttau}{t} \int_{0}^{t} \partial_{d+1}m\, \mathrm{d}x_{d+1} 
	= \frac{\ttau}{t}m^t - m^{\ttau} \eqqcolon M.
\end{align*}
Further, by minimality, $h$ is a gradient, so that also $H=-\nabla' U$ is a gradient field whose potential solves $-\Delta' U = M$, with $|M| \leq 2$. Let $x_*'$ be an arbitrary point in $Q_{\rho}'(a)$. Then, by elliptic regularity, we can bound 
\begin{align*}
	|H(x_*')|^2 \lesssim \dashint_{B_{\alpha\ell}(x_*')} |H|^2\,\mathrm{d}x' + (\alpha\ell)^2 \sup_{B_{\alpha\ell}(x_*')}|\nabla'\cdot H|^2 
	\stackrel{\eqref{eq:boundH}}{\lesssim} \alpha^{-d} \ttau t \dashint_{Q_{\ell,t}} ( h - \overline{h}_t)^2\,\mathrm{d}x + (\alpha\ell)^2,
\end{align*}
where the last inequality holds as long as $B_{\alpha\ell}(x_*') \subseteq Q_{\ell}'$, which is guaranteed for any $\alpha\leq\theta$, since $x_*' \in Q_{\rho}'(a) \subseteq Q_{2\theta\ell}'(a)$, hence $B_{\alpha\ell}(x_*') \subseteq Q_{3\theta\ell}'(a) \subseteq Q_{\ell}'$ by our assumption on $a$.  
Optimising in $\alpha$ then yields 
\begin{align}\label{eq:choice-alpha}
	\alpha = \left(\frac{\ttau t}{\ell^2} \dashint_{Q_{\ell,t}} ( h - \overline{h}_t)^2\,\mathrm{d}x \right)^{\frac{1}{d+2}} = \left( \tau f(Q_{\ell,t}) \right)^{\frac{1}{d+2}},
\end{align}
in particular the condition $\alpha \leq \theta$ is fulfilled if $\epsilon$ is chosen small enough (see the next step). From this, we deduce that
\begin{align*}
	\sup_{x' \in Q_{\rho}'(a)} |H(x')|^2 
	\lesssim \ell^2(\tau f(Q_{\ell, t}))^{\frac{2}{d+2}},
\end{align*}
which by definition of $H$ proves the claimed estimate \eqref{eq:CS-elliptic}.

\item 
\emph{(Choice of $\epsilon$ and $\delta$)}.
We can now fix the parameters $\epsilon$ and $\delta$ for $\tau=\theta^\frac32$.
	Let $C$ be the maximum of the implicit constants in inequalities \eqref{eq:n-onestep-1} and \eqref{eq:f-onestep-1}.  
	Given $\theta \in (0, \frac{1}{2}]$ we choose $\delta$ so small  that 
	\begin{align*}
		C \theta^{-(d+2)} \frac{\theta^{\frac{3}{2}}}{\theta} n(Q_{\ell, t}) \leq \frac{\theta}{3}, 	
	\end{align*}
	which is ensured by the choice $\delta = \frac{\theta^{d+\frac{5}{2}}}{3C}$. We then choose $\epsilon$ so small such that
	\begin{align*}
		C \theta^{-1} (\theta^{\frac{3}{2}}f(Q_{\ell, t}))^{\frac{1}{d+2}} \leq \frac{\delta}{2},
	\end{align*}
	and 
	\begin{align*}
		C \theta^{-(d+2)} \max\left\{(\theta^{-(d+2)} f(Q_{\ell, t}))^{\frac{1}{d+1}}, \theta^{-1}(\theta^{\frac{3}{2}}f(Q_{\ell,t}))^{\frac{1}{d+2}}\right\} &\leq \frac{\theta}{3},
	\end{align*}
	which is ensured by the choice\footnote{In fact, given any constant $A<\infty$ we may choose $\epsilon$ to satisfy $\epsilon \leq \frac{\delta^{d+2}}{A}$.} $\epsilon = \min\left\{ \theta^{d+\frac{1}{2}} \left(\frac{\delta}{2C}\right)^{d+2}, \theta^{d+2} \left(\frac{\theta^{d+3}}{3C}\right)^{d+1}, \theta^{-\frac{3}{2}}\left(\frac{\theta^{d+4}}{3C}\right)^{d+2} \right\}$. 
	Notice that in particular $\epsilon \leq \theta^{d+\frac{1}{2}}$, which implies that $\alpha \leq \left(\theta^{\frac{3}{2}} \epsilon \right)^{\frac{1}{d+2}} \leq \theta$, see \eqref{eq:choice-alpha}.\qedhere
\end{enumerate}
\end{proof}
\noindent This completes the proofs of Lemma~\ref{lem:onestep-1} and Lemma~\ref{lem:onestep-2}.
\end{proof}

\subsubsection{One-step improvement on the lateral boundary}\label{sec:onestep-boundary}
In this subsection we extend the previous one-step improvements to small cuboids $Q^*_{\ell,s}(a) \coloneqq Q_{\ell,s}(a)\cap Q_{L,T}$ with center $a \in \Gamma_{L,T} \cap \{x_{d+1}=0\}$ at the boundary\footnote{As before, we will write $Q^*_{\ell}(a)' \coloneqq Q_{\ell}'(a)\cap Q_{L}'$ for the projection of $Q^*_{\ell,s}(a)$ to the first $d$ coordinates, and $\Gamma^*_{\ell,s}(a)\coloneqq \partial Q^*_{\ell}(a)' \times [0,s]$ for the lateral boundary.}. This requires a few modifications compared to the previous two one-step improvements, in particular an adapted boundary layer construction and the appeal to boundary regularity for the elliptic PDE satisfied by the cumulated field. 

\begin{lemma}[One-step improvement -- Version C for fixed aspect ratio on the lateral boundary]\label{lem:onestep-3}
	Let $(m,h)$ be a minimiser of the non-convex energy \eqref{eq:energy-nonconvex} in $Q_{L,T}$ with zero-flux lateral boundary conditions,  $m^{B} \equiv 0$, $|m^{T}|\leq 1$. 
	
	For any $\theta \in (0, \frac{1}{2}]$ and $A<\infty$, there exist constants $\delta, \epsilon \in (0,1)$ with the property that
	\begin{align*}
		\epsilon \leq \frac{\delta^{d+2}}{A},
	\end{align*}
	such that the following holds: 

	\noindent For any $a \in \Gamma_{L,T}\cap \{x_{d+1}=0\}$, we have that if $\ell \leq L$ and $t \leq T$ are such that
	\begin{align*}
		f(Q^*_{\ell, t}(a))\leq \epsilon 
		\quad \text{and} \quad
		n(Q^*_{\ell, t}(a))\leq \delta,
	\end{align*}
	then there holds
	\begin{align*}
		f(Q^*_{\theta \ell, \theta^{\frac{3}{2}} t}(a)) - \theta f(Q^*_{\ell, t}(a)) &\lesssim_\theta \frac{t^{\frac{4}{3}}}{\ell^2}, 
		\\
		n(Q^*_{\theta \ell, \theta^{\frac{3}{2}} t}(a)) - \theta^{\frac{1}{2}} n(Q^*_{\ell, t}(a)) &\lesssim_\theta f(Q^*_{\ell, t}(a))^{\frac{1}{d+2}}. 
	\end{align*}
\end{lemma}

In order to provide a proof for Lemma~\ref{lem:onestep-3}, as before we start with some technical preliminaries that will be useful.

\begin{lemma}[Choice of good width]\label{lem:good2} Let  $a \in \Gamma_{L,T}\cap \{x_{d+1}=0\}$. If $h\in L^2(Q^*_{\ell}(a)')$ for some $\ell\in(0,L]$, with $h\cdot \nu=0$ on $Q^*_\ell(a)'\cap \Gamma_{L,T}$, then the set of $\rho \in  \left[\frac{\ell}{2},\ell\right]$ such that 
	\begin{align*}
		\dashint_{\partial Q^*_{\rho}(a)'}(h\cdot\nu)^2\,\mathrm{d}\mathcal{H}^{d-1} \lesssim \dashint_{Q_{\ell}(a)'} |h|^2 \,\mathrm{d}x'
	\end{align*}
	has positive Lebesgue measure.
\end{lemma}

\begin{proof} 
It is a minor modification of the proof of Lemma \ref{lem:good}.
\end{proof}

\begin{lemma}[Monotonicity formula]\label{lem:monotonicity2}
	Let $h \in L^2(Q^*_{\ell,T}(a))$ for some $\ell, T>0$, and $a \in \Gamma_{L,T}\cap \{x_{d+1}=0\}$. Then the map
	\begin{align*}
		t \mapsto \frac{t^2}{\ell^2} \dashint_{Q^*_{\ell,t}(a)} (h-\overline{h}_t)^2\,\mathrm{d}x
	\end{align*}
	is non-decreasing in $[0,T]$
\end{lemma}

\begin{proof}
It is a minor modification of the proof of Lemma \ref{lem:monotonicity}.
\end{proof}
We are now ready to give a proof for Lemma \ref{lem:onestep-3}.

\begin{proof}
Let $\theta\in\left(0,\frac{1}{2}\right]$, $a \in \Gamma_{L,T}\cap \{x_{d+1}=0\}$, and denote $\ltheta = \theta \ell$, $\ttheta = \theta^{\frac{3}{2}} t$. The first three steps of the proofs of Lemma~\ref{lem:onestep-1} and \ref{lem:onestep-2} can be applied almost verbatim. Let us therefore only point some of the adaptations. 
\begin{enumerate}[label=\textsc{\bf Step \arabic*},leftmargin=0pt,labelsep=*,itemindent=*,itemsep=10pt,topsep=10pt]
	\item \emph{(Choice of good width)}. 
		Let $\rho\in [\theta\ell, \frac{3}{4}\ell]$ be a good width in the sense of Lemma~\ref{lem:good2}. Then 
		\begin{align*}
			\dashint_{\Gamma^*_{\rho,\ttheta}(a)} (g-\overline{g}_{\ttheta})^2 \,\mathrm{d}\mathcal{H}^{d-1}\,\mathrm{d}x_{d+1} 
			&\lesssim \dashint_{Q^*_{\frac{3}{4}\ell, \ttheta}(a)} (h-\overline{h}_{\ttheta})^2 \,\mathrm{d}x 
			\lesssim \dashint_{Q^*_{\ell, \ttheta}} (h-\overline{h}_{\ttheta})^2 \,\mathrm{d}x,
		\end{align*}
		and applying the monotonicity formula from Lemma~\ref{lem:monotonicity2}, we obtain  
		\begin{align*}
			\frac{\ttheta^2}{\ell^2} \dashint_{Q^*_{\ell, \ttheta}(a)} (h-\overline{h}_{\ttheta})^2 \,\mathrm{d}x
			\leq \frac{t^2}{\ell^2} \dashint_{Q^*_{\ell, t}(a)} (h-\overline{h}_t)^2\,\mathrm{d}x 
			\leq f(Q^*_{\ell, t}(a)),
		\end{align*}
		hence
		\begin{align*}
			\frac{\ttheta^2}{\ell^2} \dashint_{\Gamma^*_{\rho,\ttheta}(a)} (g-\overline{g}_{\ttheta})^2 \,\mathrm{d}\mathcal{H}^{d-1} \, \mathrm{d}x_{d+1}
			\lesssim f(Q^*_{\ell, t}(a)).
		\end{align*}

	\item \emph{(Construction of a competitor)}. 
	As before, we start with the over-relaxed solution in $Q^*_{\rho, \ttheta}(a)$, i.e.\ $\left(\frac{x_{d+1}}{\ttheta} m^{\ttheta}, \overline{h}_{\ttheta}\right)$ and modify it using Proposition~\ref{prop:fine-properties-relaxed} in a boundary layer to a competitor $(m_{\mathrm{rel}}, h_{\mathrm{rel}})$ for the relaxed problem with boundary conditions $g = h\cdot \nu$ on $\Gamma^*_{\rho, \ttheta}(a)$, bottom magnetisation $m^B = 0$ and top magnetisation $m^{\ttheta} = -\nabla'\cdot \int_0^{\ttheta} h \,\mathrm{d}x_{d+1}$ coming from the minimiser $(m,h)$ in $Q_{L,T}$. 
	By means of Corollary~\eqref{cor:competitor}, $(m_{\mathrm{rel}}, h_{\mathrm{rel}})\in \mathcal{A}^{\mathrm{rel}}_{Q^*_{\rho, \ttheta}(a)}(g;0,m^{\ttheta})$ can now be massaged into a competitor for the non-convex problem.
	Appealing to (local) minimality of $(m,h)$ in $Q^*_{\rho,\ttheta}(a)$ then allows us to estimate 
	\begin{align*}
		f(Q^*_{\ltheta, \ttheta}(a)) \lesssim  \frac{\ttheta^{\frac{4}{3}}}{\ltheta^2} + \theta^{-(d+2)} \left( \left(\theta^{-(d+2)} f(Q^*_{\ell, t}(a))\right)^{\frac{1}{d+1}} + n(Q^*_{\rho,\ttheta}(a)) \right) f(Q^*_{\ell, t}(a)).
	\end{align*}
	
	\item \emph{(Estimate on $n$)}.
	We now estimate 
	\begin{align*}
		n(Q^*_{\rho, \ttheta}(a))
		&\leq \frac{\ttheta}{\rho} \sup_{Q^*_{\rho}(a)'} \left| \dashint_{0}^{t} h\,\mathrm{d}x_{d+1} \right| 
		+ \frac{\ttheta}{\rho} \sup_{Q^*_{\rho}(a)'} \left| \dashint_{0}^{\ttheta} h\,\mathrm{d}x_{d+1} - \dashint_{0}^{t} h\,\mathrm{d}x_{d+1} \right| \\
		&\leq \frac{\ttheta}{t} \frac{\ell}{\rho} n(Q^*_{\ell, t}(a)) + \frac{\ttheta}{\rho} \sup_{Q^*_{\rho}(a)'} \left| \dashint_{0}^{\ttheta} h\,\mathrm{d}x_{d+1} - \dashint_{0}^{t} h\,\mathrm{d}x_{d+1} \right|.
	\end{align*}
\end{enumerate}
	As in the previous one-step improvement, let us consider the function 
	\begin{align*}
		H(x') \coloneqq \int_{0}^{\ttheta} h\,\mathrm{d}x_{d+1} - \frac{\ttheta}{t} \int_{0}^{t} h\,\mathrm{d}x_{d+1} 
		= \int_0^{\ttheta} (h-\overline{h}_t) \mathrm{d}x_{d+1}, \quad x' \in Q^*_{\ell}(a)'.
	\end{align*}
	Since this time we need to control the maximum of $H$ over $Q^*_{\rho}(a)'$ up to the boundary, we extend $H$ across the lateral boundaries $\partial Q^*_{\ell}(a)' \cap \partial Q_L'$ in the following way: $\partial Q^*_{\ell}(a)' \cap \partial Q_L'$ may consist of up to $d$ interfaces. We present the argument for the case that the number of such interfaces is $k \in \{1, \dots, d\}$, and these interfaces are situated at $x_1=L, x_2=L, \dots, x_k=L$, respectively. In that case, as in the proof of the global lower bound on the energy (for zero-flux boundary conditions), we define successively extensions $H^{(j)}$ on the extended (by reflections along the axes $x_1, \dots, x_j$) domain  $Q_{\ell}^{*,(j)}(a)$ as follows:  
	\begin{align*}
		H^{(1)}(x') \coloneqq \begin{cases}
			H(x'), & \text{if } x' \in Q^*_{\ell}(a)' \\
			(-H_1, H_2, \dots, H_d)(2L-x_1, x_2, \dots, x_d), & \text{if } x' \in Q_{\ell}^{*,(1)}(a)'\setminus Q^*_{\ell}(a)',
		\end{cases}
	\end{align*}
	with $\nabla'\cdot H^{(1)}(x') = \nabla'\cdot H(x')$ for $x' \in Q^*_{\ell}(a)'$ and $\nabla'\cdot H^{(1)}(x') = (\nabla'\cdot H)(2L-x_1, x_2, \dots, x_d)$ for $x' \in Q_{\ell}^{*,(j)}(a)'\setminus Q^*_{\ell}(a)'$. Note that $H\cdot \nu' = 0$ on the interface $\{x_1=L\}$, so we can set 
	\begin{align*}
		m^{(1)}(x') \coloneqq \begin{cases}
			\frac{\ttheta}{t} m(x', t) - m(x', \ttheta), & \text{if } x' \in Q^*_{\ell}(a)', \\
			\frac{\ttheta}{t} m(2L-x_1, x_2, \dots, x_d, t) - m(2L-x_1, x_2, \dots, x_d, \ttheta), & \text{if } x' \in Q_{\ell}^{*,(1)}(a)'\setminus Q^*_{\ell}(a)',
		\end{cases}
	\end{align*}
	to obtain $\nabla'\cdot H^{(1)} = m^{(1)}$ with $|m^{(1)}| \leq 2$ on $Q_{\ell}^{*,(1)}(a)'$. Define inductively, for $j=1,\dots, k$,
	\begin{align*}
		H^{(j)}(x') \coloneqq \begin{cases}
		H^{(j-1)}(x'), & \text{if } x' \in Q_{\ell}^{*,(j-1)}(a)', \\
		(H_1, \dots, -H_{j-1}, \dots, H_d)(x_1, \dots, 2L-x_j, \dots, x_d), & \text{if } x' \in  Q_{\ell}^{*,(j)}(a)' \setminus  Q_{\ell}^{*,(j-1)}(a)',
	\end{cases}
	\end{align*}
	and, correspondingly,
	\begin{align*}
		m^{(j)}(x') \coloneqq \begin{cases}
			\frac{\ttheta}{t} m^{(j-1)}(x', t) - m^{(j-1)}(x', \ttheta), & \text{if } x' \in Q^{*, (j-1)}_{\ell}(a)', \\
			\frac{\ttheta}{t} m^{(j-1)}(x_1, \dots, 2L-x_{j}, \dots, x_d, t) & \\ \qquad - m^{(j-1)}(x_1, \dots, 2L-x_{j}, \dots, x_d, \ttheta), & \text{if } x' \in Q_{\ell}^{*,(j)}(a)'\setminus Q^{*,(j-1)}_{\ell}(a)'.
		\end{cases}
	\end{align*}
	Finally, we set $\widetilde{H} \coloneqq H^{(k)}$ and $\widetilde{m} \coloneqq m^{(k)}$ on $Q_{\ell}^{*,(k)}(a)'$. Then $\widetilde{H}$ satisfies $\nabla'\cdot \widetilde{H} = \widetilde{m}$, in particular $|\nabla'\cdot\widetilde{H}| = |\widetilde{m}| \leq 2$. 
	Note that this construction crucially uses that $H\cdot \nu' = 0$ across $\Gamma^*_{\ell,t}(a) \cap \Gamma_{L,T}$. By minimality of $h$ and the construction of the extension, $\widetilde{H}$ is a gradient. 
	
	Now for any $x_*' \in Q^*_{\rho}(a)'$, by interior elliptic regularity for the potential of $\widetilde{H}$, we have that 
	\begin{align*}
		|H(x_*')| 
		&\lesssim \dashint_{B_{\alpha \ell}'(x_*')} |\widetilde{H}|^2\,\mathrm{d}x' + (\alpha \ell)^2 \sup_{B_{\alpha \ell}'(x_*')} |\nabla'\cdot \widetilde{H}|^2 
		\lesssim \alpha^{-d} \dashint_{Q_{\ell}^{*,(k)}(a)'} |\widetilde{H}|^2\,\mathrm{d}x' + (\alpha \ell)^2,
	\end{align*}
	as long as $\alpha$ is small enough so that $B_{\alpha \ell}'(x_*')\subset Q_{\ell}^{*,(k)}(a)'$. By construction of $H$, we then have 
	\begin{align*}
		|H(x_*')| 
		\lesssim \alpha^{-d} \dashint_{Q_{\ell}(a)'} |H|^2\,\mathrm{d}x' + (\alpha \ell)^2   
		\lesssim \alpha^{-d} \ttheta t \dashint_{Q^*_{\ell,t}(a)} (h-\overline{h}_t)^2\,\mathrm{d}x + \alpha^2 \ell^2,
	\end{align*}
	and we can optimise in $\alpha$ to obtain 
	\begin{align*}
	\alpha = \left(\frac{\ttau t}{\ell^2} \dashint_{Q^*_{\ell,t}(a)} ( h - \overline{h}_t)^2\,\mathrm{d}x \right)^{\frac{1}{d+2}},
	\end{align*}
	hence
	\begin{align*}
		|H(x_*')|^2 \lesssim \ell^2 \left(\frac{\ttheta t}{\ell^2} \dashint_{Q^*_{\ell,t}(a)} ( h - \overline{h}_t)^2\,\mathrm{d}x \right)^{\frac{2}{d+2}},
	\end{align*}
	and by taking the supremum over all $x_*'\in Q^*_{\rho}(a)'$ we obtain 
	\begin{align*}
		\sup_{x_*' \in Q^*_{\rho}(a)'} |H(x_*')| \lesssim \ell^2 \left(\theta^{\frac{3}{2}} f(Q^*_{\ell,t}(a))\right)^{\frac{2}{d+2}}. 
	\end{align*}
	The rest of the proof now proceeds as the proof of Lemma~\ref{lem:onestep-2}. 
\end{proof}

\subsubsection{The iterations}\label{sec:iterations}
We are now in the position to iterate our one-step improvements. 

\begin{proposition}[Iteration -- Version A]\label{prop:iteration-A}
	There exist universal constants $\epsilon\in (0,1)$ and $c_{LT}\geq 1$ (depending only on the dimension $d$) such that the following holds: if $f_0(Q_{L,T})\leq \epsilon$ and $\ell\leq \frac{L}{4}$ is such that $\ell \geq c_{LT} T^{\frac{2}{3}}$, then 
	\begin{align}\label{eq:bound-iteration}
		f_0(Q_{\ell, T}(a)) \lesssim f_0(Q_{L,T}) + \frac{T^{\frac{4}{3}}}{\ell^2}
	\end{align}
	for any $a\in \RR^d \times \{0\}$ such that $Q_{\ell, T}(a) \subseteq Q_{\frac{3}{4}L,T}$. 
\end{proposition}

\begin{remark}
	The choice $\ell \leq \frac{L}{4}$ and $a$ such that $Q_{\ell, T}(a) \subseteq Q_{\frac{3}{4}L,T}$, in particular $|a|\leq \frac{L}{2}$, is arbitrary at this point. Since the iteration is based on the \emph{interior} one-step improvement Lemma~\ref{lem:onestep-1}, we just have to make sure to have some positive distance to the boundary $\Gamma_{L,T}$.
\end{remark}

\begin{proof}W.l.o.g.\ let $a=0$ and fix $\theta\in(0,\frac{1}{4}]$. The general case can be obtained by applying Lemma~\ref{lem:onestep-1} once with $a$ such that $Q_{\ell, T}(a) \subseteq Q_{\frac{3}{4}L,T}$ and then re-applying it for the fixed centre $a$.
\begin{enumerate}[label=\textsc{\bf Step \arabic*},leftmargin=0pt,labelsep=*,itemindent=*,itemsep=10pt,topsep=10pt]	
	\item \label{step:ell-N} \emph{(Proof for cuboids with geometrically related side lengths)}. We first prove the bound \eqref{eq:bound-iteration} for $\ell = \ell_N\coloneqq\theta^N L$ for some $N\in\NN$ and then extend the result to arbitrary $\ell$ such that $c_{LT} T^{\frac{2}{3}} \leq \ell \leq L$. More precisely, we prove inductively that there exists a universal constant $C_0<\infty$ such that for $k=0,1, \ldots, N$ there holds\footnote{With the convention that a summation over the empty set is equal to zero.} 
	\begin{align}
		f_0(Q_{\ell_k, T}) &\leq \theta^k f_0(Q_{L,T}) + C_0 \frac{T^{\frac{4}{3}}}{L^2} \theta^{k-1} \sum_{0\leq j\leq k-1} \theta^{-3j} \label{eq:induction-f0}.
	\end{align}
	Note that the inequality holds trivially for $k=0$, so assume that \eqref{eq:induction-f0} holds true for all $1\leq k \leq K$ with $K\leq N-1$. 
	By the induction hypothesis for $k=K$, we have that
	\begin{align}
		f_0(Q_{\ell_K, T}) 
		&\leq \theta^K f_0(Q_{L,T}) + C_0 \frac{T^{\frac{4}{3}}}{L^2} \theta^{K-1} \sum_{0\leq j\leq K-1} \theta^{-3j} \label{eq:induction-f0-K-1}\\
		&= \theta^K f_0(Q_{L,T}) + C_0 \frac{T^{\frac{4}{3}}}{L^2} \theta^{-2(K-1)} \frac{1-\theta^{3K}}{1-\theta^3}  \nonumber\\
		&= \theta^K f_0(Q_{L,T}) + C_0 \frac{T^{\frac{4}{3}}}{\ell^2} \theta^{2(N-K+1)} \frac{1-\theta^{3K}}{1-\theta^3}.\nonumber 
	\end{align}
	Note that $\theta^{2(N-K+1)} \leq 1$ since $K \leq N-1$, hence, provided that $c_{LT}^2 \geq \frac{2 C_0 (1-\theta^3)}{\epsilon}$, we may bound 
	\begin{align*}
		f_0(Q_{\ell_K, T}) 
		&\leq \theta^K f_0(Q_{L,T}) + C_0 \frac{T^{\frac{4}{3}}}{\ell^2} \frac{1}{1-\theta^3}
		\leq \theta^K f_0(Q_{L,T}) + \frac{C_0}{c_{LT}^2} \frac{1}{1-\theta^3}
		\leq 2^{-K} \epsilon + \frac{\epsilon}{2}
		\leq \epsilon.
	\end{align*}
	We can therefore apply Lemma~\ref{lem:onestep-1} to infer that
	\begin{align*}
		f_0(Q_{\ell_{K+1},T}) 
		\leq \theta f_0(Q_{\ell_K, T}) + C_{\theta} \frac{T^{\frac{4}{3}}}{\ell_K^2} 
		= \theta f_0(Q_{\ell_K, T}) + \frac{C_{\theta}}{\theta^{2K}} \frac{T^{\frac{4}{3}}}{L^2}, 
	\end{align*}
	with $C_{\theta}$ the implicit constant in \eqref{eq:osi-f}. Using \eqref{eq:induction-f0-K-1}, we then obtain
	\begin{align*}
		f_0(Q_{\ell_{K+1}, T}) 
		&\leq \theta^{K+1} f_0(Q_{L,T}) + C_0 \frac{T^{\frac{4}{3}}}{L^2} \theta^K \sum_{0\leq j\leq K-1} \theta^{-3j} + C_{\theta} \theta^{-2K} \frac{T^{\frac{4}{3}}}{L^2}  \\
		&\leq \theta^{K+1} f_0(Q_{L,T}) + C_0 \frac{T^{\frac{4}{3}}}{L^2} \theta^K \sum_{0\leq j\leq K} \theta^{-3j},
	\end{align*}
	where the last inequality holds provided that the constant $C_0$ is chosen such that $C_0\geq C_{\theta}$.
	
	Finally, by taking $k=N$ in \eqref{eq:induction-f0}, we obtain
	\begin{align*}
		f_0(Q_{\ell, T}) 
		&\leq \theta^N f_0(Q_{L,T}) + C_0 \frac{T^{\frac{4}{3}}}{L^2} \theta^{N-1} \sum_{j=0}^{N-1} \theta^{-3j}
		= \theta^N f_0(Q_{L,T}) + C_0 \frac{T^{\frac{4}{3}}}{L^2} \theta^{-2(N-1)} \frac{1-\theta^{3N}}{1-\theta^3}\\
		&= \theta^N f_0(Q_{L,T}) + C_0 \frac{T^{\frac{4}{3}}}{\ell^2} \theta^{2} \frac{1-\theta^{3N}}{1-\theta^3}
		\lesssim f_0(Q_{L,T})+\frac{T^\frac43}{\ell^2}.
	\end{align*}

	\item \emph{(General $\ell\leq \frac{L}{4}$)}. If $\ell$ is not of the form in \ref{step:ell-N}, then there exists an $N\in\NN$ such that $\ell \in (\ell_{N+1}, \ell_N)$.
	But then we can bound 
	\begin{align*}
		f_0(Q_{\ell, T}) 
		&\leq \left(\frac{\ell_N}{\ell_{N+1}}\right)^{d+2} \frac{T^2}{\ell_N^2} \left(\dashint_{Q_{\ell_N, T}} |\nabla' m| +  \dashint_{Q_{\ell_N, T}} \frac{1}{2} h^2\,\mathrm{d}x \right)
		= \theta^{-(d+2)} f_0(Q_{\ell_N,T}),
	\end{align*}
	hence \eqref{eq:bound-iteration} follows from \ref{step:ell-N}. \qedhere	
\end{enumerate}	
\end{proof}

\begin{proposition}[Iteration -- Version B]\label{prop:iteration-B}
	There exist universal constants $\epsilon, \delta \in (0,1)$ and $c_{\ell t}\geq 1$ (depending only on the dimension $d$) such that the following holds: if $f(Q_{L,T})\leq \epsilon$, $n(Q_{L,T})\leq \delta$, and $\ell\leq \frac{L}{8}$, $t\leq T$ are such that $\ell t^{-\frac{2}{3}} = L T^{-\frac{2}{3}} \geq c_{\ell t}$, then 
	\begin{align}\label{eq:bounds-iteration-B}
		f(Q_{\ell, t}(a)) \lesssim f(Q_{L,T}) + \frac{T^{\frac{4}{3}}}{L^2}, \quad \text{and} \quad 
		n(Q_{\ell, t}(a)) \lesssim n(Q_{L,T}) + \left( f(Q_{L,T}) + \frac{T^{\frac{4}{3}}}{L^2} \right)^{\frac{1}{d+2}}
	\end{align}
	for any $a \in \RR^d \times \{0\}$ such that $Q_{\ell}'(a) \subseteq Q_{\frac{3}{4}L}'$.
\end{proposition}

\begin{remark}
	As in the previous iteration, the choice $\ell \leq \frac{L}{8}$ and $a$ such that $Q_{\ell}'(a) \subseteq Q_{\frac{3}{4}L}'$ is arbitrary, as long as we make sure to stay away from the boundary $\Gamma_{L,T}$.
\end{remark}

\begin{proof} W.l.o.g.\ let $a=0$ and fix $\theta\in (0,\frac{1}{8}]$.
\begin{enumerate}[label=\textsc{\bf Step \arabic*},leftmargin=0pt,labelsep=*,itemindent=*,itemsep=10pt,topsep=10pt]	
	\item \label{step:iteration-geometric-l-t} \emph{(Proof for cuboids with geometrically related side lengths)}.  
	We let $\ell_k\coloneqq \theta^k L$, $t_k \coloneqq (\theta^{\frac{3}{2}})^k T$ for $k\in\NN_0$. We assume that $\ell = \ell_N$ and hence also $t = t_N$ for some $N\in\NN$, and extend the statement to arbitrary $\ell, t$ with fixed aspect ratio in \ref{step:iteration-general-l-t}. 
	Note that 
	\begin{align}\label{eq:L-T-rel}
		\frac{t_k^{\frac{4}{3}}}{\ell_k^2} = \frac{T^{\frac{4}{3}}}{L^2} \quad \text{for all} \quad k\in \NN. 
	\end{align}	
	In view of Lemma~\ref{lem:onestep-2} and Remark~\ref{rem:onestep-2}, we will prove inductively that, for $k=0,1, \ldots, N$, there holds\footnote{With the convention that a summation over the empty set is equal to zero.} 
	\begin{align}
		f(Q_{\ell_k, t_k}) &\leq \theta^k f(Q_{L,T}) + C \frac{T^{\frac{4}{3}}}{L^2} \sum_{0\leq j\leq k-1} \theta^j  \quad \text{and} \label{eq:induction-f}\\
		n(Q_{\ell_k, t_k}) &\leq \theta^{\frac{k}{2}} n(Q_{L,T}) + \widetilde{C} \sum_{0\leq j\leq k-1} \theta^{\frac{j}{2}} \left( \theta^{k-j-1} f(Q_{L,T}) + \frac{T^{\frac{4}{3}}}{L^2} \sum_{0\leq i\leq k-j-2} \theta^{i}\right)^{\frac{1}{d+2}}. \label{eq:induction-n}
	\end{align}

	Observe that the inequalities trivially hold for $k=0$. Assume that \eqref{eq:induction-f} and \eqref{eq:induction-n} hold true for $0\leq k \leq K$ with $K\leq N-1$.
	Let us first prove that \eqref{eq:induction-f} holds for $k=K+1$. By the assumption on level $K$, we have that 
	\begin{align}\label{eq:inductionK}
		f(Q_{\ell_K, t_K}) \leq \theta^K f(Q_{L,T}) + C \frac{T^{\frac{4}{3}}}{L^2} \sum_{j=0}^{K-1} \theta^j 
		\leq \theta^K f(Q_{L,T}) + \frac{C}{1-\theta} \frac{T^{\frac{4}{3}}}{L^2} \leq \epsilon,
	\end{align}
	provided that the constant in $L\gg T^{\frac{2}{3}}$ is chosen large enough. We can therefore apply Lemma~\ref{lem:onestep-2} to infer
	\begin{align*}
		f(Q_{\ell_{K+1},t_{K+1}}) \leq \theta f(Q_{\ell_K, t_K}) + C_{\theta,f} \frac{t_K^{\frac{4}{3}}}{\ell_K^2} 
	\end{align*}
	with $C_{\theta,f}$ the implicit constant in \eqref{eq:osi-f}. Then, by inserting the first inequality in \eqref{eq:inductionK} and using \eqref{eq:L-T-rel}, we obtain
	\begin{align*}
		f(Q_{\ell_{K+1}, t_{K+1}}) \leq \theta^{K+1} f(Q_{L,T}) + C \frac{T^{\frac{4}{3}}}{L^2} \sum_{j=1}^K \theta^j + C_{\theta,f} \frac{T^{\frac{4}{3}}}{L^2} 
		\leq \theta^{K+1} f(Q_{L,T}) + C \frac{T^{\frac{4}{3}}}{L^2} \sum_{j=0}^K \theta^j,
	\end{align*}
	where the last inequality holds provided that the constant $C$ is chosen such that $C\geq C_{\theta,f}$. 
	
	Let us now prove that \eqref{eq:induction-n} holds for $k=K+1$. As before, we can estimate
	\begin{align*}
		n(Q_{\ell_K, t_K}) &\leq \theta^{\frac{K}{2}} n(Q_{L,T}) + \widetilde{C} \sum_{j=0}^{K-1} \theta^{\frac{j}{2}} \left( \theta^{K-j-1} f(Q_{L,T}) + \frac{T^{\frac{4}{3}}}{L^2} \sum_{0\leq i\leq K-j-2} \theta^{i}\right)^{\frac{1}{d+2}} \\
		&\leq  \theta^{\frac{K}{2}} n(Q_{L,T}) + \widetilde{C}  \frac{\theta^{\frac{K-1}{d+2}}}{1-\theta^{\frac{d}{2(d+2)}}} f(Q_{L,T})^{\frac{1}{d+2}} + \frac{\widetilde{C}}{1-\theta^{\frac{1}{2}}} \left( \frac{T^{\frac{4}{3}}}{L^2} \frac{1}{1-\theta} \right)^{\frac{1}{d+2}},
	\end{align*}
	hence $n(Q_{\ell_K, t_K}) \leq \delta$ provided $\epsilon$ is chosen small enough (depending on $\theta$ and $\delta$ fixed) and the constant in $L\gg T^{\frac{2}{3}}$ is chosen large enough. We can therefore apply Lemma~\ref{lem:onestep-2} to infer
	\begin{align*}
		n(Q_{\ell_{K+1}, t_{K+1}}) \leq \theta^{\frac{1}{2}} n(Q_{\ell_K, t_K}) + C_{\theta,n} f(\ell_K, t_K)^{\frac{1}{d+2}}
	\end{align*}
	with $C_{\theta,n}$ the implicit constant in \eqref{eq:osi-n}.
	Combining with the induction hypotheses \eqref{eq:induction-n} and \eqref{eq:induction-f}, we are led to
	\begin{align*}
		n(Q_{\ell_{K+1}, t_{K+1}}) &\leq \theta^{\frac{1}{2}} \left(\theta^{\frac{K}{2}} n(Q_{L,T}) + \widetilde{C} \sum_{j=0}^{K-1} \theta^{\frac{j}{2}} \left( \theta^{K-j-1}f(Q_{L,T}) + \frac{T^{\frac{4}{3}}}{L^2} \sum_{0\leq i\leq K-j-2} \theta^i \right)^{\frac{1}{d+2}} \right) \\
		&\quad + C_{\theta,n} \left( \theta^{K} f(Q_{L,T}) + C \frac{T^{\frac{4}{3}}}{L^2} \sum_{i=0}^{K-1} \theta^i \right)^{\frac{1}{d+2}} \\
		&\leq \theta^{\frac{K+1}{2}} n(Q_{L,T}) + \widetilde{C} \sum_{j=1}^K \theta^{\frac{j}{2}} \left( \theta^{K-j} f(Q_{L,T}) + \frac{T^{\frac{4}{3}}}{L^2} \sum_{0\leq i\leq K-j-1} \theta^i \right)^{\frac{1}{d+2}} \\
		&\quad + C_{\theta,n} C \left( \theta^K f(Q_{L,T}) +  \frac{T^{\frac{4}{3}}}{L^2} \sum_{i=0}^{K-1} \theta^i \right)^{\frac{1}{d+2}} \\
		&\leq \theta^{\frac{K+1}{2}} n(Q_{L,T}) + \widetilde{C} \sum_{j=0}^{K} \theta^{\frac{j}{2}} \left( \theta^{K-j} f(Q_{L,T}) + \frac{T^{\frac{4}{3}}}{L^2} \sum_{0\leq i\leq K-j-1} \theta^i \right)^{\frac{1}{d+2}},
	\end{align*}
	provided we choose $\widetilde{C} \geq C_{\theta,n} C$. The result is thus proved.
	
	\item \label{step:iteration-general-l-t} \emph{(General $\ell, t$ with fixed aspect ratio)}. If $\ell,t$ are not of the form in \ref{step:iteration-geometric-l-t}, then there exists an $N\in\NN$ such that $\ell \in (\ell_{N+1}, \ell_N)$ and $t \in (t_{N+1},t_N)$.
	By Lemma~\ref{lem:monotonicity} we can estimate 
	\begin{align*}
		f(Q_{\ell, t}) &= \frac{t^2}{\ell^2} \dashint_{Q_{\ell, t}} |\nabla' m| + \frac{t^2}{\ell^2} \dashint_{Q_{\ell, t}} \frac{1}{2} |h-\overline{h}_t|^2\,\mathrm{d}x 
		\leq \frac{t_N^2}{\ell^2} \dashint_{Q_{\ell, t_N}} |\nabla' m| + \frac{t_N^2}{\ell^2} \dashint_{Q_{\ell, t_N}} \frac{1}{2} |h-\overline{h}_{t_N}|^2\,\mathrm{d}x \\
		&\leq \left(\frac{\ell_N}{\ell_{N+1}}\right)^{d+2} \frac{t_N^2}{\ell_N^2} \left(\dashint_{Q_{\ell_N, t_N}} |\nabla' m| +  \dashint_{Q_{\ell_N, t_N}} \frac{1}{2} |h-\overline{h}_{t_N}|^2\,\mathrm{d}x \right)
		= \theta^{-(d+2)} f(Q_{\ell_N,t_N}).
	\end{align*}
	Similarly, following the line of argument in \ref{step:one-step-improvement5B} of the proof of the one-step improvement, we may bound
	\begin{align*}
		n(Q_{\ell, t}) \lesssim \theta^{-1} \left(n(Q_{\ell_N, t_N}) + f(Q_{\ell_N, t_N})^{\frac{1}{d+2}}\right),
	\end{align*}
	hence \eqref{eq:bounds-iteration-B} follows from \ref{step:iteration-geometric-l-t} . \qedhere
\end{enumerate}
\end{proof}

\begin{proposition}[Iteration -- Version C at the boundary]\label{prop:iteration-C}
	There exist universal constants $\epsilon, \delta \in (0,1)$ and $c_{\ell t}\geq 1$ (depending only on the dimension $d$) such that the following holds: if $f(Q^*_{L,T}(a))\leq \epsilon$, $n(Q^*_{L,T}(a))\leq \delta$, for $a \in \Gamma_{L,T} \cap \{x_{d+1}=0\}$, and $\ell\leq L$, $t\leq T$ are such that $\ell t^{-\frac{2}{3}} = L T^{-\frac{2}{3}} \geq c_{\ell t}$, then 
	\begin{align*}
		f(Q^*_{\ell, t}(a)) \lesssim f(Q^*_{L,T}(a)) + \frac{T^{\frac{4}{3}}}{L^2}, \quad \text{and} \quad 
		n(Q^*_{\ell, t}(a)) \lesssim n(Q^*_{L,T}(a)) + \left( f(Q^*_{L,T}(a)) + \frac{T^{\frac{4}{3}}}{L^2} \right)^{\frac{1}{d+2}}.
	\end{align*}
\end{proposition}

\begin{proof}
	The proof of Proposition~\ref{prop:iteration-C} is a minor modification of the proof of Proposition~\ref{prop:iteration-B}, using the boundary one-step improvement (Lemma~\ref{lem:onestep-3}) instead of the interior one-step improvement (Lemma~\ref{lem:onestep-2}). 
\end{proof}

\addtocontents{toc}{\protect\setcounter{tocdepth}{-1}}
\section*{Acknowledgements}
The authors would like to thank Felix Otto for bringing their attention to the topic and for helpful discussions.
We thank the referee for helpful remarks and suggestions.

This work was partially funded by ANID FONDECYT 11190130 and ANID FONDECYT 1231593.   

TR thanks the Pontifical Catholic University of Chile and CR thanks the Max Planck Institute for Mathematics in the Sciences for their support and warm hospitality. 

TR would like to thank the Isaac Newton Institute for Mathematical Sciences for support and hospitality during the program ``Optimal design of complex materials'' funded by EPSRC grant no EP/R014604/1 and the Hausdorff Research Institute for Mathematics for support and hospitality during the Trimester Program ``Mathematics for Complex Materials'' funded by DFG EXC-2047/1 – 390685813, where work on this paper was undertaken.
\addtocontents{toc}{\protect\setcounter{tocdepth}{2}}
\appendix
\section{A useful elliptic estimate}\label{sec:appendix}

\begin{lemma}\label{lem:elliptic-line}
	Let $Q' = (0,1)^d$, $f\in L^2(Q')$ and $H=\{a_1\} \times \bigtimes_{j=2}^{d} (a_j, b_j) \subset Q'$. Let $u$ be the unique solution of the elliptic Neumann problem
	\begin{equation}\label{eq:elliptic-hyperplane}
		\begin{array}{rcll}
		-\Delta' u &=& f - \left( \frac{1}{|H|}\int_{Q'} f\,\mathrm{d}x' \right) \mathcal{H}^{d-1}\lfloor_{H} & \mathrm{in}\; Q', \\
		-\nabla'u \cdot \nu' &=& 0 &\mathrm{on}\;\partial Q', 
		\end{array}
	\end{equation}
	with zero average $\int_{Q'} u\,\mathrm{d}x' = 0$. Then 
	\begin{align}\label{eq:elliptic-hyperplane-estimate}
		\int_{Q'} |\nabla'u|^2\,\mathrm{d}x' \lesssim \frac{1}{|H|} \int_{Q'} f^2\,\mathrm{d}x'.
	\end{align}
\end{lemma}
While this result is rather standard, we give a proof for the convenience of the reader. 
\begin{proof}
	Note that the right-hand side of the PDE in \eqref{eq:elliptic-hyperplane} lies in the negative-order Sobolev space $H^{-1}(Q') \simeq H^1_0(Q')^*$. Indeed, let $\varphi \in \mathcal{C}_c^1(Q')$, then 
	\begin{align*}
		\int_{Q'} \varphi \,\mathrm{d}\mathcal{H}^{d-1}\lfloor_{H} = \lim_{\epsilon \downarrow 0} \frac{1}{\epsilon} \int_{Q'} \varphi \chi_{\epsilon}\,\mathrm{d}x',
	\end{align*}
	where $\chi_{\epsilon}(x')\coloneqq	\1_{(a_1, a_1+\epsilon)}(x_1) \prod_{j=2}^d \1_{(a_j, b_j)}(x_j)$.\footnote{W.l.o.g.\ we assume that $a_1<1$, for $a_1 = 1$ use $\chi_{\epsilon}(x')\coloneqq	\1_{(a_1-\epsilon, a_1)}(x_1) \prod_{j=2}^d \1_{(a_j, b_j)}(x_j)$.} 
	Define $\xi_{\epsilon}(x_1) \coloneqq \1_{(a_1, a_1+\epsilon)}(x_1) (x_1-a_1) + \1_{(a_1+\epsilon,1)}(x_1) \epsilon$, then $\1_{(a_1, a_1+\epsilon)}= \xi_{\epsilon}'$ and integration by parts yields
	\begin{align*}
		\frac{1}{\epsilon} \int_{Q'} \varphi \chi_{\epsilon}\,\mathrm{d}x'
		&= -\frac{1}{\epsilon} \int_{Q'} \partial_1 \varphi(x') \xi_{\epsilon}(x_1) \prod_{j=2}^d \1_{(a_j, b_j)}(x_j)\,\mathrm{d}x'.
	\end{align*}
	Hence, by the Cauchy-Schwarz inequality,
	\begin{align*}
		\left| \frac{1}{\epsilon} \int_{Q'} \varphi \chi_{\epsilon}\,\mathrm{d}x' \right|^2 
		&\leq \left( \frac{1}{\epsilon^2} \int_{Q'} \xi_{\epsilon}^2(x_1) \prod_{j=2}^d \1_{(a_j, b_j)}(x_j)\,\mathrm{d}x'  \right) \int_{Q'} |\partial_1 \varphi|^2\,\mathrm{d}x' \\
		&\leq \left( \prod_{j=2}^d (b_j-a_j) \frac{1}{\epsilon^2} \int_0^1 \xi_{\epsilon}(x_1)^2\,\mathrm{d}x_1\right) \int_{Q'} |\nabla' \varphi|^2\,\mathrm{d}x' \\
		&= |H| \left(1-a_1-\frac{2\epsilon}{3}\right) \int_{Q'} |\nabla' \varphi|^2\,\mathrm{d}x' 
		\leq |H| \int_{Q'} |\nabla' \varphi|^2\,\mathrm{d}x'.
	\end{align*}
	It follows that $\left|	\int_{Q'} \varphi \,\mathrm{d}\mathcal{H}^{d-1}\lfloor_{H}\right| \leq |H|^{\frac{1}{2}} \|\nabla'\varphi\|_{L^2(Q')}$ for all $\varphi \in C_c^1(Q')$. 
	
	We may now estimate 
	\begin{align*}
		\left| \int_{Q'} \nabla'u\cdot\nabla'\varphi \,\mathrm{d}x' \right| 
		&\stackrel{\eqref{eq:elliptic-hyperplane}}{=} \left| \int_{Q'} f \varphi \,\mathrm{d}x' -  \left( \frac{1}{|H|}\int_{Q'} f\,\mathrm{d}x' \right) \int_{Q'} \varphi \,\mathrm{d}\mathcal{H}^{d-1}\lfloor_{H} \right| \\
		&\leq \|f\|_{L^2(Q')} \|\varphi\|_{L^2(Q')} +  \left( \frac{1}{|H|}\int_{Q'} f\,\mathrm{d}x' \right) |H|^{\frac{1}{2}} \|\nabla' \varphi\|_{L^2(Q')} \\
		&\stackrel{\text{Poincaré}}{\lesssim} |H|^{-\frac{1}{2}} \|f\|_{L^2(Q')} \|\nabla'\varphi\|_{L^2(Q')}.
	\end{align*}
	Taking the supremum over all $\varphi \in C^1_c(Q')$ yields \eqref{eq:elliptic-hyperplane-estimate}.
\end{proof}

\bibliography{ReferencesMicromagnetics}
\bigskip
\end{document}